%% file: Paper2Draft19.tex
\documentclass[dvips]{article}%
\usepackage[english]{babel}
\usepackage{graphicx}
\usepackage{a4wide}
\usepackage{amssymb}
\usepackage{array}
\usepackage{subfigure}
\usepackage{amsfonts}
\usepackage{amsmath}
\usepackage{color}
\providecommand{\U}[1]{\protect\rule{.1in}{.1in}}
\newtheorem{theorem}{Theorem}[section]

\newtheorem{lemma}[theorem]{Lemma}

\newtheorem{proposition}[theorem]{Proposition}
\newtheorem{remark}[theorem]{Remark}

\newenvironment{proof}[1][Proof]{\noindent\textbf{#1.}}{\ \rule{0.5em}{0.5em}}

\catcode`\@=11

\@addtoreset{equation}{section}

\begin{document}

\title{Existence of solutions describing accumulation in a thin-film flow}
\author{C.~M. Cuesta\thanks{University of the Basque Country (UPV/EHU), 
Departamento de Matem\'aticas, Aptdo. 644, 48080 Bilbao,Spain.}, 
J.~J.~L. Vel\'{a}zquez\thanks{Institut f\"{u}r Angewandte Mathematik,
Universit\"{a}t Bonn. Endenicher Allee 60, 53115 Bonn, Germany.} }
\date{}
\maketitle

\begin{abstract}
We consider a third order non-autonomous ODE that arises as a model of fluid accumulation 
in a two dimensional thin-film flow driven by surface tension and gravity. With the appropriate matching
 conditions, the equation describes the inner structure of solutions 
around a stagnation point. In this paper we prove the existence of solutions that satisfy this problem.
In order to prove the result we first transform the equation into a four dimensional dynamical system.
In this setting the problem consists of finding heteroclinic connections that are the intersection 
of a two dimensional centre-stable manifold and a three-dimensional centre-unstable one. We then use 
a shooting argument that takes advantage of the information of the flow in the far-field, part of the analysis also requires 
the understanding of oscillatory solutions with large amplitude. The far-field is represented
 by invariant three-dimensional subspaces and the flow on them needs to be understood, most of the necessary results in 
this regard are obtained in \cite{CV}. This analysis focuses on the understanding of oscillatory solutions 
and some results are used in the current proof, although the structure of oscillations is somewhat more complicated.
\end{abstract}
\tableofcontents 
\section{Introduction}\label{sec:intro}
In this paper is to prove the existence of solutions of 
\begin{equation}\label{S2E3}%
\left(  \frac{d^{3}H}{d\xi^{3}} + \xi^{2} + a\right)  H^{3}=1\,,
\quad a\in\mathbb{R}\,. 
\end{equation}
that satisfy the following behaviour
\begin{equation}\label{S2E7}%
H\sim\frac{1}{|\xi|^{\frac{2}{3}}}\quad\mbox{as}\quad|\xi| \to\infty\,. 
\end{equation}
This equation has been deduced in \cite{CV3} (see also \cite{CV4}). 
It arises in a two dimensional model describing steady coating of a bumpy
surface by a thin-film approximation. In particular (\ref{S2E3}) results in the particular case 
that the motion of the fluid is driven by a balance of capillarity
and gravity effects. 

In some regions the curvature of the substrate induces
capillary forces of the same order of magnitude than the gravitational ones.
In the steady regime the model describing such flows has the form (cf.
(\cite{CV3})):
\begin{equation}\label{A1}%
\frac{\partial}{\partial s}
\left(\left(Q(s)+\varepsilon\frac{\partial^{3}h}{\partial s^{3}}\right)h^{3}\right)=0\,,
\quad\varepsilon > 0
\end{equation}
where we have neglected some non-relevant terms. The variable $s$ stands for
the arc-length that parametrises the substrate, and $h$ is the height of the
fluid over this surface. The parameter $\varepsilon$ is the ratio of the
characteristic height of the fluid and the characteristic radius of curvature of the
substrate. The function $Q(s)$ describes the balance between gravitational and
the capillary forces induced by the geometry of the substrate, it measures the
tendency of the fluid to move in the tangential direction to the substrate as
a result of the afore mentioned forces.

If the function $Q(s)$ has a constant sign, the motion of the fluid takes
place always in the same direction. In such case, (\ref{A1}) can be
approximated by the leading order term of (\ref{A1}):%
\begin{equation}\label{A2}%
\frac{\partial}{\partial s}(Q(s) h^{3}) =0\,. 
\end{equation}
However, this approximation breaks down and it cannot be uniformly valid for
arbitrary values of $s$ if $Q(s)$ changes sign. In such cases (\ref{A2})
predicts the onset of regions, where $Q(s)$ is close to zero, with infinite
height $h$, i.e. the fluid accumulates in those regions. As a consequence, the
approximation (\ref{A2}) must be replaced by the model (\ref{A1}). In
the particular case in which in most of the substrate $Q(s)$ is positive, but
there exists a sufficiently small region (of size $\varepsilon^{\frac{3}{17}}$
to be precise), where $Q(s)=0$, a boundary layer analysis shows that, under
suitable non-degeneracy conditions, the height of the fluid can be approximated
by means of (\ref{S2E3}), the height of the fluid becoming of order
$\varepsilon^{-\frac{2}{17}}$. This asymptotic analysis shows also that the 
solutions of (\ref{S2E3}) describing the stationary flows in those regions
must satisfy (\ref{S2E7}).

Equations similar to (\ref{A1}) where the main driving terms are the gravity 
and the curvature of the substrate, have been obtained, in a slightly different context, 
in \cite{MyersSolidII} and \cite{MyersSolidI}. This model can be obtained also as a 
particular case of the ones considered in \cite{Royetal} for specific choices 
of the parameters. See also \cite{Howell} for a model that neglects gravity. 
Similar problems have been investigated in relation with industrial applications, 
such as the drainage of (metal) foams (e.g. \cite{StockerHosoi}), manufacture of lenses (e.g. \cite{Howell} and  \cite{JChiniK}, 
although in the later case the effect of gravity can be neglected. These works offer numerical as well as formal 
(using perturbation methods) results. 

It is our aim to study the solutions of (\ref{S2E3})-(\ref{S2E7}) rigorously. 
The main result of the paper is the following:
\begin{theorem}\label{Main} 
For any $a\in\mathbb{R}$ there exists a solution of (\ref{S2E3})
satisfying (\ref{S2E7}).
\end{theorem}

We sketch the main ideas in the proof of Theorem~\ref{Main}. We first
observe that the terms $\xi^{2}H^{3}$ and $1$ in (\ref{S2E3}) give the natural
scaling $H\sim|\xi|^{-\frac{2}{3}}\Phi$ with $\Phi\sim1$ (cf. (\ref{S2E7})),
which gives the leading order behaviour of (\ref{S2E3}), namely,
\[
\left(  |\xi|^{-\frac{8}{3}} \frac{d^{3}\Phi}{d\xi^{3}}+1+a\,|\xi
|^{-2}\right)  \Phi^{3}=1\,, \quad|\xi|\gg1\,.
\]
A change of variables with behaviour 
$\tau\sim\frac{9}{17}\xi|\xi|^{\frac{8}{9}}$ as $|\xi|\to\infty$ 
gives the dominant balance problem
\[
\left(  \frac{d^{3}\Phi}{d\tau^{3}}+1\right)  \Phi^{3}\sim1 \quad
\mbox{as}\quad|\xi|\to+\infty\,,\quad \Phi\to 1\mbox{as}\quad |\tau|\to\infty\,.
\]
In such set of variables $\Phi$ and $\tau$ (\ref{S2E3}) becomes an autonomous dynamical system of the form
\begin{eqnarray*}
\frac{d^3\Phi}{d\tau^3}+1  =  \frac{1}{\Phi^{3}}-(a-1)(\cos\theta)^{2}-F(\theta)\,,\quad
\frac{d\theta}{d\tau}  = (\cos\theta)^{\frac{26}{9}} \,, 
\end{eqnarray*}
that we shall denote by $(D)$, for the unknown 
$(\Phi,d\Phi/d\tau,d^{2}\Phi/d\tau^{2},\theta) \in\mathbb{R}^{+}\times\mathbb{R}^{2}\times[-\pi/2,\pi/2]$.
Here $\theta$ is defined by $\xi=\tan\theta$ and the function $F$ is a linear combination of 
$\Phi$ and its derivatives with coefficients that depend only on $\theta$ and that vanish at 
$\theta=\pm\pi/2$. Thus, this system has the property that the three dimensional subspaces 
$\{\theta=\pm\pi/2 \}$ are invariant and the flow on them is described by the ODE 
\begin{equation} \label{ODEtapas}%
\frac{d^{3}\Phi}{d\tau^{3}}+1=\frac{1}{\Phi^{3}}\,. 
\end{equation}
The system associated to (\ref{ODEtapas}) for the unknown
$(\Phi,d\Phi/d\tau,d^{2}\Phi/d\tau^{2})$ was studied in \cite{CV}. 
It has one single critical point, $P_{s}=(1,0,0)$ and therefore $(D)$ has two critical points,
 $p_{-}=(1,0,0,-\pi/2)$ and $p_{+}=(1,0,0,\pi/2)$. Then, the solutions of (\ref{S2E3}) satisfying 
the matching conditions (\ref{S2E7}) correspond to solutions of $(D)$ contained in the trajectories 
that connect the critical point $p_{-}$ as $\tau\to-\infty$ to $p_{+}$ as $\tau\to\infty$. 
Or equivalently, they are contained in heteroclinic orbits connecting these two critical points.

The existence of a heteroclinic orbit for $(D)$ is proved by means of a shooting argument 
in the direction of decreasing $\tau$. The shooting starts close to the invariant manifold 
$\{\theta=\pi/2\}$ and the final argument will require information of the flow on the 
invariant manifold $\{\theta=-\pi/2\}$. For that reason we shall need the following 
information on (\ref{ODEtapas}). First, that the critical point $P_{s}$ is hyperbolic and has
a one-dimensional stable manifold and a two-dimensional stable manifold.
Secondly, we proved in \cite{CV} that the only possible asymptotic behaviour
of solutions on the stable manifold correspond to either
\begin{equation}\label{B1}
\lim_{\tau\to-\infty}\Phi(\tau)=\infty
\end{equation}
or to
\begin{equation}\label{B2}
\lim_{\tau\to(\tau_{*})^{+}}\Phi(\tau)= 0\quad\mbox{with} \quad\tau_{*}>-\infty\,.
\end{equation}
We shall also recall later that (\ref{ODEtapas}) has a increasing Lyapunov 
function, and that this in particular guarantees the non-existence of periodic
orbits.

To start the shooting we first prove that there exists an invariant 
two-dimensional centre-stable manifold $\mathcal{V}_{+}$ locally defined 
near the point $p_{+}$. All the trajectories associated to $(D)$ 
whose starting initial data is contained in $\mathcal{V}_{+}$ converge to
$p_{+}$ as $\tau\to\infty$. We can parametrise the set of trajectories
in $\mathcal{V}_{+}$ by means of one real parameter $\nu$ taking values in
some large interval. The behaviours (\ref{B1}) and (\ref{B2}) define two sets of values $\nu$. 
We prove that for very large values of $\nu$ the corresponding trajectory satisfies (\ref{B1}). 
On the contrary, if $\nu$ is very negative we show that there exists a 
$\tau_{\ast}=\tau_{\ast}(\nu)$ such that (\ref{B2}). 

It turns out that the sets of values $\nu$ such that the corresponding
trajectories satisfy either (\ref{B1}) or (\ref{B2}) are disjoint open sets. This
implies, the existence of $\nu$'s for which the corresponding trajectory does not satisfy neither 
(\ref{B1}) nor (\ref{B2}).

The final step is to show that the trajectories associated to such $\nu$
are globally defined in $\tau\in\mathbb{R}$ and that they satisfy
\begin{equation}\label{B3}
\left(  \frac{1}{\Phi}+\Phi\right)  +\left|  \frac{d\Phi}{d\tau}\right|
+\left|  \frac{d^{2}\Phi}{d\tau^{2}}\right|  \leq C\,,\quad\mbox{for any}\quad
\tau\in\mathbb{R} 
\end{equation}
for some $C>0$, and that
\begin{equation}\label{B4}%
\lim_{\tau\to-\infty}\theta(\tau)=-\frac{\pi}{2} \,.
\end{equation}
The idea is that if (\ref{B3}) and (\ref{B4}) are satisfied 
we can use the fact that the dynamics of $(D)$ 
become close to the ones associated to the trajectories contained in the unstable manifold
of $P_{s}$ for the system associated to (\ref{ODEtapas}) and the trajectories have no alternative 
but to approach $p_{-}$ as $\tau\to-\infty$.

The most technical part of the paper is the proofs of (\ref{B3})
and of (\ref{B4}). These require to show that oscillatory behaviours with
large amplitude for the solutions of $(D)$ as $\tau\to-\infty$ must have a
decreasing amplitude for decreasing $\tau$ if neither (\ref{B1}) nor
(\ref{B2}) take place. The key point is that the structure of oscillatory solutions can be identified by looking 
at the several asymptotic regimes of (\ref{S2E3}). There are, in particular,
two very distinctive ones. For instance, the balance $\xi^{2/3}H \sim\infty$ 
for very negative $\xi$ will be relevant in our analysis. In this case the behaviour 
of solutions is described by
\begin{equation}\label{ecuacionminima1}
\frac{d^{3} H}{d\xi^{3}}= - \xi^{2} \,.
\end{equation}
This equation can be integrated giving that, in such regions, $H$ behaves like
a fifth order polynomial. The solutions of (\ref{ecuacionminima1}) are in fact a 
two-parameter family of polynomials, as we shall see. 
On the other hand, if $\xi^{2/3}H \sim0$ on bounded intervals, the dominant
balance there is given by the equation
\begin{equation}\label{ecuacionminima2}
\frac{d^{3} H}{d\xi^{3}} = \frac{1}{H^{3}} \,.
\end{equation}
The analysis of (\ref{ecuacionminima2}) plays a crucial role in our proofs and was already
studied in \cite{CV}. 
The possibility of alternating regions where either (\ref{ecuacionminima1}) or 
(\ref{ecuacionminima2}) dominates, builds up a scenario where solutions with 
large oscillations exist: The bouncing region of the oscillations are described by 
(\ref{ecuacionminima2}) and the maximum amplitude regions are close to solutions 
of (\ref{ecuacionminima1}). This phenomenon has been already observed for 
(\ref{ODEtapas}) in \cite{WilsonJones} and explored rigorously in \cite{CV}.

In order to prove (\ref{B1}) and (\ref{B2}) we exploit this mechanism of oscillation. 
We argue by contradiction and assume first that (\ref{B3}) does not hold. 
This gives (after a number of technical lemmas) that there exists a sequence 
$\{\tau_{n}^{\ast}\}$ with $\lim_{n\to\infty}\tau_{n}^{\ast}=-\infty$ such that 
$\Phi(\tau_{n}^{\ast})$ is a local maximum and $\lim_{n\to\infty}\Phi(\tau_{n}^{\ast})=\infty$. 
We use that the oscillatory solutions with very large amplitude for very negative values of $\tau$ 
can be approximated, after a suitable rescaling, by a sequence of functions 
$|\xi|^{2/3}\mathcal{H}_{n}(\xi)$ where each $\mathcal{H}_{n}$ solves (\ref{ecuacionminima1}) 
in intervals $[\xi(\tau_{n+1}^{min}),\xi(\tau_{n}^{min})]$. The values $\tau_{n}^{min}$ being such that 
$\Phi(\tau_{n}^{min})$ is the minimum in $(\tau_{n}^{\ast},\tau_{n-1}^{\ast})$. In particular, in such intervals $\mathcal{H}_{n}(\xi)$ 
are close to a fifth order polynomial solving (\ref{ecuacionminima1}). 
The matching between two consecutive such functions is done into the inner region 
where $\Phi$ and $\mathcal{H}_{n}$ become close to $0$, as it turns out, this inner regions lies around $\tau_{n}^{min}$. 
As we have mentioned the dynamics in such bouncing region are dominated by (\ref{ecuacionminima2}) and the rigorous matching can be adapted 
from that performed for (\ref{ODEtapas}) (see \cite{CV}): the study of (\ref{ecuacionminima2}) 
reduces to the one of a phase-plane analysis in which the bouncing can be encoded into the 
behaviour of a separatrix. This object attracts trajectories for increasing $\xi$, 
implying that its behaviour is generic. Reading off this behaviour into the functions $\mathcal{H}_{n}$ 
implies that in the outer region they behave as a polynomial with a double zero near $\xi(\tau_{n+1}^{min})$. 
This in particular reduces the family of polynomials that give the outer 
region around each $\tau_{n}^{\ast}$ to a one-parameter family. Moreover, this analysis allows to get information on the relative size of 
consecutive maxima and minima, namely that the sequence of the maximum values decreases and that the sequence of 
minimum values increases (as $n \to \infty$) and these contradict the assumption that (\ref{B3}) does not hold.

The paper is organised as follows. Section~\ref{sec:preliminaries} is divided in three 
preliminary parts. First in Section~\ref{sec:summary} we give some results concerning (\ref{ODEtapas}), 
most of which are proved in \cite{CV}. In Section~\ref{sec:DS} we reformulate (\ref{S2E3}) as a four 
dimensional dynamical system and reformulate Theorem~\ref{Main} in this setting. 
The third part is Section~\ref{sec:loc-analysis} where we prove the existence of the centre-unstable 
manifold around $p_{+}$. Section~\ref{sec:3} is devoted to the analysis of the behaviours 
(\ref{B1}) and (\ref{B2}) for $(D)$; in Section~\ref{sec:stability} we show stability under small 
perturbations of solutions that satisfy either of these properties, and in Section~\ref{sec:classes} 
we give necessary conditions on solutions of (\ref{S2E3}) to satisfy either (\ref{B1}) 
or (\ref{B2}). With the analysis carried out up to here we can then prove in Section~\ref{sec:4} that there exist 
solutions on $\mathcal{V}_{+}$ that do not satisfy neither (\ref{B1}) nor (\ref{B2}). 
We continue by proving that these trajectories of $\mathcal{V}_{+}$ do satisfy (\ref{B3}) and (\ref{B4}).
In order to do that we first find in Section~\ref{sec:dynamics} that if (\ref{B3}) is not satisfied 
the sequences $\{\tau_{n}^{\ast}\}$ and $\{\tau_{n}^{min}\}$, described above, are well defined. 
Second, in Section~\ref{sec:proof:osci} we find the contradictory results that $\{\Phi(\tau_{n}^{\ast})\}$ is decreasing and that 
$\{\Phi(\tau_{n}^{min})\}$ is increasing. This part is very technical and needs by itself a few steps. 
Thus, in Section~\ref{sec:aux} we identify the scales of the outer region and the approximating polynomials 
near local maxima. This is based on the analysis of the solutions of (\ref{ecuacionminima1}) that is carried out in Appendix~\ref{sec:polynomials}.   
In Section~\ref{sec:rescale:H} we perform the right scaling of the solutions under consideration and identify the range in which they 
are approximated by the polynomials. 
In this section we also prove that the approximating polynomials must have a double zero. This step requires the analysis of 
(\ref{ecuacionminima2}) given in Appendix~\ref{sec:summary:II} as well as the {\it matching lemma} given in Appendix~\ref{sec:control}
(a result that has been adapted from \cite{CV}). In Section~\ref{sec:formal}, with detailed information of the matching regions, we derive an 
(iterative) expression that relates the elements of the sequence of local maxima and another that relates the local minima, and that 
contradict that (\ref{B3}) is not satisfied. Finally, in Section~\ref{sec:6} we finish the prove of Theorem~\ref{Main}.

Finally, we recall that equations similar to (\ref{ODEtapas}) have been studied intensively, see \cite{CrasterMatar}, 
\cite{Eggers}, \cite{EggersII}, \cite{Hoc01}, \cite{MyersRev}, \cite{MyersSolidII}, \cite{MyersSolidI}, \cite{ODB97} and \cite{wilson}), 
to mention a few, where similar equations arise in several related physical situations. Rigorous results concerning such equations can be 
found also in \cite{BerettaHP} and, concerning travelling wave solutions, in \cite{BS00}, \cite{Boatto}, \cite{KH75}, \cite{Mich88} and \cite{Ren96}. 
It is interesting to note that many of these models yield higher order ODEs describing oscillatory fluid interfaces. We refer to \cite{CV}, where 
this aspect and related works are put into context.

\section{Preliminaries}\label{sec:preliminaries}
\subsection{A summary of results for (\ref{ODEtapas})}\label{sec:summary} 
We now summarise some properties of (\ref{ODEtapas}), most
of which have been proved in \cite{CV} and will be used later in the proof of
Theorem~\ref{Main}. It is convenient to rewrite (\ref{ODEtapas}) in the
equivalent form%
\begin{equation}\label{S3E1}
\frac{d\Phi}{d\tau}=W\,,\ \frac{dW}{d\tau}=\Psi\,,\ \frac{d\Psi}{d\tau}%
=\frac{1}{\Phi^{3}}-1\,, 
\end{equation}
we then have the following result.

\begin{proposition}\label{Tapas1}
\begin{enumerate}
\item There is a unique critical point for (\ref{S3E1}) in the domain 
$\{\Phi>0\,, \ W\in\mathbb{R}\,, \ \Psi\in\mathbb{R}\}$ given by:%
\[
P_{s}=(\Phi,W,\Psi)=(1,0,0) \,.
\]

\item The point $P_{s}$ is hyperbolic. the stable manifold of (\ref{ODEtapas}) at the point
$P_{s}$ is tangent to the vector:
\[
v_{1}:=\left(
\begin{array}
[c]{c}%
3^{-\frac{2}{3}}\\
-3^{-\frac{1}{3}}\\
1
\end{array}
\right)
\]
and the corresponding eigenvalue is $\lambda_{1}:=-3^{\frac{1}{3}}$.

\item At $P_{s}$ there is a two-dimensional unstable manifold locally spanned by the
eigenvectors $v_{2}:=(-3^{\frac{1}{6}}/6\,,\,3^{\frac{2}{3}}/6\,,\,1)^T$ and 
$v_{3}:=(3^{\frac{5}{6}}/6\,,\,3^{\frac{1}{6}}/2\,,\,0)^T$. 
The eigenvalues associated to the plane spanned by $\{v_{2},v_{3}\}$ 
are $\lambda_{2}:=3^{\frac{1}{3}}(1+i\,3^{\frac{1}{2}})/2$ and $\lambda_{3}=\overline{\lambda_{2}}$.

\item The trajectories associated to (\ref{S3E1}) that are contained in the
stable manifold and satisfy $(\Phi,W,\Psi) \not\equiv P_{s}$, behave in one of the
two following ways for decreasing $\tau$: Either they are defined for all 
$\tau\in\mathbb{R}$ and satisfy
\begin{equation}\label{P1}
\lim_{\tau\to-\infty}(\Phi,W,\Psi)=(\infty,-\infty,\infty) 
\end{equation}
or, alternatively, there exists a $\tau_{\ast}>-\infty$ such that
\begin{equation}\label{P2}
\lim_{\tau\to(\tau_{\ast})^{+}}\Phi(\tau) =0 \,. 
\end{equation}
Moreover, the points of the stable manifold associated to $P_{s}$ with
$\Phi>1$ satisfy (\ref{P1}) and those with $\Phi<1$ satisfy (\ref{P2}).

\item Suppose that there exist $\tau_{0}\in \mathbb{R}$ and $C_{0}>1$ such that 
\[
(\Phi(\tau),W(\tau),\Psi(\tau)) \in
\left\{(\Phi,W,\Psi)\in\mathbb{R}^{3}: 
\ \frac{1}{C_{0}} \leq\Phi\leq C_{0},\ -C_{0}\leq W\leq C_{0}\ ,\ -C_{0}\leq\Psi\leq C_{0}\right\}
\] 
for all $\tau\leq \tau_{0}$, then
\begin{equation}\label{S3E3b}
\lim_{\tau\to-\infty}(\Phi(\tau),W(\tau),\Psi(\tau)) =P_{s}
\end{equation}
and the corresponding trajectory is contained in the unstable manifold of $P_{s}$.
\end{enumerate}
\end{proposition}

\begin{proof}[Proof]
All the statements of this proposition have been already proved in \cite{CV} 
except for \textit{(v)}. In order to prove this, we use an argument similar to
the one used to prove Lemma~2.4 in \cite{CV}. We first recall that there exists an
increasing Lyapunov functional $E$ associated to (\ref{ODEtapas}):
\[
E:=\Psi W + \frac{1}{2\Phi^{2}} + \Phi\,,\quad\frac{dE}{d\tau}=\Psi^{2}\geq0\,. 
\]
This and the assumptions made imply that
\begin{equation}\label{finite:integral}
\int_{-\infty}^{\tau_{0}}\Psi^{2}(s) ds<\infty\,.
\end{equation}
Using (\ref{S3E1}), it then follows that $\lim_{\tau\to-\infty}\Psi(\tau)=0$. 
Indeed, arguing by contradiction, one can construct a sequence $\tau_{n}\to-\infty$ 
such that there exits a $\varepsilon_{0}>0$ such that either 
$\Psi(\tau_{n})\geq \varepsilon_{0}$ or $\Psi(\tau_{n})\leq -\varepsilon_{0}$. 
Then (\ref{S3E1}) implies that $d\Psi/d\tau\geq -1$, so either 
$\Psi(\tau)\geq\varepsilon_{0}+(\tau_{n}-\tau)$ for $\tau>\tau_{n}$ or 
$\Psi(\tau)\leq -\varepsilon_{0}+(\tau_{n}-\tau)$ for $\tau>\tau_{n}$. 
But this contradicts (\ref{finite:integral})) since either 
$\int_{\tau_{n}}^{\tau_{n}+\varepsilon_{0}/2}(\Psi(\tau))^{2}d\tau\geq \varepsilon^{3}/8$ or 
$\int_{\tau_{n}-\varepsilon_{0}/2}^{\tau_{n}}(\Psi(\tau))^{2}d\tau\geq \varepsilon^{3}/8$ 
for all $n$. 

Now the second equation in (\ref{S3E1}) implies that $W$ remains 
approximately constant as $\tau\to-\infty$ in any finite interval of arbitrary
fixed length $L$. Therefore, if there is a subsequence $\{\tau_{n}\}$ with
$\lim_{n\to\infty}\tau_{n}=-\infty$ satisfying $\lim_{n\to\infty}W(\tau
_{n})\neq 0$, we obtain that $\inf_{\tau\in\left[  \tau_{n},\tau_{n}+L\right]
}|W(\tau)| \geq\varepsilon_{0}>0$ for $n$ sufficiently large. It then follows
from the first equation in (\ref{S3E1}) that the condition $\frac{1}{C_{0}}%
\leq\Phi\leq C_{0}$ fails if $L$ is assumed to be sufficiently large 
(integration on the interval $(\tau_{n},\tau_{n}+L)$ for sufficiently large $n$ 
gives that 
$|\Phi(\tau_{n}+L)-\Phi(\tau_{n})|>\varepsilon_{0} L>0$, but 
$|\Phi(\tau_{n}+L)-\Phi(\tau_{n})|<(C_{0}^{2}-1)/C_{0}$ for all $n$ and $L$). 
Therefore $\lim_{\tau\to-\infty}W(\tau) =0$. 

Using the last equation in (\ref{S3E1}) as well as the fact that 
$\lim_{\tau\to-\infty}\Psi(\tau)=0$ it then follows in a similar way that 
$\lim_{\tau\to-\infty}\Phi(\tau)=1$. This gives (\ref{S3E3b}) and the result follows.
\end{proof}

The next lemma gives the detailed asymptotic behaviour in both cases (\ref{P1}) and
(\ref{P2}):

\begin{lemma}\label{Tapas2}
The trajectories associated to (\ref{S3E1}) that are contained in the stable
manifold and satisfy $(\Phi,W,\Psi) \neq P_{s}$ satisfy that either they are
defined for all $\tau\in\mathbb{R}$ and (\ref{P1}) holds with
\begin{equation}
\lim_{\tau\to-\infty}\frac{\Phi(\tau)}{\tau^{3}}=-\frac{1}{6}\,,
\label{P1:prime}%
\end{equation}
or, alternatively, there exists a $\tau_{\ast}>-\infty$ such that (\ref{P2})
holds with
\begin{equation}
\lim_{\tau\to\tau_{\ast}^{+}}\frac{\Phi(\tau) }{(\tau-\tau_{\ast})^{\frac{3}{4}}}
=\left(  \frac{64}{15}\right)^{\frac{1}{4}}\,. 
\label{P2:prime}%
\end{equation}

\end{lemma}

\begin{proof}[Proof]
That either (\ref{P1}) or (\ref{P2}) hold is just the statement \textit{(iv)} of 
Proposition~\ref{Tapas1}. Then, (\ref{P1:prime}) follows from (\ref{P1}) and
integrating the equation. Indeed, for all $\Phi_{0}>0$ there exist a $\tau_{0}$ with
$|\tau_{0}|$ large enough so that for all $\tau<\tau_{0}$, then 
\[
\Phi(\tau)>\Phi_{0}>0\,, \ \frac{d\Phi(\tau_{0})}{d\tau}<-\Phi_{0}<0\,, 
\ \frac{d^{2}\Phi(\tau_{0})}{d\tau^{2}}>\Phi_{0}>0\,.
\]
Thus
\[
-1<\frac{d^{3}\Phi}{d\tau^{3}}<\frac{1}{\Phi_{0}^{3}}-1\,.
\]
and integrating this expression with $\tau<\tau_{0}<0$ we obtain
\[
\left(  \frac{1}{\Phi_{0}^{3}}-1\right)  \frac{(\tau-\tau_{0})^{3}}{6} \, < \,
\Phi(\tau) \, < \, \Phi(\tau_{0}) - \frac{d\Phi(\tau_{0})}{d\tau} (\tau_{0}-\tau)
+\frac{d^{2}\Phi(\tau_{0})}{d\tau^{2}}  \frac{(\tau-\tau_{0})^{2}}{2} 
-\frac{(\tau-\tau_{0})^{3}}{6}\,,
\]
for $|\tau_{0}|$ large enough. Then dividing by $-\tau^{3}/6$ and taking the limit 
$\tau\to-\infty$ implies (\ref{P1:prime}), since $\tau_{0}$ can be made arbitrarily 
negative and $\Phi_{0}$ arbitrarily large.

In order to prove (\ref{P2:prime}) we use the phase-plane analysis of 
the Appendix~\ref{sec:summary:II}. 

We employ the transformation (\ref{SEphi0-trans}) with $\zeta$ replaced by $\tau$ for (\ref{S3E1}) 
(see also \cite{CV}) that gives the system
\begin{equation}
\label{Phi-z-system}\frac{d\Phi}{dz}=u\,\Phi\, , \quad\frac{du}{dz}
=v+\frac{1}{3}u^{2}\,, \quad\frac{dv}{dz}=1+\frac{5}{3}u\,v-\Phi^{3}\,,
\end{equation}
which corresponds to (\ref{phi0system}) with $\Phi=0$ in the last equation.
For further reference, the flow field of the phase plane of (\ref{phi0system}) is also depicted in
Figure~\ref{phase-plane} where, in particular, the direction of the field, the critical 
point $(u_{e},v_{e})$ and the attractive separatrix $v=\bar{v}(u)$ are shown. 
Using Lemma~\ref{tozero-bwds} and the behaviour of trajectories of the system
(\ref{phi0system}) entering the only critical point $(u_{e},v_{e})$ as $z\to -\infty$ 
it is easy to show that if
\begin{equation}
\label{Hyp1}
\lim_{z\to -\infty}\Phi(z)\to 0\,, \ \mbox{and} \ \|(u,v)\| 
\ \mbox{is uniformly bounded as} \ z\to -\infty\,,
\end{equation}
then, by a bootstrap argument, the trajectory $(u,v)$ remains close to $(u_{e},v_{e})$, in particular
the estimate (\ref{W3E6e}) holds for $z$ large enough. That (\ref{Hyp1}) and that 
$\tau_{\ast}=\lim_{z\to-\infty}\tau(z)$ are satisfied is a consequence of the proof of 
(\ref{P2}) in \cite{CV} and the transformation (\ref{SEphi0-trans}).

Using the first equation in (\ref{Phi-z-system}) one obtains that there exists positive 
constants $C_{1}$ and $C_{2}$ such that
\begin{equation}\label{phi:0:esti}
e^{C_{2}e^{\lambda z}} e^{u_{e}(z-z_{0})}\Phi(z_{0})
<\Phi(z)<e^{C_{1}e^{\lambda z}} e^{u_{e}(z-z_{0})}\Phi(z_{0})\,,
\end{equation}
with $\lambda=\mbox{Re}(\lambda_{+})=(7/2)(1/(15^{\frac{1}{3}})>0$, for all $z<z_{0}$ where $z_{0}<0$ and $|z_{0}|$ is 
sufficiently large. And this in particular implies that
\begin{equation}\label{phi:0:limit}
\lim_{z\to-\infty}\Phi(z) e^{-u_{e}z}
=\Phi(z_{0})e^{-u_{e}z_{0}}\,.
\end{equation}
Using now the last equation of (\ref{SEphi0-trans}) and (\ref{phi:0:esti}) one
can infer that, considering $\tau$ as a function of $z$  
\[
\lim_{z\to-\infty} (\tau(z)-\tau_{\ast})^{\frac{3}{4}} e^{-u_{e}z}
=\left(  \frac{15}{64}\right)^{\frac{1}{4}}\Phi(z_{0})e^{-u_{e}z_{0}}\,,
\]
(where we use $u_{e}=(5/9)^{\frac{1}{3}}$ to compute the explicit coefficient). 
Finally, this and (\ref{phi:0:limit}) imply (\ref{P2:prime}).
\end{proof}

\subsection{A dynamical systems approach}\label{sec:DS} 
As anticipated in the Introduction, in this section we
reformulate the main result in terms of a Dynamical Systems approach. We first
transform (\ref{S2E3}) into a suitable systems of four autonomous ODEs, and
(\ref{S2E7}) into its corresponding boundary conditions.

Since we are interested in solutions for which $|\xi|^{\frac{2}{3}}H$ remains
bounded for all $\xi$ it is convenient to introduce the following change of
variables
\begin{equation}
H(\xi) =\frac{1}{(\xi^{2}+1)^{\frac{1}{3}}}\Phi(\tau) 
\label{S2E8a}%
\end{equation}
where the variable $\tau$ is defined by means of
\begin{equation}
(\xi^{2} + 1)^{\frac{4}{9}}d\xi= d\tau\,,\quad\tau
= \int_{0}^{\xi}(\eta^{2} +1)^{\frac{4}{9}}d\eta\,. 
\label{S2E8}%
\end{equation}
With this transformation, we have that
\begin{align}
\frac{dH}{d\xi}  &  
= -\frac{2}{3} \frac{\xi}{(\xi^{2}+1)^{\frac{4}{3}}}\Phi\,
+\,(\xi^{2}+1)^{\frac{1}{9}}\frac{d\Phi}{d\tau}\,,
\label{1st:dH}\\
\frac{d^{2}H}{d\xi^{2}}  &  
= -\frac{2}{3} \frac{1-\frac{5}{3}\xi^{2}}{(\xi^{2}+1)^{\frac{7}{3}}}\Phi
-\frac{4}{9}\frac{\xi}{(\xi^{2}+1)^{\frac{8}{9}}}\frac{d\Phi}{d\tau} \,
+\, (\xi^{2}+1)^{\frac{5}{9}}\frac{d^{2}\Phi}{d\tau^{2}}\,, 
\label{2nd:dH}%
\end{align}
and (\ref{S2E3}) becomes
\begin{equation}\label{S2E9}
\frac{d^{3}\Phi}{d\tau^{3}}
=\frac{1}{\Phi^{3}}-\frac{\xi(\tau)^{2}+ a}{\xi(\tau)^{2}+1}-F(\tau) \,,
\quad n\in\mathbb{N}\,,\quad s\in[-1,0]
\end{equation}
with
\begin{equation}\label{S2E9bis}
F(\tau) =\frac{16}{3}\frac{\xi}{(\xi^{2}+1)^{\frac{10}{3}}} 
\left(1-\frac{14}{9}\frac{\xi^{2}}{ \xi^{2}+1} \right)  \Phi
+ \frac{1}{(\xi^{2}+1)^{\frac{17}{9}}}\left( \frac{208}{81}\frac{\xi^{2}}{\xi^{2}+1}
-\frac{10}{9}\right)  \frac{d\Phi}{d\tau} 
+\frac{2}{3}\frac{\xi}{(\xi^{2}+1)^{\frac{13}{9}}} \frac{d^{2}\Phi}{d\tau^{2}}\,, 
\end{equation}
where $\xi$ is given as a function of $\tau$ by means of (\ref{S2E8}). In
other words, we use $\tau$ as independent variable, while $\xi$ becomes a
dependent one, making the system autonomous.

It is convenient to transform $\xi$ further into a new variable that takes
values in a compact set, namely, we define the variable $\theta$ by
\begin{equation}
\xi=\tan\theta\quad\theta\in\left[  -\frac{\pi}{2},\frac{\pi}{2}\right]  \,.
\label{xi}%
\end{equation}
Finally, we can reformulate (\ref{S2E8}) and (\ref{S2E9})-(\ref{S2E9bis})  as
\begin{align}
\frac{d\Phi}{d\tau}  &  =W\,,\label{compact1}\\
\frac{dW}{d\tau}  &  =\Psi\,,\label{compact2}\\
\frac{d\Psi}{d\tau}  &  =\frac{1}{\Phi^{3}}-1-(a-1)(\cos\theta)^{2}
-\left[
\left(  \frac{16}{3}\sin\theta-\frac{224}{27}(\sin\theta)^{3}\right)
(\cos\theta)^{\frac{17}{3}}\,\Phi\right. \nonumber\\
&  \left.  +\left(  \frac{208}{81}(\sin\theta)^{2}-\frac{10}{9}\right)
(\cos\theta)^{\frac{34}{9}}\,W+\frac{2}{3}\sin\theta(\cos\theta)^{\frac{17}{9}}\,\Psi
\right]  \,,\label{compact3}\\
\frac{d\theta}{d\tau}  &  =(\cos\theta)^{\frac{26}{9}} \,, \label{compact4}%
\end{align}
that has critical points
\begin{equation}\label{stst}
p_{-}:=\left( 1,0,0,-\frac{\pi}{2}\right) \quad\mbox{and} \quad p_{+}:=\left( 1,0,0, \frac{\pi}{2}\right)\,.
\end{equation}

We aim to prove the following theorem:
\begin{theorem}\label{hetcon} 
There exists a heteroclinic connection of the system (\ref{compact1})-(\ref{compact4}) 
between the points $p_{-}$ and $p_{+}$ given in (\ref{stst}).
\end{theorem}
We notice that Theorem~\ref{Main} is just a corollary of Theorem~\ref{hetcon};
this is implied by (\ref{S2E7}) and (\ref{S2E8a}). 

We point out that the system (\ref{compact1})-(\ref{compact4}) reduces to
(\ref{S3E1}) on the subspaces $\theta=-\pi/2$ and $\theta=+\pi/2$. We shall take
advantage of this fact in some of the arguments that follow.

\subsection{Existence of the centre-stable manifold}\label{sec:loc-analysis}
We now proceed to describe in detail the construction of a centre-stable manifold 
at $p_{+}$ that we denote by $\mathcal{V}_{+}$. Let us first define a set of 
transformations $F_{\tau}(x),\ \tau\in\mathbb{R}$ for any given 
$x\in\mathbb{R}^{+}\times\mathbb{R}^{2}\times\left[-\frac{\pi}{2},\frac{\pi}{2}\right]$ 
by means of
\begin{equation}\label{C1}
(\Phi(\tau),W(\tau),\Psi(\tau),\theta(\tau))=F_{\tau}(x) 
\end{equation}
where $(\Phi,W,\Psi,\theta)$ solves (\ref{compact1})-(\ref{compact4}) with
$(\Phi(0),W(0),\Psi(0),\theta(0))=x$. Classical ODE theory ascertains that the
family of transformations $F_{\tau}(\cdot)$ is well defined in some suitable
interval $\tau\in(\tau_{1}(x),\tau_{2}(x))$. We have the following result.

\begin{proposition}\label{set:shooting} There exists a two-dimensional $C^{1}$ manifold 
$\mathcal{V}_{+}$ contained in the ball $B_{\delta}(p_{+})\cap\mathbb{R}^{3}\times[-\pi/2,\pi/2]$ for
some $\delta>0$ sufficiently small, tangent to the subspace spanned by the vectors
\begin{equation}
\tilde{v}_{1}=\left(
\begin{array}
[c]{c}%
3^{-\frac{2}{3}}\\
-3^{-\frac{1}{3}}\\
1\\
0
\end{array}
\right)  \,,\quad\tilde{v}_{4}=\left(
\begin{array}
[c]{c}%
0\\
0\\
0\\
1
\end{array}
\right)  \, . 
\label{stable:eigenvec}%
\end{equation}
If $x\in\mathcal{V}_{+}$, the flow $F_{\tau}$ defined in (\ref{C1}) is defined
for any $\tau>0$ and
\begin{equation}
F_{\tau}(x) \in\mathcal{V}_{+}\quad\mbox{for any} \quad\tau\geq0
\label{mansta}%
\end{equation}
with
\begin{equation}
\lim_{\tau\to\infty}F_{\tau}(x) =p_{+}\,. 
\label{limman}%
\end{equation}
\end{proposition}

\begin{proof}[Proof]
In order to apply standard results it is convenient to extend the range of 
values of $\theta$, replacing $\cos\theta$ by $|\cos\theta|$, where the system 
(\ref{compact1})-(\ref{compact4}) is defined. 
The resulting system can be defined in a neighbourhood of $p_{+}$ and the right hand side of
(\ref{compact1})-(\ref{compact4}) is in $C^{\frac{17}{9}}(\mathbb{R}^{3}\times(-\pi/2,\pi/2))$. 
Since $\frac{17}{9}>1$ we can apply the results in
\cite{Gallay}. In our setting, this means the existence of a two-dimensional 
manifold $\mathcal{V}_{+}\in C^{\frac{17}{9}}(\mathbb{R}^{3}\times(-\pi
/2,\pi/2))$ tangential to the plane spanned by $\{\tilde{v}_{1}, \tilde{v}_{4}\}$ 
at $p_{+}$ that remains invariant under the flow $F_{\tau}$ if one can prove that 
the corresponding trajectories on this manifold remain inside a ball 
$B_{\delta}(p_{+})$ for some small $\delta>0$. Let us show that $\mathcal{V}_{+}$ 
is invariant.

Let us consider a four-dimensional cube $\mathcal{Q}=[1-\delta/2,1+\delta/2]\times[-\delta/2,\delta/2]^{2}\times [\pi/2-\delta/2,\pi/2]$ 
contained in a ball $B_{\delta}(p_{+})$. 
The cube has four pairs of parallel -$3$ dimensional- sides. One pair with normal
direction $\tilde{v}_{1}$, another pair with normal direction $\tilde{v}_{4}$,
the other two pairs of parallel sides contain a plane parallel to the one
spanned by $\tilde{v}_{1}$ and $\tilde{v}_{4}$. The set $\mathcal{Q}\cap\mathcal{V}_{+}$ gives
four $C^{1}$ curves and, due to the tangency of $\mathcal{V}_{+}$ to the plane
spanned by $\tilde{v}_{1}$ and $\tilde{v}_{4}$, two are contained in each of
the parallel sides of the cube that are orthogonal to $\tilde{v}_{1}$, and the
other two are contained in parallel subspaces orthogonal to $\tilde{v}_{4}$. 
More specifically, one of later is contained in the subspace $\mathbb{R}^{3}\times\{\theta=\pi/2\}$. 
Notice that $\mathcal{V}_{+}\cap(\mathbb{R}^{3}\times\{\theta=\pi/2\})$ 
gives a portion of the stable manifold associated 
to (\ref{S3E1}) for $\delta$ small enough. Therefore, if 
$x\in\mathcal{V}_{+}\cap(\mathbb{R}^{3}\times\{\theta=\pi/2\})$, $F_{\tau}(x)\in B_{\delta}(p_{+})$ 
for arbitrary values of $\tau>0$. On the other hand, for the curve contained in a subspace
with constant $\theta<\pi/2$ and orthogonal to $\tilde{v}_{4}$, we use the
fact that $\theta$ is increasing, thus trajectories could only scape the cube
through the other boundaries that intersect $\mathcal{V}_{+}$. But the
points $x$ on the other two boundary curves satisfy $\tilde{v}_{1}\cdot(x-p_{+})=\pm c\delta$ for some $c>0$ (small or at most of order one). We
then use that
\[
\left.  \frac{d}{d\tau}
\left(  \frac{\langle\tilde{v}_{1}\cdot(F_{\tau}(x)-p_{+})\rangle^{2}}{2} \right)  
\right|_{\tau=0}
=\langle\tilde{v}_{1}\cdot(x-p_{+})\rangle\left\langle \tilde{v}_{1}\cdot\left.  
\frac{dF_{\tau}(x) }{d\tau} \right|_{\tau=0}\right\rangle \,.
\]
Since the manifold $\mathcal{V}_{+}$ is tangent to the plane spanned by
$\tilde{v}_{1}$ and $\tilde{v}_{4}$ it follows, using \textit{(iii)} in
Proposition~\ref{Tapas1} as well as (\ref{stable:eigenvec}) that
\begin{equation}\label{C2}%
\left.  \frac{d}{d\tau}
\left(  \frac{\langle\tilde{v}_{1}\cdot(F_{\tau}(x)-p_{+})\rangle^{2}}{2} \right)  
\right|_{\tau=0} 
=-3^{\frac{1}{3}}\langle\tilde{v}_{1}\cdot(x-p_{+})\rangle^{2} 
+ o\left(  \langle\tilde{v}_{1}\cdot(x-p_{+}) \rangle^{2}\right)  \,. 
\end{equation}
Therefore, if $\delta$ is sufficiently small this quantity is negative and the
trajectories in $\mathcal{V}_{+}$ remain always in the ball $B_{\delta}(p_{+})$ and 
(\ref{mansta}) follows. It only remains to show (\ref{limman}).
To this end, we observe that (\ref{compact4}) implies $\lim_{\tau\to\infty}\theta(\tau)=\pi/2$. 
Using (\ref{C2}) we then obtain (\ref{limman}).
\end{proof}

For further reference, let us denote by $\Pi\subset\mathbb{R}^{4}$ the affine
plane spanned by the stable eigenvectors at $p_{+}$, namely,
\begin{equation}\label{parPlane}
\Pi=\{ w=(\nu,\sigma)\in p_{+}+\mathbb{R}^{4}:\ w-p_{+}= 
\nu\tilde{v}_{1}+\sigma\tilde{v}_{4} \ \nu\,, \sigma\in\mathbb{R} \}\,,
\end{equation}
with $\tilde{v}_{1}$ and $\tilde{v}_{4}$ as in (\ref{stable:eigenvec}). Every 
$w\in\Pi$ can be identified by its coordinates, thus we write $w=(\nu,\sigma)$ and
$p_{+}=(0,0)$ with this set of coordinates. 
Since $\mathcal{V}_{+}$ is tangent to $\Pi$ at $p_{+}$, there exist local
differentiable parametrisation of $\mathcal{V}_{+}$.

\begin{lemma}[Local parametrisation of $\mathcal{V}_{+}$]\label{trivia}
Let $\Pi$ be given by (\ref{parPlane}). There exists a $\delta_{0}>0$ 
and a differentiable mapping $\Lambda:\Pi\to\mathbb{R}^{4}$ that maps a neighbourhood of 
$\Pi$ into $\mathcal{V}_{+}\cap B_{\delta_{0}}(p_{+})\cap\{\theta\leq\pi/2\}$. Moreover, 
$\partial_{\nu}\Lambda(0,0)$, $\partial_{\sigma}\Lambda(0,0)\in\Pi$.
\end{lemma}

\section{Analysis of the behaviours (\ref{B1}) and (\ref{B2})}\label{sec:3}
\subsection{Stability}\label{sec:stability} 
We now prove that both asymptotic behaviours (\ref{B1}) and 
(\ref{B2}) represent two disjoint open sets of solutions of 
(\ref{compact1})-(\ref{compact4}). More precisely, we have the following results:

\begin{lemma}\label{continfinity}
Suppose that $F_{\tau}(x)$ is a solution of (\ref{compact1})-(\ref{compact4}) with 
$x\in\mathbb{R}^{+}\times\mathbb{R}^{2} \times(-\pi/2,\pi/2)$. 
Let us also assume that for such a solution $\lim_{\tau\to-\infty}\Phi(\tau)=\infty$. 
Then, there exists a $\delta=\delta(x)>0$ sufficiently small
such that for any 
$y\in B_{\delta}(x) \cap(\mathbb{R}^{+}\times\mathbb{R}^{2}\times(-\pi/2,\pi/2))$ 
$F_{\tau}(y) = (\tilde{\Phi}(\tau), \tilde{W}(\tau),\tilde{\Psi}(\tau),\tilde{\theta}(\tau))$ 
satisfies
\begin{equation}\label{B5}%
\lim_{\tau\to-\infty}\tilde{\Phi}(\tau) =\infty\,. 
\end{equation}
\end{lemma}

\begin{proof}[Proof]
It is convenient to use, in order to prove the result, the original equation
(\ref{S2E3}) that is equivalent in the set 
$\mathbb{R}^{+}\times\mathbb{R}^{2}\times(-\pi/2,\pi/2)$ 
to the system (\ref{compact1})-(\ref{compact4}) by means of the change of 
variables (\ref{S2E8a}), (\ref{S2E8}). 

We first recall that (\ref{compact4}) implies that $\theta\to-\pi/2$ as 
$\tau\to-\infty$. Therefore, by (\ref{xi}), $\lim_{\tau\to-\infty}\xi=-\infty$. 
On the other hand our hypothesis on $\Phi$ as well as (\ref{S2E8a}) and (\ref{S2E8}) imply
that
\begin{equation}\label{B7}
\lim_{\xi\to-\infty}|\xi|^{\frac{2}{3}} H(\xi) = \infty\,, 
\end{equation}
then this and (\ref{S2E3}) yields the existence of a  $\xi_{0}=\xi_{0}(x)<0$ with 
$|\xi_{0}|$ large enough such that
\[
\frac{d^{3}H}{d\xi^{3}}\leq-\frac{\xi^{2}}{2} \quad \mbox{for all}\quad \xi\leq\xi_{0}\,.
\]
Integration this expression gives%
\begin{align*}
\frac{d^{2}H(\xi)}{d\xi^{2}}  &  \geq-\frac{\xi^{3}}{6} 
+ \frac{\xi_{0}^{3}}{6} +\frac{d^{2}H(\xi_{0})}{d\xi^{2}}\,,\\
\frac{dH(\xi)}{d\xi}  &  \leq-\frac{\xi^{4}}{24} + \frac{\xi_{0}^{4}}{24} +
\left(  \frac{d^{2}H(\xi_{0})}{d\xi^{2}} +\xi_{0}^{3}\right)  (\xi-\xi_{0}) +
\frac{dH(\xi_{0})}{d\xi}%
\end{align*}
 Therefore, there exist a $\xi_{1}=\xi_{1}(x) <0$ with $|\xi_{1}|$ large enough such that
\begin{equation}\label{B7a}
\frac{d^{2}H(\xi)}{d\xi^{2}} >0\,,\quad \frac{dH(\xi)}{d\xi}<0\quad\mbox{for}\quad
\xi\leq\xi_{1} \,. 
\end{equation}
We assume, without loss of generality, that $\xi_{1}^{2}+a>0$ by taking 
$\xi_{1}$ even larger is necessary.

Now for 
$F_{\tau}(y)=(\tilde{\Phi}(\tau),\tilde{W}(\tau),\tilde{\Psi}(\tau), \tilde{\theta}(\tau))$ 
with $y\in B_{\delta}(x)\cap(\mathbb{R}^{+}\times\mathbb{R}^{2}\times(-\pi/2,\pi/2))$ 
we define $\tilde{H}(\xi)=(\xi^{2}+1)^{-1/3}\tilde{\Phi}(\tau)$. 
Since the changes of variables (\ref{S2E8a}) and (\ref{S2E8}) are smooth, it follows, 
using (\ref{B7}) and (\ref{B7a}) and standard continuous dependence arguments for 
ODEs, that
\begin{equation}\label{B8}
\tilde{H}(\xi_{1}) \geq\frac{2^{\frac{1}{3}}}{(\xi_{1}^{2}+a)^{\frac{1}{3}}}\,, 
\quad\frac{d^{2}\tilde{H}(\xi_{1})}{d\xi^{2}} >0\,, \quad\frac{d\tilde{H}(\xi_{1})}{d\xi} <0\,. 
\end{equation}
Integration of (\ref{S2E3}) for the unknown $\tilde{H}$ and (\ref{B8}) imply that for all $\xi\leq \xi_{1}$
\begin{equation}\label{B9}
\tilde{H}(\xi)\geq\frac{2^{\frac{1}{3}}}{(\xi_{1}^{2}+a)^{\frac{1}{3}}}
+\int_{\xi}^{\xi_{1}}\int_{s_{1}}^{\xi_{1}}\int_{s_{2}}^{\xi_{1}} (s_{3}^{2}+a) ds_{3} ds_{2}ds_{1}
-\int_{\xi}^{\xi_{1}} \int_{s_{1}}^{\xi_{1}}\int_{s_{2}}^{\xi_{1}}
\frac{ds_{3}}{(\tilde{H}(s_{3}))^{3}} ds_{2} ds_{1} \,. 
\end{equation}
Suppose that
\begin{equation}\label{F1}
\tilde{H}(s) \geq\frac{2^{\frac{1}{3}}}{(s^{2}+a)^{\frac{1}{3}}}
\quad\mbox{for}\quad\xi\leq s\leq\xi_{1}\,. 
\end{equation}
Therefore, it would follow from (\ref{B9}) that:%
\begin{equation}\label{F2}
\tilde{H}(\xi) \geq\frac{ 2^{\frac{1}{3}}}{(\xi^{2}+a)^{\frac{1}{3}}}+
\frac{1}{2}\int_{\xi}^{\xi_{1}}\int_{s_{1}}^{\xi_{1}}\int_{s_{2}}^{\xi_{1}}
(s_{3}^{2}+a)ds_{3}ds_{2}ds_{1} 
\end{equation}
where we use that $\xi^{2}\geq\xi_{1}^{2}$. We can then extend the inequality
(\ref{F1}) to a larger range of values of $\xi$ and therefore the inequality
(\ref{F2}) also follows for all $\xi\leq\xi_{1}$ with $\xi$ in the extended interval. 
Since the integral term on the
right-hand side of (\ref{F2}) tends to infinity as $\xi\to-\infty$, we obtain
(\ref{B5}) as well.
\end{proof}

\begin{lemma}\label{contzero} 
Suppose that $F_{\tau}(x)$ is a solution of (\ref{compact1})-(\ref{compact4}) with 
$x\in\mathbb{R}^{+}\times\mathbb{R}^{2}\times(-\pi/2,\pi/2)$. Let us also assume that there 
exists a $\tau_{\ast}>-\infty$ such that $\lim_{\tau\to(\tau_{\ast})^{+}}\Phi(\tau)=0$. 
Then, there exists a $\delta=\delta(x)>0$ sufficiently small such that for 
any $y\in B_{\delta}(x)\cap(\mathbb{R}^{+}\times\mathbb{R}^{2}\times(-\pi/2,\pi/2))$ 
there exists a $\tilde{\tau}_{\ast}>-\infty$ such that
 $F_{\tau}(y)=(\tilde{\Phi}(\tau),\tilde{W}(\tau),\tilde{\Psi}(\tau),\tilde{\theta}(\tau))$ 
satisfies
\[
\lim_{\tau\to(\tilde{\tau}_{\ast})^{+}}\tilde{\Phi}(\tau)=0\,.
\]
\end{lemma}

\begin{proof}[Proof]
As in the previous proof, it is more convenient to use the original formulation (\ref{S2E3}). 
We again use the smooth transformations (\ref{S2E8a}) and (\ref{S2E8}) to 
interpret the results between either formulation. Thus let $H$ be the solution of 
(\ref{S2E3}) associated to $F_{\tau}(x)$. 
Let also $\xi_{\ast}$ be defined by $\tau_{\ast}=\int_{0}^{\xi_{\ast}}(\eta^{2}+1)^{\frac{4}{9}}d\eta$, 
We observe that $\tau_{\ast}>-\infty$ implies that $\xi_{\ast}>-\infty$, and the hypothesis on $\Phi$ becomes 
\begin{equation}\label{H:2:zero}
\lim_{\xi\to\xi_{\ast}} H(\xi)=0\,.
\end{equation}
Thus in regions close to $\xi_{\ast}$ we expect that the solutions are described by 
(\ref{ecuacionminima2}) and we employ the change of variables
(cf. Appendix~\ref{sec:summary:II}, (\ref{ecuacionminima}) and (\ref{SEphi0-trans})), namely, 
\begin{equation}\label{V1E1}
\frac{dH}{d\xi} = H^{-\frac{1}{3}}\,u \,, \quad 
\frac{d^{2}H}{d\xi^{2}}=H^{-\frac{5}{3}}\,v \,, \quad 
H(\zeta)= H(\xi)\,,\quad
\xi=\Omega(z)
\end{equation}
where $\Omega(z)$ is defined by means of
\begin{equation}
z=-\int_{\Omega(z)}^{0}\frac{ds}{(H(s))^{\frac{4}{3}}} \label{V1E2}\,.
\end{equation}
Then, $u(z)$ and $v(z)$ are defined for any 
$z>z_{\ast}$ where $z_{\ast}$ is given by $\xi_\ast=\Omega(z_\ast)$. 
Notice that $|z_{\ast}|$ may or may not be finite. Moreover, $(H,u,v)$ satisfy
\begin{equation} \label{V1E3}
\frac{dH}{dz}=u\,H\,,\quad 
\frac{du}{dz}=v+\frac{u^{2}}{3}\,, \quad
\frac{dv}{dz}=1 +\frac{5}{3}u\,v -(\Omega^{2}+a) H^{3} 
\end{equation}
where all functions, including $\Omega$, are functions of $z$.
The hypothesis on $\Phi$ translates into
\[
\lim_{z\to(z_{\ast})^{+}}((\Omega(z))^{2}+a)(H(z))^{3}=0\,.
\]

The phase-plane analysis associated to (\ref{V1E3}) with $H(z)\equiv 0$
is included in Appendix~\ref{sec:summary:II}. Relevant to the current analysis 
are Lemma~\ref{separa} (where $v=\bar{v}(u)$ is defined) and Lemma~\ref{tozero-bwds} 
that describes the overall flow.

We claim that there is a sequence $\{z_{n}\}$ such that $z_{n}\to(z_{\ast})^{+}$ as $n\to\infty$
 and that $(u(z_{n}),v(z_{n}))\in \{(u,v):\ u>0\,,\ v<0\}$ for all $n$ large enough. 
Before we prove this we note that for any sequence $\{z_{n}\}$ such that 
$z_{n}\to(z_{\ast})^{+}$ as $n\to\infty$, the trajectory 
$(u(z),v(z))$ must be in the half-plane $\{(u,v):\ u>0\}$ 
for $z<z_{n}$ if $n$ is large enough. 
Indeed, otherwise the first equation in (\ref{V1E3}) implies that 
$dH(\xi)/d\xi\leq 0$ for all $\xi$ near $\xi_{\ast}$ and this contradicts (\ref{H:2:zero}). 
Let us now prove that we can select such a sequence and that it also satisfies 
$v(z_{n})<0$ for all $n$ large enough.

Let $\{z_{n}\}$ be such that $z_{n}\to(z_{\ast})^{+}$ as $n\to\infty$ and suppose that $v(z_{n}) = 0$. 
Then, the third equation in (\ref{V1E3}) implies that
\[
\frac{dv}{dz}(z_{n})= 1-((\Omega(z_{n}))^{2}+a)(H(z_{n}))^{3}
\]
and since the last term converges to zero as $n\to\infty$, 
it follows that $v(z)$ becomes negative for some $z<z_{n}$ close to 
$z_{n}$ for $n$ large enough. Thus we can construct another sequence 
$\{\hat{z}_{n}\}$ with $\hat{z}_{n}<z_{n}$,   
$\hat{z}_{n}\to(z_{\ast})^{+}$ as $n\to\infty$ and such that 
$v(\hat{z}_{n}) < 0$ for $n$ large enough. 

Suppose now that $v(z_{n}>0$ for large enough $n$. Then,
the second equation in (\ref{V1E3}) implies that $(u(z),v(z))$ arrives 
to the half-line $\{u=0\,,\ v>0\}$ at some $\bar{z}_{n}< z_{n}$. 
For otherwise, the last equation in (\ref{V1E3}) 
implies that $(u(z),v(z))$ crosses the line $\{v=0\}$, and the
argument of the previous case applies. Therefore, there exists a
sequence $\{\hat{z}_{n}\}$ with $\hat{z}_{n}\to z_{\ast}$ as $n\to\infty$
such that one of the following possibilities take place:%
\begin{align}
\lim_{n\to\infty}v(\hat{z}_{n})  &  >\bar{v}(0)\,,\label{Cs1}\,,\\
\lim_{n\to\infty}v(\hat{z}_{n})  &  <\bar{v}(0)\,,\label{Cs2}\,,\\
\lim_{n\to\infty}v(\hat{z}_{n})  &  =\bar{v}(0)\,. \label{Cs3a}\,.%
\end{align}

In the case (\ref{Cs1}), we can approximate the evolution of $(u(z),v(z))$ 
in intervals of the form $z\in[\bar{z}_{n}-L,\bar{z}_{n}]$ by the system 
(\ref{phi0system}) using standard continuous dependence 
results and Lemma~\ref{tozero-bwds} implies that $(u(z),v(z))$ enters 
$\{(u,v):\ u>0,\ v<0\}$ at some $z<\bar{z}_{n}$ for $n$ large enough, 
and the claim follows.

Suppose now that (\ref{Cs2}) takes place. Using again continuous dependence we obtain that 
$(u(z),v(z))\in\{(u,v):\ 1+ 5uv/3<0\,, \ u<0\,,\ v>0\}=R_{5}$ 
for some $z<\bar{z}_{n}$ and $n$ large enough. 
In this region, and with $z$ close to $z_{\ast}$, then $v$ 
increases for decreasing $z$. Therefore, $d^{2}H(\xi)/d\xi^{2}$ 
remains positive and $dH(\xi)/d\xi$ is negative as long as 
$(u(z),v(z))$ stays in $R_{5}$. Moreover, due to the second equation in (\ref{V1E3}) 
$|u(z)|$ increases for decreasing $z$. This implies that the inequality 
$1+5 uv/3<0$ remains valid during all the evolution until $z=z_{\ast}$, 
thus also the inequalities $d^{2}H(\xi)/d\xi^{2}>0$, $dH(\xi)/d\xi<0$ remain valid. 
However, this contradicts (\ref{H:2:zero}) and (\ref{Cs2}) cannot hold.

It remains to study the case (\ref{Cs3a}). In this case there exist a small $L$
such that for $z\in (\bar{z}_{n}-L,\bar{z}_{n})$ $(u(z),v(z))$ 
remains close to the separatrix $v=\bar{v}(u)$. On the other hand, $(u(z),v(z))$
must return to $\{(u,v):\ u>0\}$ infinitely often as $z_{n}\to z_{\ast}$. 
Thus the trajectory must remain close to $\bar{v}$ for $z$ close to $z_{\ast}$, 
or otherwise the trajectory enters $R_{5}$ giving a contradiction as before, or it enters 
the region $\{(u,v):\ v<-\frac{u^{2}}{3}\}$ which contradicts (\ref{Cs3a}). 
Then Lemma~\ref{separa} \textit{(i)} implies that $d^{2}H(\xi)/d\xi^{2}>0$ and 
$dH(\xi)/d\xi<0$ remain valid during all the evolution for decreasing 
$z<\hat{z}_{n}$ for $n$ large enough, and this contradicts (\ref{H:2:zero}).

As in the proof of Lemma~\ref{continfinity} for 
$F_{\tau}(y)=(\tilde{\Phi}(\tau),\tilde{W}(\tau),\tilde{\Psi}(\tau), \tilde{\theta}(\tau))$ 
with $y\in B_{\delta}(x)\cap(\mathbb{R}^{+}\times\mathbb{R}^{2}\times(-\pi/2,\pi/2))$ 
we define $\tilde{H}(\xi)=(\xi^{2}+1)^{-1/3}\tilde{\Phi}(\tau)$ and 
the transformed functions $(\tilde{u}(\tilde{z})),\tilde{v}(\tilde{z}))$ 
by means of the transformations (\ref{V1E1}) and (\ref{V1E2}) with the obvious changes 
of notation.

We then notice that, by continuous dependence of solutions on the initial data, 
if $\delta>0$ is chosen sufficiently small then 
$(\tilde{u}(\bar{z}),\tilde{v}(\bar{z}))$ 
enters the region $\{ (u,v):\ u>0\,,\ v<0\}$ for some $\bar{z}$ close to
$z_{\ast}$ and therefore $d\tilde{H}(\bar{\xi})/d\xi>0$ and 
$d^{2}\tilde{H}(\bar{\xi})/d\xi^{2}<0$ for some $\bar{\xi}$ close to 
$\xi_{\ast}$ with $\bar{\xi}>\xi_{\ast}$. We have that $\tilde{H}(\bar{\xi})$ 
is small and $d^{3}\tilde{H}/d\xi^{3}>0$ as long as $\tilde{H}(\bar{\xi})$ is small. 
Integrating this inequality for $\xi<\bar{\xi}$ we obtain that 
$\frac{d\tilde{H}}{d\xi}(\xi)>0$, $\frac{d^{2}\tilde{H}}{d\xi^{2}}(\xi)<0$ 
and $\tilde{H}(\xi)$ remains small for $\xi<\bar{\xi}$ as long as $\tilde{H}$
is defined. Then, $\tilde{H}(\xi)$ vanishes for some $\tilde{\xi}_{\ast}>-\infty$, 
so the lemma follows.
\end{proof}

\subsection{Characterisation}\label{sec:classes}
We now give necessary conditions for the solutions of (\ref{S2E3}) to either
satisfy that
\begin{equation}\label{B1:H}
\lim_{\xi\to-\infty}H(\xi)=+\infty
\end{equation}
or that
\begin{equation}\label{B2:H}
\lim_{\xi\to(\xi_{\ast})^{+}}H(\xi)=0 \quad\mbox{for some}
\quad\xi_{\ast}>-\infty\,.
\end{equation}
Observe that these behaviours imply (\ref{B1}) and (\ref{B2}) respectively,
for the corresponding function $\Phi(\tau)$ given by (\ref{S2E8a}) and
(\ref{S2E8}). We start by giving necessary conditions for (\ref{B1:H}), but first we need
the following auxiliary calculus result.

\begin{lemma}\label{polyn} Given the polynomials
\[
P_{1}(Y) =-\frac{Y^{5}}{60}+\frac{Y^{4}}{12}-\frac{Y^{3}}{6}\quad 
\mbox{and}\quad P_{2}(Y) =-\frac{Y^{3}}{6}\,,
\]
then, they are strictly decreasing and positive for $Y<0$.
Moreover, if $\lambda\in \mathbb{R}$ satisfies $1+2\lambda>0$, then
\begin{equation}\label{In2}
P_{1}(Y)+\lambda P_{2}(Y)    \geq\frac{1}{2}P_{1}(Y) \,, 
\quad \mbox{for}\quad Y<0\,.
\end{equation}
If $1+2\lambda\leq 0$ then 
\begin{equation}\label{In3}
P_{1}(Y)+2\lambda P_{2}(Y)  > - \frac{4}{5}3^{\frac{3}{2}}\max\{|1+2\lambda|^{\frac{5}{2}} ,1 \}
\end{equation}
for $Y<0$, but $P_{1}(Y)+2\lambda P_{2}(Y)\geq 0$ if $Y\leq 5/2-\sqrt{25-40(1+2\lambda)}/2$. 
\end{lemma}
\begin{proof}[Proof]
The monotonicity properties of $P_{1}$ and $P_{2}$ are just an elementary
calculus exercise. The inequality (\ref{In2}) is a consequence of the fact
that the polynomial$-\frac{Y^{5}}{60}+\frac{Y^{4}}{12}-c\frac{Y^{3}}{6}=P_{1}(Y)+(c-1)P_{2}(Y)$ 
is non-negative and decreasing if $c\geq 0$, in particular
\[
\frac{1}{2}P_{1}(Y)+\lambda P_{2}(Y) 
=\frac{1}{2}\left[  -\frac{Y^{5}}{60}
+\frac{Y^{4}}{12}-(1+2\lambda)\frac{Y^{3}}{6}\right]
\]
is non-negative if $(1+2\lambda)\geq 0$, thus (\ref{In2}) holds.

If $c<0$ (i.e. $1+2\lambda<0$) then $P_{1}(Y)+(c-1)P_{2}(Y)<0$ in 
$Y\in((5-\sqrt{25-40c})/2,0)$. But there the polynomial is larger than or 
equal than the value of the minimum in $Y<0$, namely, 
\[
P_{1}(Y)+(c-1)P_{2}(Y) \geq
\left(\frac{1}{30}-\frac{(4-6c)^{\frac{1}{2}}}{60}-\frac{c}{15}\right)
(2-(4-6c)^{\frac{1}{2}})^{3} > -\frac{4(1-2c)^{\frac{5}{2}}}{15}  
\]
and (\ref{In3}) follows.
\end{proof}

We now give necessary conditions for (\ref{B1:H}) to hold.
\begin{proposition}\label{Linf} Let us assume that there exists a positive 
constant $c_{1}=c_{1}(a)>0$ and some $\xi_{0}\in\mathbb{R}$ with 
\begin{equation}\label{bigness1}
c_{1}>\left(\frac{24}{5}\right)^{3}|a|^{5}(1+3|a|) 
\quad \mbox{if} \quad |\xi_{0}|^{2}< -2a \quad (a<0)
\end{equation}
and
\begin{equation}\label{bigness2}
c_{1}>16(2+|a|)  
\quad \mbox{if} \quad |\xi_{0}|^{2}> -2a\,,
\end{equation}
such that a solution of (\ref{S2E3}) satisfies 
$((\xi_{0})^{2}+1+|a|)(H(\xi_{0}))^{3}\geq c_{1}$, $dH(\xi_{0})/d\xi<0$ and 
$d^{2}H(\xi_{0})/d\xi^{2}>0$. Then (\ref{B1:H}) holds.
\end{proposition}

\begin{proof}[Proof]
Integrating (\ref{S2E3}) three times for $\xi<\xi_{0}$ we obtain:
\begin{equation}\label{In1}%
H(\xi)\geq H(\xi_{0}) + \int_{\xi}^{\xi_{0}}\int_{s_{1}}^{\xi_{0}}\int_{s_{2}}^{\xi_{0}} 
\left((s_{3}^{2}+a) - \frac{1}{(H(s_{3}))^{3}}\right)  ds_{3}ds_{2}ds_{1}\,.
\end{equation}

Given the polynomials defined in Lemma~\ref{polyn} and letting, 
for every $\xi<\xi_{0}$, 
\[
Y=\left\{
\begin{array}{l}
\frac{\xi}{|\xi_{0}|} - 1  \ \mbox{if} \ \xi_{0} > 0 \\
\frac{\xi}{|\xi_{0}|} + 1  \ \mbox{if} \ \xi_{0} < 0
\end{array}
\right.
\]  
then, we can write
\begin{equation}\label{xi0:int}
\int_{\xi}^{\xi_{0}}\int_{s_{1}}^{\xi_{0}}\int_{s_{2}}^{\xi_{0}} 
(s_{3}^{2}+a)ds_{3}ds_{2}ds_{1}
  = |\xi_{0}|^{5}\left[  P_{1}(Y) + \frac{a}{|\xi_{0}|^{2}} P_{2}(Y)\right]\,.  
\end{equation}
Clearly $\xi/|\xi_{0}|<-1$ if $\xi<\xi_{0}$ and $\xi_{0}<0$ and 
$\xi/|\xi_{0}|<1$ if $\xi<\xi_{0}$ and $\xi_{0}>0$, thus in either case 
$Y<0$ and the polynomials are in the range of values considered 
in Lemma~\ref{polyn}. We can now distinguish two cases. 

Suppose first that $\xi_{0}^{2}<-2a$. Then a Gronwall type of argument shows 
that for any $\xi<\xi_{0}$ with $\xi\in[-\sqrt{2|a|},\sqrt{2|a|}]$ then 
$H(\xi)>1>0$. Indeed, as long as $H(\xi)>1$ then (\ref{In1}) can be 
estimated from below by
\begin{equation}\label{In4}
H(\xi)\geq \frac{c_{1}^{\frac{1}{3}}}{(3|a|+1)^{\frac{1}{3}}} + 
|\xi_{0}|^{5}\left[  P_{1}(Y) +\left(\frac{a}{|\xi_{0}|^{2}}
-1\right) P_{2}(Y)\right] 
\end{equation}
(using (\ref{xi0:int})). Then we can apply (\ref{In3}) with 
$2\lambda=a/|\xi_{0}|^{2}-1(\leq -1)$, hence 
\[
P_{1}(Y)+\left(\frac{a}{|\xi_{0}|^{2}}
-1\right) P_{2}(Y)  > - \frac{4}{5}3^{\frac{3}{2}}\max\left\{\frac{|a|^{\frac{5}{2}}}{|\xi_{0}|^{5}} ,1 \right\}
\]
using this in (\ref{In4}) yields
\[
\begin{array}{l}
H(\xi)>\frac{c_{1}^{\frac{1}{3}}}{(3|a|+1)^{\frac{1}{3}}} - \frac{4}{5}3^{\frac{3}{2}}|a|^{\frac{5}{2}}
\end{array}
\]
and (\ref{bigness1}) implies the claim by a continuity argument.

Let us assume now that $\xi_{0}^{2}>-2a$. Using (\ref{xi0:int}) we obtain
\[
\int_{\xi}^{\xi_{0}}\int_{s_{1}}^{\xi_{0}}\int_{s_{2}}^{\xi_{0}} 
(s_{3}^{2}+a)ds_{3}ds_{2}ds_{1}>\frac{|\xi_{0}|^{5}}{2}P_{1}(Y) 
\]
by Lemma~\ref{polyn}. Applying now this inequality to (\ref{In1}) we obtain the 
following estimate for $\xi_{0}^{2}>-2a$:
\begin{equation}\label{G1}
H(\xi)\geq H(\xi_{0})
+\frac{|\xi_{0}|^{5}}{2}P_{1}(Y)  
-\int_{\xi}^{\xi_{0}} \int_{s_{1}}^{\xi_{0}} \int_{s_{2}}^{\xi_{0}}
\frac{ds_{3}}{(H(s_{3}))^{3}}ds_{2}ds_{1}  \,, \quad\xi\leq\xi_{0}\,. 
\end{equation}

Now, we can use a Gronwall type of argument to prove that if $c_{1}$ 
satisfies (\ref{bigness2}) then (\ref{G1}) implies 
\begin{equation}\label{G2}
(H(\xi))^{3}\geq\frac{c_{1}}{2( |\xi_{0}|^{2}+1+|a|)  }
\quad \mbox{for} \quad \xi\leq\xi_{0} \,.
\end{equation}
We observe that (\ref{G2}) holds by hypothesis and that it also holds for 
$\xi$ close to $\xi_{0}$ by continuity. Then, as long as (\ref{G2}) is 
satisfied, (\ref{G1}) implies that 
\begin{equation}\label{G3}%
H(\xi)\geq\left(  \frac{c_{1}}{|\xi_{0}|^{2}+1+|a|}\right)^{\frac{1}{3}}
+\frac{|\xi_{0}|^{5}}{2}P_{1}(Y)
-\frac{2(|\xi_{0}|^{2}+1+|a|) }{c_{1}}|\xi_{0}|^{3}P_{2}(Y) \,. 
\end{equation}

We can apply Lemma~\ref{polyn}, and this implies that the last term in 
(\ref{G3}) can be estimated by the previous one for any $\xi<\xi_{0}$ if 
$c_{1}$, $a$ and $\xi_{0}$ satisfy 
\begin{equation}\label{cond:c1:2}
c_{1}>8(|\xi_{0}|^{2}+1+|a|)/|\xi_{0}|^{2}
\end{equation}
and (\ref{G2}) follows. Let us then prove (\ref{cond:c1:2}).

If $|\xi_{0}| \geq 1$, (\ref{bigness2}) implies 
(\ref{cond:c1:2}). If $|\xi_{0}| < 1$ we consider two further cases. 
For $| \xi-\xi_{0}|\leq 2$ we obtain that the last two terms in (\ref{G3}) can 
be bounded from below by $-8(2+|a|)/(3c_{1})$. But this quantity can be absorbed by 
the first term if $c_{1}>(40)^{\frac{3}{4}}(2+|a|)$,
which is satisfied if (\ref{bigness2}) is satisfied. Therefore 
(\ref{G2}) holds for this range of values.

On the other hand, if $|\xi- \xi_{0}| >2$, then the second term of 
(\ref{G3}) can be estimated from below by $(24/15)|\xi- \xi_{0}|^{3}$, 
while the last term in 
(\ref{G3}) can be estimated by $-(2+|a|)|\xi-\xi_{0}|^{3}/(3c_{1})$. 
Then,  we can absorb the last term in (\ref{G3}) into the second one 
if $c_{1}>5(2+|a|)/24$ which is guaranteed by (\ref{bigness2}).

Thus, the inequality (\ref{G3}) holds for arbitrary values of $\xi\leq\xi_{0}$, 
and this implies (\ref{B1:H}) by taking the limit $\xi\to-\infty$.
\end{proof}

We end this section by giving necessary conditions for (\ref{B2:H}) to hold
\begin{proposition}\label{Lext} Let us assume that there exist positive 
constants $c_{2}$ and $c_{3}$, depending on $a$, and some $\xi_{0}\in\mathbb{R}$ with 
\begin{equation}\label{smallness}
\frac{c_{2}^{\frac{1}{3}}}{c_{3}} (|\xi_{0}|^{2}+1+|a|)^{\frac{4}{3}}< \frac{1}{10}\,
\end{equation}
such that a solution of (\ref{S2E3}) satisfies 
$0<((\xi_{0})^{2}+1+|a|)(H(\xi_{0}))^{3}\leq c_{2}$, 
$(|\xi_{0}|+1+|a|)^{\frac{5}{3}}dH(\xi_{0})/d\xi>c_{3}$ and 
$d^{2}H(\xi_{0})/d\xi^{2}<0$. Then there exists 
$\xi_{\ast}\in(-\infty,\xi_{0})$ such that (\ref{B2:H}) holds.
\end{proposition}

\begin{proof}[Proof]
Suppose that $c_{2}$ is sufficiently small. Then, as long as 
$0<((\xi)^{2}+1+|a|)(H(\xi))^{3}\leq2 c_{2}$ we obtain from (\ref{S2E3}) that 
$d^{3}H(\xi)/d\xi^{3}>0$. Integrating this equation over $(\xi,\xi_{0})$ once we obtain that,
 as long as $(\xi^{2}+1+|a|)(H(\xi))^{3}\leq 2 c_{2}$ is satisfied for
$\xi<\xi_{0}$, then $d^{2}H(\xi)/d\xi^{2}<0$ and, integrating a second time, also that  
$dH(\xi)/d\xi>c_{3}(|\xi_{0}|^{2}+1+|a|)^{-\frac{5}{3}}$. 
Then this concavity implies that $H(\xi)$ vanishes at some $\xi=\xi^*$. 
But a third integration implies that
\[
H(\xi)  \leq  H(\xi_{0})-\frac{dH(\xi_{0})}{d\xi} (\xi_{0}-\xi)\leq 
\left( 2^{\frac{1}{3}}  -     \frac{c_{3}  (\xi_{0}-\xi)    }{ c_{2}^{\frac{1}{3}}(|\xi_{0}|^{2}+1+|a|)^{\frac{4}{3}}} \right) \left( \frac{c_{2}}{|\xi_{0}|^{2}+1+|a|}  \right)^{\frac{1}{3}}  
\]
thus $\xi^*\geq\xi_{0}-\frac{c_{2}^{\frac{1}{3}}}{c_{3}}(|\xi_{0}|^{2}+1+|a|)^{\frac{4}{3}}$.
Finally the condition (\ref{smallness}) implies that we can replace $\xi_{0}$ by $\xi\in (\xi^*,\xi_{0})$, thus 
$(\xi^{2}+1+|a|)(H(\xi))^{3}\leq 2 c_{2}$ follows in this interval and 
the result follows by a classical continuation argument.
\end{proof}

\section{Shooting argument}\label{sec:4}
In this section we apply a standard shooting argument to prove the existence of solutions 
of (\ref{compact1})-(\ref{compact4}) such that (\ref{B4}) holds, and such that $\Phi$ remains positive 
and bounded for all $\tau\in \mathbb{R}$. Specifically, the main result of this section is:

\begin{proposition}\label{existence:1} There exists a solution of 
(\ref{compact1})-(\ref{compact4}) $(\Phi(\tau),W(\tau),\Psi(\tau),\theta(\tau))$ defined for
all $\tau\in(-\infty,\infty)$ such that $\lim_{\tau\to\infty}(\Phi
(\tau),W(\tau),\Psi(\tau),\theta(\tau))=(1,0,0,\pi/2)$ and satisfying $\Phi(\tau)>0$ for all 
$\tau\in \mathbb{R}$, (\ref{B4}) and 
\begin{equation}\label{F4E3}
\lim\inf_{\tau\to-\infty}\Phi(\tau) <\infty\,.
\end{equation}
\end{proposition}

The proof of Proposition~\ref{existence:1} is divided in several steps. 
First we prove that points placed in the curve 
$\mathcal{V}_{+}\cap\mathbb{R}^{3}\times\{\theta=\frac{\pi}{2}-\varepsilon\}$ 
with $\varepsilon>0$ sufficiently small, yield
solutions of the equation (\ref{S2E3}) satisfying the hypotheses 
of Proposition~\ref{Linf} if $\nu>0$ and those of Proposition~\ref{Lext} if $\nu<0$:

\begin{lemma}\label{start:shooting} Let $\delta_{0}$ and $\Lambda(\nu,\sigma)$ be as in
Lemma~\ref{trivia}. Then there exist $\nu_{0}>0$ and $\varepsilon>0$, such
that for $w=(\nu,\sigma)\in\Pi$ with $\nu_{0}\leq\nu\leq \delta_{0}/4$
and $\sigma=-\varepsilon$ the trajectory associated to (\ref{compact1})-(\ref{compact4}) 
starting at $\Lambda(\nu,\sigma)$ satisfies (\ref{B1}).
Moreover, if $-\delta_{0}/4\leq\nu\leq-\nu_{0}$ and $\sigma=-\varepsilon$ the
corresponding trajectory of (\ref{compact1})-(\ref{compact4}) satisfies (\ref{B2}).
\end{lemma}

\begin{proof}[Proof]
The dynamics induced by the system (\ref{compact1})-(\ref{compact4}) on the
invariant subspace $\mathbb{R}^{3}\times\{\theta=\pi/2\}$ have been summarised
in Proposition~\ref{Tapas1} and Lemma~\ref{Tapas2}. In particular, 
the trajectory starting at $\Lambda(\nu,0)$ with $\nu>0$ sufficiently small
satisfies (\ref{P1:prime}) and, as it can be easily deduced, also that 
\[
\lim_{\tau\to-\infty}\frac{W(\tau)}{\tau^{2}}=-\frac{1}{2}\quad \mbox{and}\quad \lim_{\tau\to-\infty}\frac{\Psi(\tau)}{\tau}=-1\,.
\]

Then, classical continuous dependence results for ODEs imply that for any
$\rho_{0}>0$ arbitrarily small and $\nu_{0}>0$ small enough there exists
$\varepsilon$ sufficiently small such that, for $\nu_{0}\leq\nu\leq\delta_{0}/4$ 
the trajectory starting at $\Lambda(\nu,-\varepsilon)$ at $\tau=0$ satisfies:%
\begin{equation}\label{G4}%
\left|  \Phi(\tau_{0})+\frac{\tau_{0}^{3}}{6}\right|  \leq\rho_{0}|\tau_{0}|^{3} \,,
\quad\left|  W(\tau_{0})+\frac{\tau_{0}^{2}}{2}\right|  
\leq\rho_{0}\left|  \tau_{0}\right|  ^{2}\,,
\quad\left|  \Psi(\tau_{0}) + \tau_{0}\right|  
\leq\rho_{0}|\tau_{0}| 
\end{equation}
for some $\tau_{0}<0$. Using (\ref{G4}) and (\ref{S2E8a})-(\ref{2nd:dH}) to
get $H$, $dH/d\xi$ and $d^{2}H/d\xi^{2}$ at the value $\xi_{0}$ (given by (\ref{S2E8})),
 we obtain
\[
H(\xi_{0}) \geq c_{1}|\xi_{0}|^{-\frac{2}{3}} \,, \quad\frac{dH(\xi_{0})}{d\xi} <0
\,,\quad\frac{d^{2}H(\xi_{0})}{d\xi^{2}} >0
\]
where $c_{1}>0$ can be made arbitrarily large choosing $\varepsilon$ sufficiently
small and $|\tau_{0}|$ sufficiently large to guarantee that (\ref{bigness2}) is satisfied. 
Then we apply Proposition~\ref{Linf} to obtain (\ref{B1:H}) and hence (\ref{B1}) follows.

On the other hand the trajectories starting at $\Lambda(\tilde{\nu},0)$ with $\tilde{\nu}<0$ satisfy 
$\lim_{\tau\to\tau_{\ast}^{+}}\Phi(\tau)=0$, for some $\tau_{\ast}^{+}>-\infty$.
Moreover, (\ref{P2:prime}) is satisfied, as well as 
\[
\lim_{\tau\to\tau_{\ast}^{+}}(\tau-\tau_{\ast})^{\frac{1}{4}}W(\tau)
=\frac{3}{4}\left(\frac{64}{15}\right)^{\frac{1}{4}} 
\quad \mbox{and} \quad
\lim_{\tau\to\tau_{\ast}^{+}}(\tau-\tau_{\ast})^{\frac{5}{4}}\Psi(\tau)=
 -\frac{3}{16}\left(  \frac{64}{15}\right)^{\frac{1}{4}}\,.
\]

Suppose now that $-\delta_{0}/4\leq\zeta\leq-\nu_{0}$, $\sigma
=-\varepsilon$. Assuming again that $\varepsilon$ is sufficiently small we
obtain that the numbers
\[
\frac{\Phi(\tau_{0}) }{(\tau_{0}-\tau_{\ast})^{\frac{3}{4}}} -\left(
\frac{64}{15}\right)  ^{\frac{1}{4}}\,, \quad(\tau_{0}-\tau_{\ast})^{\frac
{1}{4}}W(\tau_{0}) -\frac{3}{4}\left(  \frac{64}{15}\right)  ^{\frac{1}{4}}\,,
\quad(\tau_{0}-\tau_{\ast})^{\frac{5}{4}}\Psi(\tau_{0}) +\frac{3}{16}\left(
\frac{64}{15}\right)  ^{\frac{1}{4}}
\]
can be made arbitrarily small for $\tau_{0}$ close to $\tau_{\ast}$, 
$\tau_{0}>\tau_{\ast}$. We can use this approximation to obtain that
\[
H(\xi_{0}) \leq c_{2}| \xi_{0}|^{-\frac{2}{3}} \,,\quad \frac{dH(\xi_{0})}{d\xi}>c_{3} \,,
\quad \frac{d^{2}H(\xi_{0})}{d\xi^{2}} <0
\]
where $c_{2}\propto(\tau_{0}-\tau_{\ast})^{1/4}$ and $c_{3}\propto(\tau_{0}-\tau_{\ast})^{-1/4}$, 
thus they can be chosen to satisfy (\ref{smallness}) by taking $\varepsilon>0$ sufficiently 
small. We can now apply Proposition~\ref{Lext} to conclude the proof of the result.
\end{proof}

Next we prove that if for every compact set $K\subset
(-\infty,\infty)$ we have that\ $\lim\inf_{\tau\to(\tau_{\ast})^{+}}\Phi(\tau)
=0$ for some $\tau_{\ast}>-\infty$, then $\lim_{\tau\to(\tau_{\ast})^{+}}%
\Phi(\tau)=0$. Therefore, we will be in the situation stated in Lemma
\ref{contzero} and it will be possible to prove continuity of this behaviour
for small changes of the initial values. 

\begin{lemma}\label{L1} 
Let $\Phi(\tau)$ be a solution of (\ref{S2E9})-(\ref{S2E9bis}) defined in some interval
$(\tau_{\ast},\tau^{\ast})$ with $\tau_{\ast}>-\infty$, $\tau^{\ast}\leq\infty$,
$\Phi(\tau)>0$ for $\tau>\tau_{\ast}$. Then 
\begin{equation}\label{F1E1:2}
\lim\inf_{\tau\to(\tau_{\ast})^{+}}\Phi(\tau) =0\,.
\end{equation}
implies (\ref{B2}) for this value $\tau=\tau_{\ast}$.
\end{lemma}
\begin{proof}[Proof]
It is easier to work with the original equation (\ref{S2E3}), observe that then, 
(\ref{F1E1:2}) is equivalent to 
\begin{equation}\label{F1E1}
\lim\inf_{\xi\to(\xi_{\ast})^{+}}H(\xi) =0 
\end{equation}
for $\xi_{\ast}$ given by $\tau_{\ast}=\int_{0}^{\xi_{\ast}}(\eta^{2} +1)^{\frac{4}{9}}d\eta$. 
Let us then prove that (\ref{F1E1}) implies (\ref{B2:H}), and therefore (\ref{B2}) will follow.

We argue by contradiction. We then assume that (\ref{F1E1}) is satisfied, but
(\ref{B2:H}) does not hold, this means that also
\begin{equation}\label{F1E3}
\lim\sup_{\xi\to(\xi_{\ast})^{+}}H(\xi)>0 \,.
\end{equation}
On the one hand (\ref{F1E1}) gives the existence a decreasing sequence $\{\bar{\xi}_{n}\}$ such that $\bar{\xi}_{n}\to\xi_{\ast}$ as $n\to\infty$, $H(\bar{\xi}_{n+1})<H(\bar{\xi
}_{n})$ and $\lim_{n\to\infty}H(\bar{\xi}_{n}) =0$. And (\ref{F1E3}) implies the existence 
of a sequence with elements $\tilde{\xi}_{n}\in(\xi_{\ast},\bar{\xi}_{n+1})$ 
such that $H(\tilde{\xi}_{n})=H(\bar{\xi}_{n})$. Then there exists another sequence $\{\xi_{n}\}$ 
with $\xi_{n}\in(\tilde{\xi}_{n},\bar{\xi}_{n})$ and $\lim_{n\to\infty}\xi_{n}=\xi_{\ast}$, where local minima are attained, i.e. satisfying
\begin{equation}\label{F3E4:I}
H(\xi_{n}) =\min_{\xi\in(\tilde{\xi}_{n},\hat{\xi}_{n})}H(\xi)\,, \quad
\lim_{n\to\infty}H(\xi_{n}) =0 \,,\quad \frac{dH(\xi_{n})}{d\xi} =0 \,, \quad
\frac{d^{2}H(\xi_{n})}{d\xi^{2}} \geq0\,.
\end{equation}
Let also $\{\hat{\xi}_{n}\}$ be the sequence where local maxima 
are attained, such that $\xi_{n+1}<\hat{\xi}_{n}<\xi_{n}$ and satisfying
\begin{equation}\label{F3E4}%
 H(\hat{\xi}_{n}) =\max_{\xi\in(\xi_{n+1},\xi_{n})}H(\xi)\,, \quad
\lim\sup_{n\to\infty}H(\hat{\xi}_{n}) >0\,, \quad
\frac{dH}{d\xi}(\hat{\xi}_{n})=0\,,\quad\frac{d^{2}H}{d\xi^{2}}(\hat{\xi}_{n})
\leq 0\,.
\end{equation}

Let us now show that
\begin{equation}\label{F3E5}
\lim_{n\to\infty} \left(\left|  \frac{d^{2}H(\hat{\xi}_{n})}{d\xi^{2}} \right|  \frac{1}{H(\hat{\xi}_{n}) }\right) >0\,. 
\end{equation}
Indeed, from (\ref{S2E3}) we obtain that $d^{3}H/d\xi^{3}\geq-C_{1}$ if $\xi
\in\left[  \xi_{n+1},\xi_{n}\right]$ and integrating this inequality, we also obtain
\[
H(\xi_{n}) \geq H(\hat{\xi}_{n}) +\frac{d^{2}H(\hat{\xi}_{n})}{d\xi^{2}} 
\frac{(\xi_{n}-\hat{\xi}_{n})^{2}}{2} -C_{1}\frac{( \xi_{n}-\hat{\xi}_{n})^{3}}{6}\,.
\]
Now, if (\ref{F3E5}) fails, it follows that $H(\xi_{n}) >\frac{H(\hat{\xi}_{n})}{2}$ 
for some subsequence, and this contradicts (\ref{F3E4:I}) and (\ref{F3E4}). Thus
(\ref{F3E5}) holds.

We now claim that (\ref{F3E5}) implies that $H(\xi)$ vanishes for some $\xi
\in[\xi_{n+1},\hat{\xi}_{n}]$. Indeed, since $d^{3}H/d\xi^{3}\geq-C_{1}$ we
then have that, for $n$ large enough,
\[
\frac{d^{2}H(\xi)}{d\xi^{2}} \leq\frac{d^{2}H(\hat{\xi}_{n})}{d\xi^{2}} +
C_{1}(\hat{\xi}_{n}-\xi_{n+1}) \leq\frac{1}{2}\frac{d^{2}H(\hat{\xi}_{n}%
)}{d\xi^{2}} <0\,,\quad\xi\in[\xi_{n+1},\hat{\xi}_{n}]\,.
\]
This implies that $dH(\xi)/d\xi>0$ for $n$ large enough with $\xi\in[\xi
_{n+1},\hat{\xi}_{n}]$, but this contradicts the definition of $\xi_{n+1}$,
and so for $n$ large enough there is a first value $\xi\in[\xi_{n+1},\hat{\xi
}_{n}]$ such that $H(\xi)=0$, i.e. (\ref{B2:H}) holds.
\end{proof}

\begin{remark}\label{positivity} Notice that a classical Gronwall argument implies that any
solution $\Phi(\tau)$ can be extended for arbitrary negative values of $\tau$ as long as 
$\Phi(\tau)$ remains away from zero. More precisely, if 
$\lim\inf_{\tau\to\tau_{0}^{+}}\Phi(\tau)>0$ for any $\tau_{0}\geq\tau_{\ast}>-\infty$, 
it is possible to extend $\Phi(\tau)$ as a solution of
(\ref{compact1})-(\ref{compact4}) for times $\tau>\tau_{\ast}-\delta$ and some
$\delta>0$. Reciprocally, the maximal existence time, due to Lemma \ref{L1},
is finite and it is given by $\tau_{\ast}>-\infty$ if $\lim\inf_{\tau\to
\tau_{\ast}}\Phi(\tau)=0$.
\end{remark}

We are now ready to prove Proposition~\ref{existence:1}.

\begin{proof}[Proof of Proposition~\ref{existence:1}]
We consider the one-dimensional family of solutions of 
(\ref{compact1})-(\ref{compact4}) obtained choosing in Lemma~\ref{start:shooting} the
parameters $\sigma=\varepsilon>0$ with $\varepsilon>0$ small enough and 
$\nu\in(-\delta_{0}/4,\delta_{0}/4)$. We define as
$\mathcal{U}_{+}$ the set of values of $\nu$ such that the corresponding
solution of (\ref{compact1})-(\ref{compact4}) satisfies (\ref{B1}). 
On the other hand, we denote by $\mathcal{U}_{-}$ the
set of values of $\nu$ such that the corresponding solution of 
(\ref{compact1})-(\ref{compact4}) satisfy (\ref{B2}) 
for some $\tau_{\ast}>-\infty$. Due to Lemma~\ref{start:shooting} we have that
$\mathcal{U}_{+}\neq\varnothing$ and $\mathcal{U}_{-}\neq\varnothing$. 
Moreover, by definition $\mathcal{U}_{+}\cap\mathcal{U}_{-}=\varnothing$. Due
to lemmas \ref{continfinity} and \ref{contzero} we have that the sets
$\mathcal{U}_{+}$ and $\mathcal{U}_{-}$ are open sets. Therefore, there exists
$\bar{\nu}\in\left(  -\frac{\delta_{0}}{4},\frac{\delta_{0}}{4}\right)  $ such
that $\bar{\nu}\notin\mathcal{U}_{+}\cup\mathcal{U}_{-}$.

The corresponding solution of (\ref{compact1})-(\ref{compact4}) associated to
the parameter $\bar{\nu}$ has the property that, for any $\tau_{0}>-\infty$ we
have $\inf_{\tau\in(\tau_{0},\infty) }\Phi(\tau) \geq C_{-}(\tau_{0}) >0$,
since otherwise $\bar{\nu}\in\mathcal{U}_{-}$ due to
Lemma~\ref{L1}. This implies also that $\sup_{\tau\in(\tau_{0},\infty) }%
\Phi(\tau) \leq C_{+}(\tau_{0}) <\infty$ because the right-hand side of
(\ref{compact1})-(\ref{compact4}) is bounded in compact sets if 
$\Phi(\tau)\geq C_{-}(\tau_{0})$. Therefore, this solution is globally defined for 
$\tau\in(-\infty,\infty)$. Moreover, (\ref{F4E3}) holds, since otherwise 
$\bar{\nu}\in\mathcal{U}_{+}$ and the result follows.
\end{proof}

\section{Oscillatory solutions}\label{sec:dynamics}
We recall that the final aim is to prove that 
the solutions found in Proposition~\ref{existence:1} have no alternative but to approach 
the invariant subspace $\theta=-\pi/2$ as $\tau\to-\infty$ and they remain uniformly bounded 
while $\Phi$ stays positive (see (\ref{B3})-(\ref{B4})). The argument is by contradiction 
and in this section we prove the following lemma that is the first step in the argument.

\begin{proposition}\label{char:max} 
Suppose that $(\Phi(\tau),W(\tau),\Psi(\tau),\theta(\tau))$ 
is a solution of (\ref{compact1})-(\ref{compact4})
defined for all $\tau\in(-\infty,\infty)$ and satisfying
\begin{equation}\label{F5E1}%
\lim\inf_{\tau\to-\infty}\Phi(\tau) <\infty\,
\end{equation}
and that
\begin{equation}\label{G1E6}%
\lim\sup_{\tau\to-\infty}\left(  \Phi(\tau) +\left|  \frac{d\Phi(\tau)}{d\tau}\right|  
+ \left|  \frac{d^{2}\Phi(\tau)}{d\tau^{2}} \right|  \right)
=\infty\,.
\end{equation}
Then, there exists a decreasing sequence $\{\tau_{n}^{\ast}\}$ with 
$\lim_{n\to\infty}\tau_{n}^{\ast}=-\infty$ and a sequence $\{\varepsilon_{n}\}$ with 
$\varepsilon_{n}>0$ small enough such that 
\begin{equation}\label{def:max}
\Phi(\tau_{n}^{\ast}) =\max_{\tau\in[\tau_{n}^{\ast}-\varepsilon_{n},\tau_{n}^{\ast}+\varepsilon_{n}]}\Phi(\tau) 
\end{equation}
and that
\begin{equation}\label{G5E3}
\lim\sup_{n\to\infty}\Phi( \tau_{n}^{\ast})=\infty\,. 
\end{equation}
\end{proposition}

Before we prove this result prove three auxiliary lemmas. 
First we show that there exists a decreasing 
sequence of local minima attained at certain $\tau=\tau_{n}$ with 
$\lim_{n\to\infty}\tau_{n}=- \infty$.
\begin{lemma}\label{L2} 
Let $(  \Phi(\tau),W(\tau),\Psi(\tau),\theta(\tau))$ 
satisfy the assumptions of Proposition~\ref{char:max}. Then, there exists a decreasing sequence 
$\{\tau_{n}\}$ such that $\lim_{n\to\infty}\tau_{n}=-\infty$ and that
\begin{equation}\label{G1E7}
\lim_{n\to\infty}\Phi(\tau_{n}) \leq 1 
\quad\mbox{and}\quad\quad\frac{d\Phi(\tau_{n})}{d\tau}=0\,, 
\quad\frac{d^{2}\Phi(\tau_{n})}{d\tau^{2}} \geq 0
\quad\mbox{for all}\quad n\,.
\end{equation}
\end{lemma}

\begin{proof}[Proof]
First, we claim that
\begin{equation}\label{G1E3}
\lim\inf_{\tau\to-\infty}\Phi(\tau) \leq 1\,,
\quad \lim\sup_{\tau\to-\infty} \Phi(\tau) \geq 1\,. 
\end{equation}
Indeed, suppose first that $\lim\inf_{\tau\to-\infty}\Phi(\tau) >1$. Then,
there exists $\varepsilon_{0}>0$ and $\tau_{0}$ sufficiently negative, such
that $\Phi(\tau)\geq 1 + 2\varepsilon_{0}$ for $\tau\leq\tau_{0}$. Then,
(\ref{S2E8a}) and (\ref{S2E8}) imply $H(\xi)\geq(1+\varepsilon_{0})|\xi|^{-2/3}$ 
for $\xi\leq\xi_{0}$, where $\xi_{0}$ is related to $\tau_{0}$ 
by means of (\ref{S2E8}). This inequality applied to (\ref{S2E3}) gives
\begin{equation}
\frac{d^{3}H}{d\xi^{3}}=\frac{1}{H^{3}}-(\xi^{2}+a) \leq-\varepsilon_{1}\xi^{2} 
\label{G1E4}%
\end{equation}
for $\xi\leq\xi_{0}$ and some $\varepsilon_{1}>0$ (by taking a more negative
$\tau_{0}$ if necessary). Integrating (\ref{G1E4}) three times for $\xi\leq\xi_{0}$ 
gives $H(\xi)>\varepsilon_{1}|\xi_{0}|^{5}P_{1}(1+\xi/|\xi_{0}|)$ (where $P_{1}$ is as 
in Lemma~\ref{polyn}). Thus $\lim_{\xi\to-\infty}H(\xi)=\infty$, but this contradicts (\ref{F5E1}).

We now prove the second inequality in (\ref{G1E3}). Suppose on the contrary that  
$\lim\sup_{\tau\to-\infty}\Phi(\tau)<1$, then 
$\Phi(\tau) \leq1 -\varepsilon_{0}$ for some $\varepsilon_{0}>0$ and for
$\tau\leq\tau_{0}$ if $\tau_{0}<0$ with $|\tau_{0}|$ large enough. 
The transformation (\ref{S2E8})-(\ref{S2E8a}), with the obvious correspondence
in notation, implies that $H(\xi) \leq(1-\varepsilon_{0})|\xi|^{-\frac{2}{3}}$ for
$\xi\leq\xi_{0}$, whence (\ref{S2E3}) yields 
\begin{equation}\label{G1E5}
\frac{d^{3}H}{d\xi^{3}}=\frac{1}{H^{3}}-(\xi^{2}+a) \geq\varepsilon_{2}\xi^{2}
\end{equation}
for some $\varepsilon_{2}>0$ and $\xi\leq\xi_{0}$. Integrating (\ref{G1E5})
for $\xi\leq\xi_{0}$ we obtain $H(\xi)<-\varepsilon_{1}|\xi_{0}|^{5}P_{1}(1+\xi/|\xi_{0}|)$. 
This implies the existence of a $\xi_{\ast}>-\infty$ such that $\lim_{\xi\to(\xi^{\ast})^{+}}H(\xi)=0$. 
This contradicts the assumption that $\xi_{\ast}=-\infty$ (see Remark~\ref{positivity}) and
(\ref{G1E3}) follows.

Suppose now that
\begin{equation}\label{lim:1}
\lim\inf_{\tau\to-\infty}\Phi(\tau)= \lim\sup_{\tau\to-\infty}\Phi(\tau)
=\lim_{\tau\to-\infty}\Phi(\tau)=1\,. 
\end{equation}
We define a sequence of functions $\{\Phi_{n}(s)\}$ with $s\in[-1,0]$ as
follows. For every $n\in\mathbb{N}$ the variable $\xi_{n}(s)$ is given by, 
cf. (\ref{S2E8a}),
\begin{equation}\label{G1E9}%
s = -\int_{\xi_{n}(s) }^{-n}( 1+\eta^{2})^{\frac{4}{9}}d\eta\,,\quad
n\in\mathbb{N} \,, 
\end{equation}
then, each $\Phi_{n}(s)$ is defined by, cf. (\ref{S2E8}),
\[
\Phi_{n}(s) =\left(1+|\xi_{n}(s)|^{2}\right)^{\frac{1}{3}}H(\xi_{n}(s))
\,,\quad s\leq0\,,\quad n\in\mathbb{N} \,. 
\]
We observe that then
\[
\Phi_{n}(s) =\Phi(s-s_{n}) \,,\quad\mbox{where} \quad s_{n}
= \int_{-n}^{0}(1+\eta^{2})^{\frac{4}{9}}d\eta\,,
\]
where $\Phi$ solves (\ref{S2E9})-(\ref{S2E9bis}). Also, for every $s\in[-1,0]$ 
the corresponding sequence $\tau_{n}=s-s_{n}$ converges to $-\infty$ as $n\to\infty$, 
since $\lim_{n\to\infty}(s_{n}) =\infty$. On the other hand, the functions $\Phi_{n}(s)$ 
solve (cf. (\ref{S2E9})-(\ref{S2E9bis}))
\begin{equation} \label{G2E1}
\frac{d^{3}\Phi_{n}}{ds^{3}}=\frac{1}{\Phi_{n}^{3}}
-\frac{\left(  \xi_{n}^{2}+a\right)  }{\xi_{n}^{2}+1} - F_{n}(s) \,,
\quad n\in\mathbb{N}\,,\quad s\in[-1,0] 
\end{equation}
where $F_{n}(s)$ is given by the expression of $F$ in (\ref{S2E9bis}) 
with $\Phi$ and $\xi$ replaced by $\Phi_{n}$ and $\xi_{n}$, respectively.

Consider now the result of integrating (\ref{G2E1}):
\begin{equation}\label{phi:n:inte}%
\begin{array}[c]{ll}
\Phi_{n}(s) & = \displaystyle{ \Phi_{n}(0) 
+\frac{d\Phi_{n}(0)}{ds}s +\frac{d^{2}\Phi_{n}(0)}{ds^{2}}\frac{s^{2}}{2} }\\
& \\
& \displaystyle{ 
-\int_{s}^{0}\int_{s_{1}}^{0}\int_{s_{2}}^{0}
\left[  
\frac{1}{(\Phi_{n}(s_{3}))^{3}}-\frac{|\xi_{n}(s_{3})|^{2}+a}{|\xi_{n}(s_{3})|^{2}+1}
\right] ds_{3}ds_{2}ds_{1}
 }\\
& \\
& \displaystyle{ -\int_{s}^{0}\int_{s_{1}}^{0}\int_{s_{2}}^{0}F_{n}(s_{3})ds_{3}ds_{2}ds_{1} }\,.
\end{array}
\end{equation}
We now pass to the limit in the integral terms. 
Observe that the assumption (\ref{lim:1}) implies that $\lim_{n\to\infty}\Phi_{n}(s)=1$
uniformly on $[-1,0]$. Moreover, (\ref{G1E9}) yields $\lim_{n\to\infty}\xi
_{n}(s) =-\infty$ uniformly on $[-1,0]$.
The first term in (\ref{phi:n:inte}) can be seen
to converge to zero using the limit properties of $\Phi_{n}(s)$ and $\xi_{n}(s)$. 
In the last term we integrate by parts where necessary in order to
get integrands with $\Phi_{n}(s)$ as a coefficient (this gives boundary terms
with a double or single integral, but these are estimated similarly, because
$s\in[-1,0]$). The resulting integrands have $\Phi_{n}(s)$ multiplied by a
function of $\xi_{n}(s)$ and its derivatives, which can be computed using
(\ref{G1E9}): $d\xi_{n}(s)/ds = (|\xi_{n}(2)|^{2}+1)^{-4/9}$ and 
$d^{2}\xi_{n}(s)/ds^{2} = -4(|\xi_{n}(2)|^{2}+1)^{-17/9}/9$.
Then, one can conclude that the limit of the last term in (\ref{phi:n:inte})
 tends also to zero as $n\to\infty$, and we are left with 
\[
\lim_{n\to\infty}\left(\left|\frac{d\Phi_{n}(0)}{ds}\right|
+ \left|\frac{d^{2}\Phi_{n}(0)}{ds^{2}}\right|\right)=0\,,
\] 
but this contradicts (\ref{G1E6} and (\ref{lim:1}) cannot hold. Then 
$\lim\inf_{\tau\to-\infty}\Phi(\tau)<\lim\sup_{\tau\to-\infty}\Phi(\tau)$. 

We can now construct a sequence that satisfies (\ref{G1E7}). 
We first take the following quantity
\[
\alpha:=\frac{1}{2} \left(  \lim\sup_{\tau\to-\infty}\Phi(\tau) 
+ \lim\inf_{\tau\to-\infty}\Phi(\tau)\right)
\]
(that might be infinite if $\lim\sup_{\tau\to-\infty}\Phi(\tau)=\infty$). Due to 
the continuity of $\Phi$, there exist decreasing sequences $\{\tilde{\tau}_{n}\}$ 
and $\{ \hat{\tau}_{n}\} $ such that $\lim_{n\to\infty}\tilde{\tau}_{n}
=\lim_{n\to\infty}\hat{\tau}_{n}=-\infty$, $\tilde{\tau}_{n}<\hat{\tau}_{n}$ and that 
$\max_{\tau\in(\tilde{\tau}_{n},\hat{\tau}_{n}) }\Phi(\tau)\leq\alpha$. We define another 
sequence $\{\tau_{n}\}$ by
\[
\Phi(\tau_{n}) =\min_{\tau\in[ \tilde{\tau}_{n},\hat{\tau}_{n}] }\Phi(\tau)\,,
\]
and this one satisfies (\ref{G1E7}).
\end{proof}

We continue with another consequence of assuming (\ref{G1E6}), namely,
\begin{lemma}\label{L3} 
Suppose that $(\Phi(\tau),W(\tau),\Psi(\tau),\theta(\tau))$ satisfies the assumptions 
of Proposition~\ref{char:max}. Then, at least one of the following identities holds:
\[
\lim\inf_{\tau\to-\infty}\Phi(\tau) =0\,, \quad
\lim\sup_{\tau\to-\infty}\Phi( \tau) =\infty\,. 
\]
\end{lemma}

\begin{proof}[Proof]
We argue by contradiction. Suppose that
\[
\lim\inf_{\tau\to-\infty}\Phi(\tau) >0
\quad\mbox{and} \quad
\lim\sup_{\tau\to-\infty}\Phi(\tau) <\infty\,,
\]
then, there exists a $C_{0}>0$ and a $\tau_{0}$ sufficiently negative such that
\begin{equation}
\frac{1}{C_{0}}\leq\Phi(\tau) \leq C_{0}\,,\quad\tau\leq\tau_{0}\,. \label{G2E2a}%
\end{equation}
Then, (\ref{G1E6}) implies the existence of a sequence $\{ \hat{\tau}_{n}\}$
with $\hat{\tau}_{n}\to-\infty$ and such that
\begin{equation}
\lim_{n\to\infty} \left(  \left|  \frac{d\Phi( \hat{\tau}_{n})}{d\tau}\right|  
+\left|  \frac{d^{2}\Phi(\hat{\tau}_{n})}{d\tau^{2}}\right|  \right)
=\infty\,. \label{G2E2}%
\end{equation}
We now define a sequence of functions $\Phi_{n}(z)$ by means of $
\Phi_{n}(z) =\Phi(z+\hat{\tau}_{n})$ 
and observe that they solve (\ref{S2E9})-(\ref{S2E9bis}) with the obvious 
changes in notation and with $\xi_{n}(z)$ defined by
\[
z+\hat{\tau}_{n}=-\int_{\xi_{n}(z)}^{0}(1+\eta^{2})^{\frac{4}{9}}d\eta\,,
\quad n\in\mathbb{N}\,.
\]
They also satisfy, due to (\ref{G2E2}), that
\[
\lim_{n_{\infty}} \left(  \left|  \frac{d\Phi_{n}( 0) }{dz}\right|  
+\left|\frac{d^{2}\Phi_{n}( 0) }{dz^{2}}\right|  \right)  = \infty\,.
\]
This allows us to introduce, for every $n$, the length scale
\[
\gamma_{n}=\left(  \left|  \frac{d\Phi_{n}(0)}{dz} \right|  
+ \sqrt{\left|\frac{d^{2}\Phi_{n}(0)}{dz^{2}}\right|  } \right)^{-1}\,,
\]
which clearly satisfies $\lim_{n\to\infty} \gamma_{n}=0$. We now set $z=\gamma_{n}\bar{z}$, 
$\bar{\xi_{n}}(\bar{z})=\xi_{n}(\gamma_{n}\bar{z})$, 
$\bar{\Phi}_{n}(\bar{z})=\Phi_{n}(\gamma_{n}\bar{z})$ and 
$\bar{F}_{n}(\bar{z})=F_{n}(\gamma_{n}\bar{z})$, to obtain that $\bar{\Phi}_{n}$ 
satisfies
\[
\frac{d^{3}\bar{\Phi}_{n}}{d\bar{z}^{3}}
= (\gamma_{n})^{3} \left[  \frac{1}{\bar{\Phi}_{n}^{3}} 
- \frac{\bar{\xi}_{n}^{2} + a }{\bar{\xi}_{n}^{2} +1}\right]  
+ (\gamma_{n})^{3} \bar{F}_{n}(\bar{z}) \,,\quad n\in\mathbb{N}\,.
\]
It is clear that there exists a $\bar{C}>0$ such that for $n$ large enough
\begin{equation}\label{init:bdd}
\left|  \left(  \bar{\Phi}_{n}(0),\frac{d\bar{\Phi}_{n}(0)}{d\bar{z}}, 
\frac{d^{2}\bar{\Phi}_{n}(0)}{d\bar{z}^{2}}\right)  \right|
\leq \bar{C} \quad \mbox{and}\quad
\left|  \left(  \frac{d\bar{\Phi}_{n}(0)}{d\bar{z}}, 
\frac{d^{2}\bar{\Phi}_{n}(0)}{d\bar{z}^{2}} \right)  \right|  \geq\frac{1}{\bar{C}}\,.
\end{equation}
We can now use classical continuous dependence results for ODEs. Let 
$\bar{\Phi}_{\infty}$ denote a function that solves the limiting problem 
$d^{3}\bar{\Phi}_{\infty}/d\bar{z}^{3}=0$ with initial conditions close to 
$(\bar{\Phi}_{n}(0),d\bar{\Phi}_{n}(0)/d\bar{z}, d^{2}\bar{\Phi}_{n}(0)/d\bar{z}^{2})$ 
for $n$ large enough, and thus also satisfying (\ref{init:bdd}) (with $n=\infty$),
then
\[
\bar{\Phi}_{n}(\bar{z}) \to\bar{\Phi}_{\infty}(\bar{z}) \quad\mbox{as}\ n\to\infty\,,
\]
as well as its derivatives, uniformly in compact sets of $\bar{z}$.

Since $\bar{\Phi}_{\infty}$ is a polynomial at most of second order that is not
identically constant, there exist values of $\bar{s}$ such that
either $\bar{\Phi}_{\infty}(\bar{z}) =0$ or $\bar{\Phi}_{\infty}(\bar{z})\geq 2M$. 
But this contradicts (\ref{G2E2a}).
\end{proof}

The third auxiliary lemma is the following:
\begin{lemma}\label{L4} 
Suppose that $(\Phi(\tau),W(\tau),\Psi(\tau),\theta(\tau))$ satisfies is the assumptions 
of Proposition~\ref{char:max}. Then,
\[
\lim\sup_{\tau\to-\infty}\Phi(\tau) =\infty\,.
\]
\end{lemma}

\begin{proof}[Proof]
Suppose that
\begin{equation}\label{G6E2}
\lim\sup_{\tau\to-\infty}\Phi(\tau) <\infty\,. 
\end{equation}
Then, due to Lemma~\ref{L3} we have $\lim\inf_{\tau\to-\infty}\Phi( \tau) =0$, and 
there exists a sequence of points $\{\tau_{n}\}$ such that (\ref{G1E7}) and that 
$\lim_{n\to\infty}\Phi(\tau_{n}) =0$.
For every $n$ we introduce the changes of variables
\begin{equation}\label{G4E2a}
\varepsilon_{n}:=\Phi(\tau_{n}) \,,\quad s= \varepsilon_{n}^{-\frac{4}{3}}( \tau-\tau_{n})
\quad\mbox{and}\quad\Phi(\tau) =\varepsilon_{n} \varphi(s) \,. 
\end{equation}
Observe that if $\xi_{n}$ is related to $\tau_{n}$ by means of (\ref{S2E8}). 
In particular, this reflects that when $|\tau-\tau_{n}|$ remains bounded as 
$n \to\infty$, then $|\xi|\to\infty$ as $n\to\infty$, a fact that we shall apply below.

We write (\ref{S2E9}) in the new variables, then $\varphi$ solves
\begin{equation}\label{G4E2}
\frac{d^{3}\varphi}{ds^{3}} 
= \frac{1}{\varphi^{3}} - \varepsilon_{n}^{3} \frac{\xi^{2}+a}{\xi^{2}+1}  
- \varepsilon_{n}^{\frac{4}{3}} F_{n}(s)
\end{equation}
with (cf. \ref{S2E9bis})
\begin{eqnarray*}
F_{n}(s) &=& \varepsilon_{n}^{\frac{8}{3}} \frac{16}{3}\frac{\xi}{(\xi^{2}+1)^{\frac{10}{3}}} 
\left[1-\frac{14}{9}\frac{\xi^{2}}{ \xi^{2}+1} \right]  \varphi(s) \\
& + & \varepsilon_{n}^{\frac{4}{3}}  \frac{1}{(\xi^{2}+1)^{\frac{17}{9}}}\left[  \frac{208}{81}\frac{\xi^{2}}{\xi^{2}+1}
-\frac{10}{9}\right]  \frac{d\varphi(s)}{ds} 
+   \frac{2}{3}\frac{\xi}{(\xi^{2}+1)^{\frac{13}{9}}} \frac{d^{2}\varphi(s)}{ds^{2}}\,, 
\end{eqnarray*}
and subject to
\begin{equation}\label{G4E3}
\varphi(0) =1\,,\quad\frac{d\varphi(0)}{ds}=0\,,\quad\frac{d^{2}\varphi(0) }{ds^{2}}\geq0\,.
\end{equation}
The coefficients involving $\xi$ are functions of $s$; 
$\xi(\tau)=\xi(\tau_{n}+\varepsilon_{n}^{\frac{4}{3}}s)$. 
Then, since $\varepsilon_{n}\to0$, we can use classical continuous dependence results
for ODEs to approximate the solutions of (\ref{G4E2})-(\ref{G4E3}) by the solutions 
$\bar{\varphi}$ of the limiting problem (with $\varepsilon_{n}=0$)
\begin{align*}
\frac{d^{3}\bar{\varphi}}{ds^{3}} - \frac{1}{\bar{\varphi}^{3}}  &  =0\\
\bar{\varphi}(0)  &  =1\,,\quad\frac{d\bar{\varphi}(0) }{ds}=0 \,,
\quad\frac{d^{2}\bar{\varphi}(0) }{ds^{2}}=\frac{d^{2}\varphi(0)}{ds^{2}} \geq0\nonumber
\end{align*}
on compact sets of $s$. Clearly, $d^{2}\bar{\varphi}(s)/ds^{2}$ is 
increasing for $s>0$, and there exists a $\delta>0$ such that 
$d^{2}\bar{\varphi}(s)/ds^{2} \geq \delta>0$ for all $s\geq 1$. Thus, 
$d\bar{\varphi}(s)/ds$ increases at least linearly for $s$ large enough and
one can find a value $s_{0}>0$ such that $d\bar{\varphi}(s_{0})/ds \geq 2$. 
Continuous dependence results then imply that $d\varphi(s_{0})/ds \geq 1$ if
$n$ is sufficiently large (see (\ref{G4E2})). We now return to the original
variables and get estimates on $\Phi(\tau)$ for every $n$ in an interval
around the local minimum. Using (\ref{G4E2a}), (\ref{G6E2}) and the notation
\[
\bar{\tau}_{n}=\tau_{n}+\varepsilon_{n}^{\frac{4}{3}} s_{0}\,,
\]
we obtain that 
\[
\lim_{n\to\infty} \Phi(\bar{\tau}_{n})=0\,,\quad
\lim_{n\to\infty} \frac{d\Phi(\bar{\tau}_{n})}{d\tau}
=\infty\,,\quad\lim_{n\to\infty}\frac{d^{2}\Phi(\bar{\tau}_{n})}{d\tau^{2}}
=\infty\,.
\]
For $n$ large enough and $\tau$ such that $|\tau-\bar{\tau}_{n}|\leq1$, there
exists a positive sequence $\{B_{n}\}$ with $\lim_{n\to\infty}B_{n}=\infty$ 
and such that
\[
B_{n}<\frac{1}{\Phi(\tau)^{3}}-\frac{\xi(\tau)^{2}+a}{\xi(\tau)^{2}+1}%
<B_{n+1}\,.
\]
Let us denote by
\[
G_{n}(\tau)= \int_{\bar{\tau}_{n}}^{\tau}\int_{\bar{\tau}_{n}}^{\tau_{1}}%
\int_{\bar{\tau}_{n}}^{\tau_{2}} F(\tau_{3}) d\tau_{3}d\tau_{2}d\tau_{1}\,,
\]
where the function $F(\tau)$ is given by (\ref{S2E9bis}), and let also
\begin{align*}
p_{n}(\tau)= \Phi(\bar{\tau}_{n})+\frac{d\Phi(\bar{\tau}_{n})}{d\tau}  (\tau-\bar{\tau}_{n}) 
+\frac{d^{2}\Phi(\bar{\tau}_{n})}{d\tau^{2}}\frac{(\tau-\bar{\tau}_{n})^{2}}{2}\,.
\end{align*}
Then, integrating (\ref{S2E9}) we can write for $\tau\in[\bar{\tau}_{n}-1,\bar{\tau}_{n}+1]$ that
\begin{equation}\label{G5E1} 
\frac{B_{n}}{3}\frac{d^{k}(\tau-\bar{\tau}_{n})^{3}}{d\tau^{k}}<\frac{d^{k}(\Phi-p_{n}-G_{n})(\tau)}{d\tau^{k}} <  \frac{B_{n+1}}{3}\frac{d^{k}(\tau-\bar{\tau}_{n})^{3}}{d\tau^{k}} \,, \quad k=0,1,2\,.
\end{equation}
Then, taking $n$ large enough and combining the inequalities (\ref{G5E1}) we obtain that
\begin{align}
|\Phi(\tau)|  < \Phi(\bar{\tau}_{n})   +   2 \frac{d\Phi(\bar{\tau}_{n})}{d\tau}     |\tau-\bar{\tau}_{n} | 
+ 2 \frac{d^{2}\Phi(\bar{\tau}_{n})}{d\tau^{2}} |\tau-\bar{\tau}_{n} |^{2} + |G_{n}(\tau)|\,,\nonumber\\
\left|  \frac{d\Phi(\tau)}{d\tau} \right|  < 2 \frac{d\Phi(\bar{\tau}_{n})}{d\tau} + 2 \frac{d^{2}\Phi(\bar{\tau}_{n})}{d\tau^{2}}
|\tau-\bar{\tau}_{n} | + \left|  \frac{dG_{n}}{d\tau} \right|  \,, \label{G5E1tris} \\
\left|  \frac{d^{2}\Phi(\tau)}{d\tau^{2}}\right|  < 2\frac{d^{2}\Phi(\bar{\tau}_{n})}{d\tau^{2}}
 + \left|\frac{d^{2}G_{n}}{d\tau^{2}} \right|  \,, \nonumber%
\end{align}
and that
\begin{equation}\label{low:bdd}
\Phi(\tau)\geq \Phi(\bar{\tau}_{n}) + \frac{d\Phi(\bar{\tau}_{n})}{d\tau}  (\tau-\bar{\tau}_{n}) 
+ \frac{d^{2}\Phi(\bar{\tau}_{n})}{d\tau^{2}}    \frac{(\tau-\bar{\tau}_{n})^{2}}{2}  
+ \bar{B}_{n}(\tau-\bar{\tau}_{n})^{3} + G_{n}(\tau) 
\end{equation}
for some sequence $\bar{B}_{n}>0$ with $\lim_{n\to\infty} \bar{B}_{n}=\infty$ and for
all $\tau\in[\bar{\tau}_{n}-1, \bar{\tau}_{n}+1]$. We now observe that $F$ has the form 
$F(\tau) = f_{1}(\tau) \Phi(\tau) +f_{2}(\tau) d \Phi(\tau)/d\tau+f_{3}(\tau) d^{2}\Phi(\tau)/d\tau^{2}$ 
where the functions $f_{1}(\tau)$, $f_{2}(\tau)$ and $f_{3}(\tau)$ converge
uniformly to zero on sets $\left|  \tau-\bar{\tau}_{n}\right|\leq1$ for every $n$ 
(cf. (\ref{S2E8}) and (\ref{S2E9bis})). This allows to get estimates on the integral terms 
(those involving $G_{n}$) as follows:
\[
|G_{n}(\tau)|  \,, \ \left|  \frac{dG_{n}(\tau)}{d\tau}\right|\,,\ 
\left|  \frac{d^{2} G_{n}(\tau)}{d\tau^{2}} \right|  \leq \tilde{\varepsilon}_{n}
\left(  \sup_{|\tau-\bar{\tau}_{n}|\leq1}\Phi(\tau) 
+ \sup_{|\tau-\bar{\tau}_{n}|\leq1}\left|  \frac{d\Phi(\tau)}{d\tau}\right|  
+ \sup_{|\tau-\bar{\tau}_{n}|\leq1}\left|  \frac{d^{2}\Phi(\tau)}{d\tau^{2}}\right|  \right)
\]
for all $n$ with $\tau$ such that $|\tau-\bar{\tau}_{n}|\leq 1$ and where 
$\tilde{\varepsilon}_{n}\to 0^{+}$. Applying this to (\ref{G5E1tris}) and to (\ref{low:bdd}) we obtain, taking $n$ sufficiently large, that  
\[
\Phi(\tau) \geq\frac{d\Phi(\bar{\tau}_{n})}{d\tau}\frac{(\tau-\bar{\tau}_{n})}{2}
+\frac{d^{2}\Phi(\bar{\tau}_{n})}{d\tau^{2}}\frac{(\tau-\bar{\tau}_{n})^{2}}{4}+\bar{B}_{n}(\tau-\tau_{n})^{3}%
\]
for $\left|  \tau-\bar{\tau}_{n}\right|  \leq1$. 
Choosing, say $\tau-\bar{\tau}_{n}=1$, we obtain that $\lim_{n\to\infty}\Phi(\bar{\tau}_{n}+1)=\infty$, 
but this contradicts (\ref{G6E2}), whence the lemma follows.
\end{proof}

We are now ready to prove Proposition~\ref{char:max}.

\begin{proof}[Proof of Proposition~\ref{char:max}]
The assumptions imply that we can use the statements of lemmas~\ref{L2} and \ref{L4}. 
In particular, by Rolle's theorem, we can guarantee the existence of local maxima in each interval
$(\tau_{n+1},\tau_{n})$ where $\{\tau_{n}\}$ is the sequence of minima defined
in Lemma~\ref{L2}. We then observe that the regularity of a solution $\Phi$ of
(\ref{S2E9}) guarantees that the points at which $\Phi$ attains local maxima
or minima are isolated. Otherwise $\Phi$ would take constant values on closed
intervals, but constants are not solutions of (\ref{S2E9}). Hence, we can
define the sequence such that (\ref{def:max}) holds. On the other hand,
Lemma~\ref{L4} implies (\ref{G5E3}).
\end{proof}

\section{Properties of oscillatory solutions}\label{sec:proof:osci} 
In this section we proof the following proposition:
\begin{proposition}\label{max:increase} 
Let the assumptions of Proposition~\ref{char:max} hold and let $\{\tau_{n}^{\ast}\}$ be the 
sequence found in this proposition. Then there exists $n_{0}\in\mathbb{N}$ large such that for all $n>n_{0}$
\begin{equation}\label{G5E4}
\Phi( \tau_{n-1}^{\ast}) > \Phi( \tau_{n}^{\ast}) \,.
\end{equation}
\end{proposition}
We observe that this result is in contradiction with (\ref{G5E3}). 
We now define a sequence $\{\tau_{n}^{min}\}$ as follows
\begin{equation}\label{def:min}
\Phi(\tau_{n}^{min})=\min_{\tau\in(\tau_{n}^{\ast},\tau_{n-1}^{\ast})} \Phi(\tau)\,, 
\quad \frac{d\Phi(\tau_{n}^{min})}{d\tau}=0\,, 
\quad  \frac{d^{2}\Phi(\tau_{n}^{min})}{d\tau^{2}}\geq 0\,, 
\end{equation}
i.e. $\Phi$ reaches the minimum in the interval $(\tau_{n}^{\ast},\tau_{n-1}^{\ast})$ at 
$\tau=\tau_{n}^{min}$. Then as part of the construction necessary to prove (\ref{G5E4}) it will 
follow that: 
\begin{proposition}\label{min:decrease} 
Let the assumptions of Proposition~\ref{char:max} hold. Then the sequence $\{\tau_{n}^{min}\}$ 
given in (\ref{def:min}) is well-defined, $\lim_{n\to\infty}\tau_{n}^{min}=-\infty$ and there exists 
$n_{0}$ such that for all $n>n_{0}$ then
\begin{equation}\label{G5E5}
\Phi(\tau_{n-1}^{min})<\Phi(\tau_{n}^{min}) \,.
\end{equation}
\end{proposition}

The proofs of these propositions are divided in several steps that we outline below for clarity. 
We first identify a two-parameter family of polynomials that approximate $H$ 
near a large maximum of $\Phi$. Most part of this analysis is done in Appendix~\ref{sec:polynomials}, 
where we identify and give some properties of the polynomials that solve (\ref{ecuacionminima1}) 
(with the reverse sign). Then for each $n$ and around $\tau_{n}^{\ast}$ we identify a length scale 
that transform these polynomials into polynomials of order one. We then translate 
the properties found into the rescaled polynomials. We also rescale accordingly 
the function $H$ near each $\xi_{n}^{\ast}$ defined by (\ref{S2E8a}) for 
$\tau=\tau_{n}^{\ast}$ and give the approximating lemma that in particular implies 
that $H$ will get close to $0$ in a linear decreasing way. Next we adapt 
the {\it matching lemma}, Lemma~\ref{MatchingLemma} in Appendix~\ref{sec:control}, 
that gives the behaviour of the solutions in the inner regions near each 
$\xi_{n}^{min}=\xi(\tau_{n}^{min})$. From this result we can conclude that the 
approximating polynomial in the outer region must have a double zero in order to match. 
This, in particular, reduces the class of approximating polynomials to a one-parameter family. 
We finally derive an iterative relation between the elements of the sequence $\{\Phi(\tau_{n}^{\ast})\}$ 
if $n$ is large enough that implies Proposition~\ref{max:increase}, as well as 
a relation for the elements of $\{\Phi(\tau_{n}^{min})\}$ that implies Proposition~\ref{min:decrease}.

\subsection{The outer variables and the auxiliary polynomials}\label{sec:aux}
Given the sequence of $\{\tau_{n}^{\ast}\}$ found in Proposition~\ref{char:max}, 
see (\ref{def:max}), and the sequence of local minima $\{\tau_{n}^{min}\}$ defined in (\ref{def:min}), 
we define the sequences $\{\xi_{n}^{\ast}\}$, $\{\xi_{n}^{min}\}$, $\{M_{n}\}$ and $\{\beta_{n}\}$ 
by means of (see (\ref{H6}) and (\ref{S4E3})):
\begin{equation}\label{S5E2}
\tau_{n}^{\ast}=\int_{0}^{\xi_{n}^{\ast}}(\eta^{2}+1)^{\frac{4}{9}}d\eta\,, \quad 
\tau_{n}^{min}=\int_{0}^{\xi_{n}^{min}}(\eta^{2}+1)^{\frac{4}{9}}d\eta\,,
\end{equation}%
\begin{equation}\label{S5E3}%
M_{n}=\frac{1}{|\xi_{n}^{\ast}|^{\frac{17}{3}}}\,\Phi(\tau_{n}^{\ast})\,,\quad
\beta_{n}=\frac{((\xi_{n}^{\ast})^{2}+1)^{\frac{8}{9}}}{|\xi_{n}^{\ast}|^{\frac{11}{3}}}\,
\frac{d^{2}\Phi(\tau_{n}^{\ast})}{d\tau^{2}}\,.
\end{equation}%
Observe that the definition of $\tau_{n}^{\ast}$ implies that $\beta_{n}<0$, that 
$\lim_{n\to\infty}\xi_{n}^{\ast}=-\infty$ and that $\lim_{n\to\infty}\xi_{n}^{min}=-\infty$.

Observe that $M_{n}$ is the value of the maximum of $\Phi$ at each $\tau_{n}^{\ast}$
rescaled appropriately with the position of the maximum in the variable $\xi$, $\xi_{n}^{\ast}$ 
(this scaling near a maximum resembles that $\Phi\sim|\xi|^{\frac{2}{3}}H$ there and that $H$ is 
approximated by a fifth order polynomial, cf. (\ref{P:Z:eq}) and (\ref{S4E3}) in 
Appendix~\ref{sec:polynomials}). The definition of $\beta_{n}$ results from similar 
considerations, but is a parameter that captures the value of the second derivative 
of $\Phi$ at $\tau_{n}^{\ast}$.

Following Appendix~\ref{sec:polynomials}, associated to every $\tau_{n}^{\ast}$ we construct the 
two-parameter family of polynomials
\[
P(Z_{n};M_{n},\beta_{n}) \quad\mbox{with}\quad Z_{n}=\frac{\xi}{\xi_{n}^{\ast}}
\]
that solve (\ref{P:Z:eq}) with $Z$ replaced by $Z_{n}$, and that are given by
(\ref{S4E4}). We recall that they satisfy that $P(1;M_{n},\beta_{n})=M_{n}$. We shall see 
later that these polynomials are close to $-(\xi_{n}^{\ast})^{5}H(\xi)$ for $\xi$ close to 
$\xi_{n}^{\ast}$ in an interval contained in $[\xi_{n+1}^{min},\xi_{n}^{min}]$. 
Thus we have to consider $Z_{n}$ in some interval containing $Z_{n}=1$ and where $P$ stays positive. 
Moreover, since $\xi_{n}^{\ast}\to -\infty$ the approximation will be applicable for $Z_{n}>0$ only. 
In this regard, for each $n$, we have derived a number of properties that are outlined in 
lemmas~\ref{Pol2}, \ref{Pol4}, \ref{Pol5} and \ref{new:pol}. 
These give, in particular, that the largest root of 
$P(Z_{n};M_{n},\beta_{n})$ in $Z_{n}<1$ is attained at a value $Z_{n}=Z_{0}(M_{n},\beta_{n})$ 
for every $M_{n}>0$ and every $\beta_{n}<0$. It is also shown that for every $M_{n}>0$ 
there exists a unique value $\beta_{n}=\beta_{\ast}(M_{n})$ 
such that $P(Z_{n};M_{n},\beta_{\ast}(M_{n}))$ has a double zero 
at some $Z_{n}=Z_{\ast}(M_{n})>1$. 

In these lemmas the asymptotic behaviour as $M_{n}\to 0$ and as $M_{n}\to\infty$
of $\beta_{\ast}(M_{n})$, $Z_{\ast}(M_{n})$ and $Z_{0}(M_{n})=Z_{0}(M_{n},\beta_{\ast}(M_{n}))$ 
is also given. But, as we shall see later and assume now, the sequence $\{ M_{n}\}$ 
is bounded.

Taking these considerations into account, we now introduce a rescaling of 
$P(Z_{n};M_{n},\beta_{n})$ for every $n$ in order to have values of 
order one in the relevant range of parameters. Namely, we set
\begin{equation} \label{S6E1}
\bar{P}(\zeta_{n};M_{n},\beta_{n})
=\frac{P(Z_{n};M_{n},\beta_{n})}{M_{n}}\,,\quad
\zeta_{n}=-\frac{1}{M_{n}^{\frac{1}{3}}}\left(  \frac{\xi}{\xi_{n}^{\ast}}-1\right)\,. 
\end{equation}
Observe that now the variable $\zeta_{n}$ is meaningful in an interval around $\zeta_{n}=0$ and with 
$\zeta_{n}<1/M_{n}^{1/3}$. We note that $M_{n}^{\frac{1}{3}}$ is a characteristic length scale which 
measures the distance between $Z_{n}=1$ and $Z_{n}=Z_{\ast}(M_{n})$, relevant if $M_{n}$ is very small, 
see Lemma~\ref{Pol4}.  We also observe that the polynomials $\bar{P}(\zeta_{n};M_{n},\beta_{n})$ 
are explicitly given by
\begin{equation}\label{S5E6}
\bar{P}(\zeta_{n};M_{n},\beta_{n}) = -\frac{\zeta_{n}^{3}}{60}
\left( M_{n}^{\frac{2}{3}}\zeta_{n}^{2} -5 M_{n}^{\frac{1}{3}}\zeta_{n}+10 \right)
+ \left( \frac{5}{9}M_{n}^{\frac{2}{3}}\zeta_{n}^{2}   + \frac{2}{3} M_{n}^{\frac{1}{3}}\zeta_{n} +1\right)
+ \frac{\beta_{n}}{M_{n}^{\frac{1}{3}}}\frac{\zeta_{n}^{2}}{2}\,.
\end{equation}
It is natural to define the following values of $\zeta_{n}$:
\begin{equation} \label{min:zeta:root}
\zeta_{0}(M_{n},\beta_{n})= -\frac{Z_{0}(M_{n},\beta_{n})-1}{M_{n}^{\frac{1}{3}}}\,,
\end{equation}
(see Lemma~\ref{Pol5}, (\ref{max:Zroot})), thus clearly 
$\zeta_{0}(M_{n},\beta_{n})=\min\{\zeta_{n}>0:\ \bar{P}(\zeta_{n};M_{n},\beta_{n})=0\}$.
And if $\beta_{n}=\beta_{\ast}(M_{n})$ we define also
\begin{equation}\label{S5E8}
\zeta_{\ast}(M_{n})=-\frac{Z_{\ast}(M_{n})-1}{M_{n}^{\frac{1}{3}}}\,,
\end{equation}
therefore $\bar{P}$ has a double zero at this value (see Lemma~\ref{Pol2}). 
When $\beta_{n}=\beta_{\ast}(M_{n})$ and for simplicity of notation, we shall write:%
\[
\bar{P}(\zeta_{n};M_{n})=\bar{P}(\zeta_{n};M_{n},\beta_{\ast}(M_{n})  )
\]
and
\[
\zeta_{0}(M_{n})=\zeta_{0}(M_{n},\beta_{\ast}(M_{n}))\,.
\]

The following result follows easily:
\begin{lemma}\label{AppPolynomials}
For every $n$ and $M_{n}$ the polynomial $\bar{P}(\zeta_{n};M_{n},\beta_{n})$ 
solves
\begin{equation}\label{S6E6}
\frac{d^{3}\bar{P}(\zeta;M_{n},\beta_{n})}{d\zeta^{3}}+(1-M_{n}^{\frac{1}{3}}\zeta)^{2}=0
\end{equation}
with initial conditions
\begin{equation}\label{S6E62}
\bar{P}(0;M_{n},\beta_{n})=1\,,\quad\frac{d\bar{P}(0;M_{n},\beta_{n})}{d\zeta}
=\frac{2M_{n}^{\frac{1}{3}}}{3}\,,
\quad\frac{d^{2}\bar{P}(0;M_{n},\beta_{n})}{d\zeta^{2}}
=\frac{\beta_{n}}{M_{n}^{\frac{1}{3}}}+\frac{10M_{n}^{\frac{2}{3}}}{9}\,.
\end{equation}
\end{lemma}

We now reformulate the results of Appendix~\ref{sec:polynomials} for these
approximating functions:
\begin{lemma}\label{scaled:pols} 
Let $\bar{P}(\zeta_{n};M_{n},\beta_{n})$ be given by (\ref{S5E6}). They satisfy that 
if $\beta_{n}>\beta_{\ast}(M_{n})$, then $\bar{P}(\zeta_{n};M_{n},\beta_{n})>0$ 
in $\zeta_{n}<0$ and if $\beta_{n}<\beta_{\ast}(M_{n})$ then 
there are two zeros of $\bar{P}(\zeta_{n};M_{n}\beta_{n})$ in $\zeta_{n}<0$. 
The derivative of $\bar{P}(\zeta_{n};M_{n},\beta_{n})$ with 
respect to $\zeta_{n}$ is positive at the largest root in $\zeta_{n}<0$. If 
$\beta_{n}=\beta_{\ast}(M_{n})$ there is only one double zero in $\zeta_{n}<0$ and is placed 
at $\zeta_{n}=\zeta_{\ast}(M_{n})$. Moreover,
\begin{equation}\label{zeta:gamma:M0}
\zeta_{\ast}(M_{n})\sim -12^{\frac{1}{3}}(1+M_{n}^{\frac{1}{3}})\,,\quad
\beta_{\ast}(M_{n})\sim -\left(\frac{3M_{n}}{2}\right)^{\frac{1}{3}}
\quad\mbox{as}\quad M_{n} \to 0
\end{equation}%
\begin{equation}\label{zeta:gamma:M0II}
\frac{\partial^{2}\bar{P}(\zeta_{\ast}(M_{n});M_{n})}{\partial\zeta_{n}^{2}}
\sim\left(  \frac{3}{2}\right)^{\frac{1}{3}}\quad \mbox{as}\quad M_{n}\to 0\,.
\end{equation}
Also, the value (\ref{min:zeta:root}) is well-defined for every $n$ and 
$M_{n}$, and if $\beta_{n}=\beta_{\ast}(M_{n})$, then
\begin{equation}\label{zeta:plus:M0}
\zeta_{0}(M_{n})\sim\left(\frac{3}{2}\right)^{\frac{1}{3}}\,,\quad
\frac{\partial\bar{P}(\zeta_{0}(M_{n});M_{n})}{\partial\zeta_{n}}
\sim-\left(  \frac{3}{2}\right)^{\frac{5}{3}}\quad\mbox{as}\quad M_{n}\to 0\,.
\end{equation}
\end{lemma}

Finally, we have also that 
\begin{lemma}\label{new:pol:rescale}
The value $\zeta_{0}(M_{n},\beta_{n})$ is the unique root of $\bar{P}(\zeta_{n};M_{n},\beta_{n})$ in 
$\zeta_{n} \geq 0$. Moreover,  if $\zeta_{0}(M_{n},\beta_{n}) \leq 2/M_{n}^{\frac{1}{3}}$, 
there exists a positive constant $c_{0}$ independent of $M_{n}$ and $\beta_{n}$ such that
\[
\frac{d\bar{P}(\zeta_{0}(M_{n},\beta_{n});M_{n},\beta_{n})  }{d\zeta}\leq - c_{0}\max\{1,M_{n}^{\frac{1}{3}}\}
 \,.
\]
\end{lemma}

\begin{remark} We point out that the case $M_{0} \to 0$ corresponds to the the approximating polynomials obtained for (\ref{ODEtapas}) in \cite{CV}. The asymptotics (\ref{zeta:plus:M0}) are in agreement with this observation. 
\end{remark}
\subsection{The sequence of rescaled $H(\xi)$ near each $\xi_{n}^{\ast}$}\label{sec:rescale:H}
In order to compare $H$ with a polynomial $\bar{P}(\zeta_{n};M_{n},\beta_{n})$ we need to apply
the scaling (\ref{S6E1}) to $H$ around $\xi=\xi_{n}^{\ast}$. We then obtain:
\begin{lemma}\label{H:rescaling} 
Let us assume that $\Phi$ satisfies the assumptions of Proposition~\ref{char:max}, 
so that the sequence (\ref{def:max}) is well-defined. 
Let $H(\xi)$ be the solution of (\ref{S2E3}) related to $\Phi$ by means of (\ref{S2E8a}) and 
(\ref{S2E8}). Let the sequence of functions $\{\mathcal{H}_{n}(\zeta_{n})\,, \ \zeta_{n}\in\mathbb{R}\}$ 
be defined by
\begin{equation}\label{S6E2}
H(\xi)=|\xi_{n}^{\ast}|^{5}M_{n}\mathcal{H}_{n}(\zeta_{n})\,,\quad
 \zeta_{n}=-\frac{1}{M_{n}^{\frac{1}{3}}}\left(  \frac{\xi}{\xi_{n}^{\ast}}-1\right)  \,,
\end{equation}
then, for each $n$, $\mathcal{H}_{n}$ solves
\begin{equation} \label{S6E2a}
\frac{d^{3}\mathcal{H}_{n}}{d\zeta_{n}^{3}}+R_{n}(\zeta_{n};M_{n})
=\frac{\delta_{n}}{(\mathcal{H}_{n})^{3}}\,, 
\end{equation}
where
\begin{equation} \label{parameters}
\delta_{n}=\frac{1}{|\xi_{n}^{\ast}|^{17}M_{n}^{3}}\,,\quad
R_{n}(\zeta_{n};M_{n})
=\left(  (1-M_{n}^{\frac{1}{3}}\zeta_{n})^{2}+\frac{a}{(\xi_{n}^{\ast})^{2}}\right)  \,,
\end{equation}
with initial conditions
\begin{equation}\label{S6E4}
\mathcal{H}_{n}(0)
=\frac{|\xi_{n}^{\ast}|^{\frac{2}{3}}}{(1+|\xi_{n}^{\ast}|^{2})^{\frac{1}{3}}}\,,
\quad\frac{d\mathcal{H}_{n}(0)}{d\zeta_{n}}
=\frac{2}{3}M_{n}^{\frac{1}{3}}\frac{|\xi_{n}^{\ast}|^{\frac{8}{3}}}{(|\xi_{n}^{\ast}|^{2}+1)^{\frac{4}{3}}}
\end{equation}
and
\begin{equation} \label{S6E5}%
\frac{d^{2}\mathcal{H}_{n}(0)}{d\zeta_{n}^{2}}
=\frac{\beta_{n}}{M_{n}^{\frac{1}{3}}}\frac{|\xi_{n}^{\ast}|^{\frac{2}{3}}}{(|\xi_{n}^{\ast}|^{2}+1)^{\frac{1}{3}}}
+\frac{2M_{n}^{\frac{2}{3}}}{3}
\left(  \frac{\left(  \frac{5}{3}|\xi_{n}^{\ast}|^{2}-1\right)
|\xi_{n}^{\ast}|^{\frac{8}{3}}}{(|\xi_{n}^{\ast}|^{2}+1)^{\frac{7}{3}}}\right)  \,.
\end{equation}
\end{lemma}

\begin{proof}[Proof]
That each $\mathcal{H}_{n}$ solves (\ref{S6E2a}) follows by changing variables
in (\ref{S2E3}). The initial conditions follow from (\ref{def:max}) and
(\ref{S5E3}).
\end{proof}

The following lemma will be used in the following to approximate the functions
$\mathcal{H}_{n}(\cdot)$ by polynomials $\bar{P}(\cdot;M_{n},\beta_{n})$.

\begin{lemma}\label{ApproximationLemma}
For each $n$ let $\mathcal{H}_{n}(\zeta_{n})$ solve 
(\ref{S6E2a})-(\ref{S6E5}). Then, for every $\varepsilon>0$ there exists $n_{0}\in\mathbb{N}$ such that
for all $n>n_{0}$ the estimates 
\begin{equation}\label{W1E9}
\mathcal{H}_{n}(\zeta_{n})\geq \frac{ \zeta_{0}(M_{n},\beta_{n})-\zeta_{n}}{2\zeta_{0}(M_{n},\beta_{n})}  \,,
\end{equation}
\begin{align*}
\left| \mathcal{H}_{n}(\zeta_{n})-\bar{P}(\zeta_{n};M_{n},\beta_{n})\right|  & \leq
\varepsilon\left| \bar{P}(\zeta_{n};M_{n},\beta_{n})\right| \\
\left|\frac{d\mathcal{H}_{n}(\zeta_{n}) }{d\zeta_{n}}
-\frac{d\bar{P}(\zeta_{n};M_{n},\beta_{n})}{d\zeta_{n}}\right|  & 
\leq\varepsilon\left(\left| \frac{d\bar{P}(\zeta_{n};M_{n},\beta_{n})}{d\zeta_{n}}\right|+1\right)  \\
\left| \frac{d^{2}\mathcal{H}_{n}(\zeta_{n})}{d\zeta_{n}^{2}}
-\frac{d^{2}\bar{P}(\zeta_{n};M_{n},\beta_{n})}{d\zeta_{n}^{2}}\right|  &
\leq\varepsilon\left(\left| \frac{d^{2}\bar{P}(\zeta;M_{n},\beta_{n})}{d\zeta_{n}^{2}}\right| +1\right)
\end{align*}
hold in $\zeta_{n}>0$ and as long as 
\[
\frac{\zeta_{0}(M_{n},\beta_{n})-\zeta_{n}}{2\zeta_{0}(M_{n},\beta_{n})} \geq \varepsilon\,.
\]
\end{lemma}

\begin{proof}[Proof]
The proof of this result is a standard bootstrap argument similar to the ones that has been used 
repeatedly. The idea is that the initial conditions (\ref{S6E4}) and (\ref{S6E5}) tend to the ones 
for $\bar{P}$ as $n\to \infty$, see (\ref{S6E62}). Also the term $\delta_{n}/(\mathcal{H}_{n})^{3}$ in 
(\ref{S6E2a}) is negligible if $n$ is sufficiently large since $\delta_{n}\to 0$ as $n\to\infty$ 
(observe that $\Phi(\tau_{n}^{\ast})=\delta_{n}^{-1}$). On the other hand, the term $R_{n}(\zeta_{n};M_{n})$ 
can be approximated by $(1-M_{n}^{\frac{1}{3}}\zeta_{n})^{2}$ as $n\to\infty$. The resulting limiting 
equation is then (\ref{S6E6}) and the values of $\mathcal{H}_{n}$ and of its derivatives can be 
approximated at $\zeta_{n}=0$ by those of $\bar{P}(\cdot;M_{n},\beta_{n})$ and its derivatives there. 
The difference between $\mathcal{H}_{n}(\zeta_{n})$ and $\bar{P}(\zeta_{n};M_{n},\beta_{n})$ can then be 
approximated arguing as in, for example, Lemma~\ref{continfinity} as well as in Lemma~4.3 of \cite{CV}.
 We observe that, as in \cite{CV}, (\ref{W1E9}) implies upon integration of (\ref{S6E2a}), 
that a condition on $\zeta_{0}(M_{n},\beta_{n})-\zeta_{n}$ of the form 
\[
\frac{\zeta_{0}(M_{n},\beta_{n})-\zeta_{n}}{2\zeta_{0}(M_{n},\beta_{n})} \geq 
\delta_{n} 
\left|\log\left(\frac{\zeta_{0}(M_{n},\beta_{n})-\zeta_{n}}{2\zeta_{0}(M_{n},\beta_{n})} \right) \right| 
+ \frac{a}{|\xi_{n}^{\ast}|^{2}} 
\]
must be satisfied for $n$ large enough. Then for every $\varepsilon$ we can choose $n_{0}$ 
large enough to obtain that the $\varepsilon$ is larger than the solution of 
$\delta_{n_{0}} |\log(x)|+ \zeta_{0}a/|\xi_{n_{0}}^{\ast}|^{2}= x/2$
and that the initial data are close enough to those of $\bar{P}$.
\end{proof}

We now prove that $\beta_{n}\sim\beta_{\ast}(M_{n})$ as $n\to\infty$. The idea is to use the fact that 
the derivative of the approximating polynomial $\bar{P}(\zeta_{n};M_{n},\beta_{n})$ at 
$\zeta_{0}(M_{n},\beta_{n})$ is of order one (and negative) by Lemma~\ref{new:pol:rescale}. 
Then, we can use Lemma~\ref{MatchingLemma} of Appendix~\ref{sec:control} that gives the 
behaviour in the boundary layer where $\mathcal{H}_{n}$ becomes small (near $\xi=\xi_{n+1}^{min})$),
 to conclude that the next polynomial in the outer region is close to one having a double zero 
in the matching region near $\xi=\xi_{n+1}^{min})$. 

\begin{lemma}\label{LemmaDoubleZero}
Suppose that the $\Phi$ and its derivatives satisfy the
assumptions of Proposition~\ref{char:max} and let $\{\tau_{n}^{\ast}\}$ be the sequence found 
there. Let the sequences $\{\xi_{n}^{\ast}\}$, $\{M_{n}\}$ and $\{\beta_{n}\}$ be defined by
means of (\ref{S5E2}) and (\ref{S5E3}), and let the functions
$\beta_{\ast}(M_{n})$ be as in (\ref{S5E8}) and the sequence of functions $\mathcal{H}_{n}$ 
be given by (\ref{S6E2}). 
Then, for any $\varepsilon>0$, there exists a $L=L(\varepsilon)>0$ and a
$n_{0}$ large enough such that if $\Phi(\tau_{n}^{\ast})\geq L$ and $n\geq n_{0}$ 
then 
\begin{equation}\label{J3E7}
\left|\beta_{n}-\beta_{\ast}(M_{n})\right| \leq
\varepsilon\left|\beta_{\ast}(M_{n})\right| \,.
\end{equation}

Also, for all $n>n_{0}$ 
\begin{equation}\label{linear:entry}
\mathcal{H}_{n}(\zeta_{n})\simeq -K_{n}(\zeta_{n} - \zeta_{0}(M_{n}))
\quad\mbox{as}\quad\zeta_{n} \to (\zeta_{0}(M_{n}))^{-}%
\end{equation}
\begin{equation}\label{linear:entry2}
\quad\mbox{with}\quad K_{n}
=-\frac{d\bar{P}(\zeta_{0}(M_{n}))}{d\zeta_{n}}>0
\end{equation}
and
\begin{equation} \label{quadratic:exit}
\mathcal{H}_{n}(\zeta_{n})\simeq \frac{D_{n}}{\delta_{n}}(\zeta_{n}-\zeta_{0}(M_{n}))^{2}
\quad\mbox{as}\quad\zeta_{n} \to (\zeta_{0}(M_{n}))^{+} \,,
\end{equation}
where $D_{n}$ is proportional to $K_{n}^{5}$ by a constant of order one, and 
there exist $\alpha_{1}$ and $\alpha_{2}\in\mathbb{R}$ independent of $M_{n}$ 
such that $0<\alpha_{1}\leq K_{n}\leq\alpha_{2}$.

Moreover, for all $n>n_{0}$ there exists $\varepsilon_{0}>0$ small enough an 
$\xi_{crit,n}\in ( \xi_{n}^{min}-\varepsilon_{0}, \xi_{n}^{min}+\varepsilon_{0})$
such that 
\begin{equation}\label{J1E2}
\zeta_{0}(M_{n})=-\frac{1}{M_{n}^{\frac{1}{3}}}\left(  \frac{\xi_{crit,n}}{\xi_{n}^{\ast}}-1\right)
\quad \mbox{and}\quad \zeta_{\ast}(M_{n-1})=-\frac{1}{M_{n-1}^{\frac{1}{3}}}\left(  \frac
{\xi_{crit,n}}{\xi_{n-1}^{\ast}}-1\right) 
\end{equation}
and
\begin{equation}\label{quadratic:exit3}
\mathcal{H}_{n-1}(\zeta_{n-1})\simeq \Gamma_{n-1}(\zeta_{n-1}-\zeta_{\ast}(M_{n-1}))^{2}
\quad\mbox{as}\quad\zeta_{n-1} \to (\zeta_{\ast}(M_{n-1}))^{+} 
\end{equation}
\begin{equation} \label{quadratic:exit2}
\quad \mbox{with} \quad
\Gamma_{n-1}=\frac{1}{2}\frac{d^{2}\bar{P}(\zeta_{\ast}(M_{n-1});M_{n-1})}{d\zeta_{n-1}^{2}}>0\,.
\end{equation}
\end{lemma}

\begin{proof}[Proof]
Suppose that $n$ is very large. We apply Lemma~\ref{ApproximationLemma} for $n$, thus 
starting $\zeta_{n}=0$ or at $\xi=\xi_{n}^{\ast}$.
It then follows that we can approximate $\mathcal{H}_{n}$ by the polynomial
$\bar{P}(\zeta_{n};M_{n},\beta_{n})$ in intervals of the 
form $\zeta_{n}\in\left[0,\zeta_{0}(M_{n},\beta_{n})-\varepsilon_{1}\right]$ 
with $\varepsilon_{1}>0$ small but fixed and $n$ large enough. This in particular implies 
(\ref{linear:entry}) and (\ref{linear:entry2}). The fact that $K_{n}$ is bounded from above 
and below follows from Lemma~\ref{new:pol:rescale}. Using then Lemma~\ref{MatchingLemma} 
we obtain that $\mathcal{H}_{n}(\zeta)$ can be approximated as a quadratic polynomial for 
$\zeta_{n}=\zeta_{0}(M_{n},\beta_{n})+ \varepsilon_{1}$. This implies 
\[
\left| \mathcal{H}_{n}(\zeta_{n})-\frac{A K_{n}^{5}}{\delta_{n}}
(\zeta_{n}-\zeta_{0}(M_{n}) )^{2} \right|  
\leq\frac{\varepsilon_{2}K^{5}_{n}}{\delta_{n}}(\zeta_{n}-\zeta_{0}(M_{n}))^{2}\,,
\]
for some $\varepsilon_{2}>0$ small enough and $A$ of order one, and thus 
(\ref{quadratic:exit}) follows. We can then replace the variables 
$\zeta_{n}$ by $\zeta_{n-1}$ and $\mathcal{H}_{n}$ by $\mathcal{H}_{n-1}$ using (\ref{S6E2}) and applying again Lemma~\ref{ApproximationLemma} we can then approximate the function $\mathcal{H}_{n-1}(\zeta)$ 
by one polynomial which has a double root at the value of $\zeta_{n-1}$ corresponding to $\zeta_{n}=\zeta_{0}(M_{n},\beta_{n})+\varepsilon_{1}$. This implies (\ref{quadratic:exit}), but also (\ref{J3E7}) follows by the definition of $\beta_{\ast}$, and therefore also 
(\ref{J1E2}), (\ref{quadratic:exit3}) and (\ref{quadratic:exit2}) follow.
\end{proof}

\subsection{Proof of propositions \ref{max:increase} and \ref{min:decrease}}
\label{sec:formal}
In order to prove the propositions we derive information from Lemma~\ref{LemmaDoubleZero} 
in the matching region around $\xi_{n}^{min}$ for $n$ large enough. Let us then assume that $\Phi$ 
satisfies the assumptions of Proposition~\ref{char:max}, so that the sequence (\ref{def:max}) 
is well-defined. Let $H(\xi)$ be defined by means of (\ref{S2E8a}) and (\ref{S2E8}), and satisfies 
(\ref{S2E3}). Let the sequence of functions $\mathcal{H}_{n}$ be defined by (\ref{S6E2}), so that,
by Lemma~\ref{H:rescaling}, each such function satisfies (\ref{S6E2a}) with
initial conditions (\ref{S6E4}) and (\ref{S6E5}).

We further assume in the following that the approximating polynomials have $\beta_{n}=\beta_{\ast}(M_{n})$
thus they are as described in Lemma~\ref{scaled:pols} and we drop the dependency on $\beta_{n}$ 
in the notation. 

Lemma~\ref{LemmaDoubleZero} (\ref{J1E2}) and the definition of the variables $Z_{n}$ give
\begin{equation}\label{J1E1}
Z_{0}(M_{n})  \xi_{n}^{\ast} = Z_{\ast}(M_{n-1})\xi_{n-1}^{\ast} \,.
\end{equation}

\begin{remark}
We can now argue that $\{M_{n}\}$ is a bounded sequence. Indeed, if $M_{n}$ is very 
large, (\ref{Z0:Minfty}) implies that $Z_{0}(M_{n})$ is very negative. This would imply, 
using (\ref{J1E1}) that $\xi_{n-1}^{\ast}>0$ (notice that $Z_{\ast}(M_{n-1})>0$ by definition). 
However, we cannot have $\xi_{n-1}^{\ast}>0$ for large $n$, because $\xi_{n}^{\ast}\to-\infty$. 
\end{remark}

Using now the definition of $\zeta_{n}$ (see (\ref{S6E2})) and (\ref{J1E2}) we can compute 
$(\zeta_{n}-\zeta_{0}(M_{n}))$ and $(\zeta_{n-1}-\zeta_{\ast}(M_{n-1}))$ to get
\begin{equation} \label{J1E3}%
M_{n}^{\frac{1}{3}}\xi_{n}^{\ast}(\zeta_{n}-\zeta_{0}(M_{n}))  
=M_{n-1}^{\frac{1}{3}}\xi_{n-1}^{\ast}(\zeta_{n-1}-\zeta_{\ast}(M_{n-1}))\,. 
\end{equation}
On the other hand, the definition of the sequence $\mathcal{H}_{n}$ (in (\ref{S6E2})) gives
\begin{equation}\label{J1E4}
|\xi_{n}^{\ast}|^{5}M_{n}\mathcal{H}_{n}(\zeta_{n})
=|\xi_{n-1}^{\ast}|^{5}M_{n-1}\mathcal{H}_{n-1}(\zeta_{n-1}) \,.
\end{equation}
We then change variables according to (\ref{J1E3}) and (\ref{J1E4})
 in (\ref{quadratic:exit}) in order to write it in terms of 
the variables $\zeta_{n-1}$ and $\mathcal{H}_{n-1}$. This gives the asymptotic formula, 
for $n$ large enough,
\[
\mathcal{H}_{n-1}(\zeta_{n-1})
\sim\frac{D_{n}}{\delta_{n}}\frac{|\xi_{n}^{\ast}|^{3}M_{n}}{|\xi_{n-1}^{\ast}|^{3}M_{n-1}}
\frac{M_{n-1}^{\frac{2}{3}}}{M_{n}^{\frac{2}{3}}}
(\zeta_{n-1}-\zeta_{\ast}(M_{n-1}))^{2}\,,
\quad\zeta_{n-1}\to (\zeta_{\ast}(M_{n-1}))^{+}\,.
\]
But comparing this to (\ref{quadratic:exit3}) implies that the approximation 
\[
\Gamma_{n-1}M_{n-1}^{\frac{1}{3}}|\xi_{n-1}^{\ast}|^{3}
=D_{n}|\xi_{n}^{\ast}|^{20} M_{n}^{\frac{10}{3}}
\]
is valid for $n$ large enough (here we have also used (\ref{parameters})).

Using the fact that $M_{n}$ is bounded, and also that $D_{n}$ can be estimated
from above and below by a constant independent on $n$, we obtain:%
\[
C_{1}|\xi_{n}^{\ast}|^{20} M_{n}^{\frac{10}{3}}\leq
\Gamma_{n-1}M_{n-1}^{\frac{1}{3}}|\xi_{n-1}^{\ast}|^{3}\leq
C_{2}|\xi_{n}^{\ast}|^{20}M_{n}^{\frac{10}{3}}%
\]
for $0<C_{1}\leq C_{2}$. Using (\ref{S5E3}), then
\[
C_{1}\frac{|\xi_{n}^{\ast}|^{\frac{10}{9}}}{|\xi_{n-1}^{\ast}|^{\frac{10}{9}}}(\Phi(\xi_{n}^{\ast}))^{\frac{10}{3}}
\leq\Gamma_{n-1}(\Phi(\xi_{n-1}^{\ast}))^{\frac{1}{3}}
\leq C_{2}\frac{|\xi_{n}^{\ast}|^{\frac{10}{9}}}{|\xi_{n-1}^{\ast}|^{\frac{10}{9}}}(\Phi(\xi_{n}^{\ast}))^{\frac{10}{3}}\,.
\]

Using now (\ref{J1E1}) as well as the fact that $Z_{0}(M_{n})$ and $Z_{\ast}(M_{n-1})$ 
are bounded from above and below for $M_{n}$ and $M_{n-1}$ bounded (cf. (\ref{S5E4}), 
(\ref{Z0:Mzero})), we obtain, for different $C_{1}$ and $C_{2}$ if necessary, that
\[
C_{1}(\Phi(\xi_{n}^{\ast}))^{\frac{10}{3}}
\leq\Gamma_{n-1}(\Phi(\xi_{n-1}^{\ast}))^{\frac{1}{3}}
\leq C_{2}(\Phi(\xi_{n}^{\ast}))^{\frac{10}{3}}\,.
\]

Using now (\ref{zeta:gamma:M0II}) for $n-1$ we can estimate $\Gamma_{n-1}$ from above
and below by positive constants independent on $n$. Then:%
\begin{equation}\label{J3E8}
\Phi(\xi_{n-1}^{\ast})\simeq C(\Phi(\xi_{n}^{\ast}))^{10}
\end{equation}
for some $C>0$.

We are now in the position of proving Proposition~\ref{max:increase}. 

\begin{proof}[Proof of Proposition~\ref{max:increase}]
Due to Lemma~\ref{LemmaDoubleZero} we can assume that (\ref{J3E7}) 
holds for $n$ large. In a similar fashion as in \cite{CV} we can make rigorous the argument 
outlined above by combining the lemmas~\ref{MatchingLemma} and 
\ref{ApproximationLemma} and prove indeed that 
 (\ref{J3E8}) hold. Since by hypothesis $\Phi(\xi_{n}^{\ast})  \to\infty$ as 
$n\to\infty$ it then follows that $\Phi(\xi_{n}^{\ast})<\Phi(\xi_{n-1}^{\ast})$ and this
gives Proposition~\ref{max:increase}. 
\end{proof}

For each $n$, let now $\tau_{n}^{min}$ be the value of $\tau$ 
at which $\Phi$ reaches the minimum in the interval $(\tau_{n}^{\ast},\tau_{n-1}^{\ast})$ 
as defined in (\ref{def:min}). We can now prove Proposition~\ref{min:decrease}:

\begin{proof}[Proof of Proposition~\ref{min:decrease}]
As before, we only give the formal steps of the proof and refer to \cite{CV} for details.
Let $\xi_{n}^{min}=\xi(\tau_{n}^{min})$ be defined by means of (\ref{S2E8a}). 
Then, by Lemma~\ref{ApproximationLemma}, the fact that $\delta_{n}\ll 1$ and 
Lemma~\ref{MatchingLemma} we can write, to leading order for $n$ large enough, 
\[
\xi_{n}^{min}=\xi_{n}^{\ast} Z_{0}(M_{n})
\] 
and that, 
\[
\Phi(\tau_{n}^{min}) = |\xi_{n}^{min}|^{\frac{2}{3}} |\xi_{n}^{\ast}|^{5}M_{n}\mathcal{H}_{n}(\zeta_{0}(M_{n}))
\]
It is clear that for $n$ large enough $\Phi(\tau_{n}^{min})$ approaches $0$ by Lemma~\ref{ApproximationLemma} and employing the scaling  
\[
\mathcal{H}_{n}=\delta_{n}h_{n}\,,\quad\zeta_{n}-\zeta_{0}(M_{n})
=\delta_{n}\,s_{n}
\]
(that is analogous to the one used in the proof of Lemma~\ref{MatchingLemma}),
give that the following is a valid approximation 
\[
\Phi(\tau_{n}^{min}) = \left|\frac{\xi_{n}^{min}}{\xi_{n}^{\ast}}\right|^{\frac{2}{3}}
 \frac{1}{(\Phi(\tau_{n}^{\ast}))^{2}}  h_{n}(0)
\]
for $n$ large enough. Here we also use that de definition of $M_{n}$ (see \ref{S5E3}) 
and that $\delta_{n}=(\Phi(\tau_{n}^{\ast})^{-3})$ (see (\ref{parameters})). 
Thus for $n$ large enough one also has, by (\ref{J3E8}), that
\[
\Phi(\tau_{n-1}^{min}) \simeq \left|\frac{\xi_{n-1}^{min}}{\xi_{n-1}^{\ast}}\right|^{\frac{2}{3}}
 \frac{C}{(\Phi(\tau_{n}^{\ast}))^{20}}  h_{n-1}(0)
\]
for some order one constant $C>0$.

We finally observe that the values $h_{n}(0)$ are of order one if $n$ is large enough, 
by Lemma~\ref{MatchingLemma}. Also, we can approximately write the quotients 
$|\xi_{n}^{min}|/|\xi_{n}^{\ast}|=Z_{0}(M_{n})$. But each $Z_{0}(M_{n})$ is an order one constant, 
since the sequence $M_{n}$ is uniformly bounded. Then, since $\Phi(\tau_{n}^{\ast})\to\infty$, 
we have that $\Phi(\tau_{n}^{min})\to 0$. Moreover, there exists a constant $C>0$ 
(different from the one above) such that, for $n$ large enough, 
\[
\Phi(\tau_{n-1}^{min}) \simeq C (\Phi(\tau_{n}^{min}))^{10}\,,
\]
and thus (\ref{min:decrease}) follows.
\end{proof}

\section{Convergence to the equilibrium point $p_{-}$}\label{sec:6}
In this section we finish the proof of Theorem~\ref{hetcon}. First we prove the following
\begin{proposition}\label{ControlOscillations} 
Suppose that $(\Phi(\tau),W(\tau),\Psi(\tau),\theta(\tau))$ 
is a solution of (\ref{compact1})-(\ref{compact4}) as found in Proposition~\ref{existence:1} 
and satisfying (\ref{F5E1}). Then, there exist positive constants $C_{1}$ and $C_{2}$ 
depending only on $a$ such that
\begin{align}
\lim\sup_{\tau\to-\infty}\left(  \Phi(\tau) +\left|  \frac{d\Phi(\tau)}{d\tau}
\right|  +\left|  \frac{d^{2}\Phi(\tau)}{d\tau^{2}}\right|  \right)  \leq
C_{1} \,,\label{F5E2}\\
\nonumber\\
\mbox{and} \quad\lim\inf_{\tau\to-\infty}\Phi(\tau) \geq C_{2}>0 \,.
\label{F5E3}%
\end{align}
\end{proposition}

\begin{proof}[Proof]
We recall that the solutions found in Proposition~\ref{existence:1} 
are defined for all $\tau\in(-\infty,\infty)$ and satisfy   
$\lim_{\tau\to\infty}(\Phi(\tau),W(\tau),\Psi(\tau),\theta(\tau))=(1,0,0,\pi/2)$. 
Moreover, $\lim_{\tau\to-\infty}\theta(\tau)=-\frac{\pi}{2}$, $\Phi(\tau)>0$ for any
$\tau\in(-\infty,\infty)$ and $\lim\inf_{\tau\to-\infty}\Phi(\tau)<\infty$. 
We now claim that (\ref{F5E2}) holds for some $C_{1}>0$. Indeed, otherwise, 
due to Proposition~\ref{char:max} and Proposition~\ref{max:increase} there would exist 
a sequence $\{\tau_{n}^{\ast}\}$ such that $\lim_{n\to\infty}\tau_{n}^{\ast}=-\infty$ and 
$\lim_{n\to\infty}\Phi(\tau_{n}^{\ast})=\infty$ but such that there exits $n_{0}$ with 
$\Phi(\tau_{n-1}^{\ast})>\Phi(\tau_{n}^{\ast})$ for all $n>n_{0}$. Then, since 
$\lim\sup_{n\to\infty}\Phi(\tau_{n}^{\ast})=\infty$ (cf. (\ref{G5E3})) it follows that 
$\Phi(\tau_{n_{0}}^{\ast})=\infty$, this yields a contradiction and, therefore, (\ref{F5E2}) 
is satisfied.

Suppose now that (\ref{F5E3}) is not satisfied, then in particular this implies, probably taking a subsequence, that $\lim_{n\to \infty} \Phi(\tau_{n}^{min})=0$ but this contradicts (\ref{G5E5}) of Proposition~\ref{min:decrease}. 
\end{proof}

We can now finish the proof of the main result.

\begin{proof}[Proof of Theorem \ref{hetcon}]
Due to Proposition~\ref{existence:1} there exists a solution of (\ref{compact1})-(\ref{compact4}) 
defined for all $\tau\in(-\infty,\infty)$ such that  
$\lim_{\tau\to\infty}(\Phi(\tau),W(\tau),\Psi(\tau),\theta(\tau))=(1,0,0,\pi/2)$. 
Moreover, $\lim_{\tau\to-\infty}\theta(\tau)=-\frac{\pi}{2}$, $\Phi(\tau)>0$ for 
any $\tau\in(-\infty,\infty)$ and $\lim\inf_{\tau\to-\infty}\Phi(\tau)<\infty$. 
Then Proposition~\ref{ControlOscillations} gives that (\ref{F5E2}) and (\ref{F5E3}) hold.

We now define a sequence of functions:%
\[
\Phi_{n}(\tau)=\Phi(\tau-n)\,,  \quad 
W_{n}(\tau)=W(\tau-n)\,,\quad \Psi_{n}(\tau)=\Psi(\tau-n)\quad n=1,2,3,\dots
\]
Using (\ref{compact1})-(\ref{compact4}), 
(\ref{F5E2}), (\ref{F5E3}), standard compactness arguments and the fact that 
$\lim_{\tau\to-\infty}\theta(\tau)=-\frac{\pi}{2}$ we can show that there exists 
a subsequence $\{n_{j}\}$ satisfying $\lim_{j\to\infty}n_{j}=\infty$ and such that 
$\{(  \Phi_{n_{j}}(\tau),W_{n_{j}}(\tau),\Psi_{n_{j}}(\tau))\}$ 
converges uniformly in compact sets of $\tau$ to a bounded solution of (\ref{S3E1}), 
say $(\Phi_{\infty}(\tau),W_{\infty}(\tau),\Psi_{\infty}(\tau))$. 
Moreover, we have $\Phi_{\infty}(\tau)\geq C_{1}>0$, 
$\tau\in(-\infty ,\infty)$. Due to Proposition~\ref{Tapas1} it follows that 
$(\Phi_{\infty}(\tau),W_{\infty}(\tau),\Psi_{\infty}(\tau))$ is close to 
$P_{s}$ if $\tau<0$ and $|\tau|$ is large enough. 
Using the Stable Manifold Theorem it then follows that 
$\left(\Phi_{\infty}(\tau),W_{\infty}(\tau),\Psi_{\infty}(\tau)\right)$ is 
contained in the unstable manifold of $P_{s}$. However, due to 
Proposition~\ref{Tapas1}{\it (v)} it follows that the only bounded
trajectory contained in the unstable manifold of $P_{s}$ is 
the is the critical point itself,
thus $(\Phi_{\infty}(\tau),W_{\infty}(\tau),\Psi_{\infty}(\tau))\equiv P_{s}$. 
This implies that the sequence $P_{j}=(\Phi_{n_{j}}(0),W_{n_{j}}(0),\Psi_{n_{j}}(0)),\theta_{n_{j}}(0))$ 
converges to the equilibrium $p_{-}$ as $j\to\infty$. Therefore 
the points $P_{j}$ are contained in the centre-unstable manifold of $p_{-}$, 
whence $\lim_{\tau\to-\infty}(\Phi(\tau),W(\tau),\Psi(\tau),\theta(\tau))=p_{-}$ 
and the result follows.
\end{proof}

{\bf Acknowledgements} This work was supported by the Hausdorff Center of the University of Bonn. C.M. Cuesta was also partially supported by the DGES Grant MTM2011-24-109.

\section*{Appendix~}
\appendix
\section{Analysis of the solutions of (\ref{ecuacionminima2})}\label{sec:summary:II} 
We recall here the results concerning the following equation
\begin{equation}\label{ecuacionminima}%
\frac{d^{3}\Phi}{d\tau^{3}}=\frac{1}{\Phi^{3}}
\end{equation}
(cf. (\ref{ecuacionminima2})) that have been shown in \cite{CV}. 
For simplicity, we henceforth use the same notation for the dependent and independent 
variables as for (\ref{ODEtapas}). The following holds.
\begin{theorem}\label{matchsolution} 
There exists a unique solution of (\ref{ecuacionminima}) with the matching condition:
\[
\Phi(\tau) \sim-K\tau + o(1) \quad\mbox{as}\quad\tau \to -\infty\,.
\]
Moreover, the asymptotics of $\Phi(\tau)$ for large $\tau$ is given by:
\begin{equation}\label{para:match}
\Phi(\tau) \sim\Gamma\tau^{2}\quad\mbox{as}\quad
\tau\to\infty\quad\text{for some}\quad\Gamma>0\,.
\end{equation}
Finally, the exists a unique solution of (\ref{para:match}) with matching
condition
\[
\Phi(\tau) \sim-\tilde{K}\tau+o(1) \quad\mbox{as}\quad\tau\to
\infty\,\quad\text{for some}\quad\tilde{K}>0\,.
\]
It also satisfies that there exists a finite $\tau_{\ast}$ such that
\begin{equation}\label{back:zero}
\Phi(\tau)\to0 \quad\mbox{as}\quad\tau\to(\tau_{\ast})^{+}\,.
\end{equation}
All other solutions satisfy (\ref{para:match}) for increasing $\tau$, and,
for decreasing $\tau$, either (\ref{back:zero}) or
\[
\Phi(\tau) \sim\tilde{\Gamma}\tau^{2} \quad\mbox{as}\quad\tau\to
-\infty\quad\text{for some}\quad\tilde{\Gamma}>0
\] 
holds.
\end{theorem}

The proof of Theorem~\ref{matchsolution} is done by a series of lemmas. The
crucial step is to apply the transformation
\begin{equation}\label{SEphi0-trans}
\frac{d\Phi}{d\tau}=\Phi^{-\frac{1}{3}}\,u\,,
\ \frac{d^{2}\Phi}{d\tau} =\Phi^{-\frac{5}{3}}\,v \,,\ d\tau
=\Phi^{\frac{4}{3}}dz\,,
\end{equation}
that reduces (\ref{ecuacionminima}) to the system
\begin{equation}\label{phi0system}
\frac{d\Phi}{dz}=u\,\Phi\, , \quad\frac{du}%
{dz}=v+\frac{1}{3}u^{2}\,, \quad\frac{dv}{dz}=1+\frac{5}{3}u\,v\,.
\end{equation}
The lemmas then give the behaviour of the corresponding trajectories and are given
below for reference.

The last two equations in (\ref{phi0system}) can be studied independently by
means of a phase-plane analysis. The isoclines of this system are $\Gamma_{1}=
\{(u,v):\ v+\frac{1}{3}u^{2}=0\}$, that has $du/dz=0$, and $\Gamma_{2}=
\{(u,v):\ 1+\frac{5}{3}v\,u=0\}$, that has $dv/dz=0$. The only critical point
is $p_{e}=(u_{e},v_{e})=((9/5)^{\frac{1}{3}},-(1/3)(9/5)^{\frac{2}{3}})$, and
linearisation gives two complex eigenvalues with positive real part, namely
$\lambda_{\pm}=(1/2)(1/15)^{\frac{1}{3}}(7\pm\sqrt{11}i)$.

We distinguish five regions, $R_{1}$ to $R_{5}$, in the phase plane that are
separated by the isoclines. These are depicted in Figure~\ref{phase-plane} 
where the direction field is also shown. 
\begin{figure}[th]
\begin{center}
\input{phase-plane-small.pstex_t}
\end{center}
\caption{Phase portrait associated to (\ref{phi0system}) showing the direction
field. The thick solid lines represent the isoclines and the dashed ones the
separatrices.}%
\label{phase-plane}%
\end{figure}
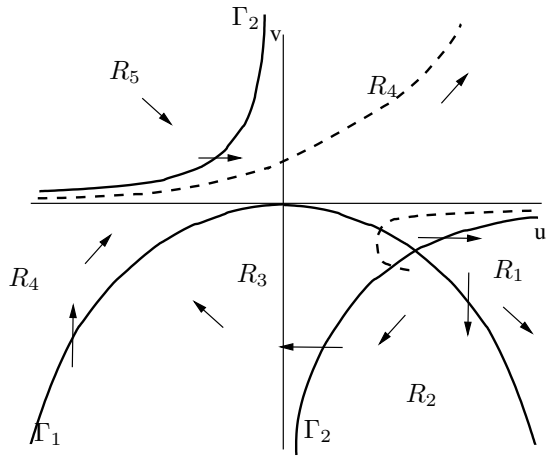

Standard arguments imply that any orbit on the phase plane eventually crosses
the isoclines into the region $R_{4}$ forwardly in $z$. We recall that 
$R_{4}=\{(u,v): \ -u^{2}/3<v<-3/(5u) \ \mbox{if} \ u<0\,, \ v>\max(-u^{2}/3,-3/5u) \ \mbox{if}\ u>0\}$ 
and the field in it satisfies $du/dv>0$. 
If, however, an orbit has $(u,v)\in R_{4}$ at some value of $z$, it is possible to discern from
which of the regions is coming from for smaller values of $z$ by identifying
the separatrices of the system. We have the following result.
\begin{lemma}[Separatrices]\label{separa}
\begin{enumerate}
\item There exists a unique orbit $v=\bar{v}(u)$ in the phase plane associated
to system (\ref{phi0system}) that is contained in $R_{4}$ for all
$u\in\mathbb{R}$. Moreover, $\bar{v}(u)$ has the following asymptotic
behaviour
\[
\bar{v}(u) =\frac{u^{2}}{2} + O\left(  u^{\frac{4}{5}}\right)
\quad\mbox{as}\quad u\to\infty\,,\quad\bar{v}(u) = -\frac{1}{2u}(1+o(1))
\quad\mbox{as}\quad u\to-\infty
\]

\item There exists a unique orbit $v=\hat{v}(u)$ in the phase plane associated
to system (\ref{phi0system}) that has the following asymptotic behaviour
\[
\hat{v}(u)= -\frac{1}{2u}\left(  1+o\left(  1\right)  \right)
\quad\mbox{as}\quad u\to+\infty
\]
and
\begin{equation}\label{W3E6e}
|(u,\hat{v})-(u_{e},v_{e})|\leq C e^{\mbox{Re}(\lambda_{+})z}
\quad\mbox{as}\quad z\to-\infty\,.
\end{equation}

\end{enumerate}
\end{lemma}

\begin{lemma}\label{tozero-bwds} All orbits associated to solutions of (\ref{phi0system})
enter $R_{4}$. Those that are below the separatrix $v=\bar{v}(u)$ come from the
critical point $p_{e}$ as $z\to-\infty$ and satisfy (\ref{W3E6e}). All other
orbits, except for $\bar{v}$, come from the region $R_{5}$.
\end{lemma}

\section{Analysis of the polynomial solutions of (\ref{ecuacionminima1})}\label{sec:polynomials} 
We identify the functions giving the leading order behaviour of the
solutions of (\ref{S2E3}) if $|\xi|^{2/3}H(\xi)$ becomes large on a bounded
interval around some negative value $\xi=\xi^{\ast}$ that gives a large local maximum of $\Phi$ 
(after the change of variables (\ref{S2E8a})). 

Due to (\ref{S2E8}), these approximating functions, if $\xi\to-\infty$ and $|\xi|^{2/3}H$ is large, are 
\begin{equation}\label{Q1E3}
\bar{\Phi}(\xi) =|\xi|^{\frac{2}{3}}\bar{H}(\xi) \,,
\end{equation}
where $\bar{H}(\xi)$ solves (\ref{ecuacionminima1}). 

Let $\xi=\xi^{\ast}<0$ be such that $\bar{\Phi}(\xi^{\ast})$ is a local maximum. We  shall derive in the 
following some properties of these functions, such us conditions to ensure that they have a double zero 
at a point $\xi^{min}<\xi^{\ast}$. But first we need to normalise them in an adequate way. 

We can readily integrate (\ref{ecuacionminima1}) with given initial conditions at $\xi=\xi^{\ast}$. 
This gives a family of fifth order polynomials that depend on $\xi^{\ast}$, $\bar{H}(\xi^{\ast})$ 
and the first and second derivatives of $\bar{H}$ evaluated at $\xi=\xi^{\ast}$. Using that 
$d\bar{\Phi}(\xi^{\ast})/d\xi=0$ we can write
\[
|\xi^{\ast}|\frac{d\bar{H}(\xi^{\ast})}{d\xi}= \frac{2}{3}\bar{H}(\xi^{\ast})
\]
thus we can eliminate one parameter and obtain that 
\begin{eqnarray}
\bar{H}(\xi) &=&-\frac{(\xi-\xi^{\ast})^{3}}{60}
\left((\xi-\xi^{\ast})^{2}+5\xi^{\ast}(\xi-\xi^{\ast})+10(\xi^{\ast})^{2}\right)
\nonumber
\\
&+&\bar{H}(\xi^{\ast}) \left(1+ \frac{2(\xi-\xi^{\ast})}{3\xi^{\ast}}\right)
+ \frac{d^{2}\bar{H}(\xi^{\ast})}{d\xi^{2}}\frac{(\xi-\xi^{\ast})^{2}}{2} 
\,.\label{int:H:bar}
\end{eqnarray}
We want to characterise the functions $\bar{\Phi}$ in terms of the parameters $\xi^{\ast}$, $\bar{\Phi}(\xi^{\ast})$ and
 $d^{2}\bar{\Phi}(\xi^{\ast})/d\xi^{2}$, however. In order to do that we first compute 
\[
\bar{\Phi}(\xi^{\ast})= |\xi^{\ast}|^{\frac{2}{3}}\bar{H}(\xi^{\ast})\,, 
\quad \frac{d^{2}\bar{\Phi}(\xi^{\ast})}{d\xi^{2}}= -\frac{10}{9}|\xi^{\ast}|^{-\frac{4}{3}}\bar{H}(\xi^{\ast})
+|\xi^{\ast}|^{\frac{2}{3}} \frac{d^{2}\bar{H}(\xi^{\ast})}{d\xi^{2}}
\]
and then define a family of polynomials $P$ that normalises $\bar{H}$ as follows:
\[
\bar{H}(\xi) =|\xi^{\ast}|^{5} P(Z;M,\beta) \,,
\quad Z=\frac{\xi}{\xi^{\ast}}\,, 
\]
thus $P$ solves
\begin{equation}\label{P:Z:eq}
\frac{d^{3} P }{dZ^{3}} = Z^{2}\,,
\end{equation}
and the parameters $M$ and $\beta$ are defined by 
\begin{equation}\label{S4E3}
M=P(1;M,\beta) \quad 
\mbox{and} \quad
\beta= \frac{1}{|\xi^{\ast}|^{\frac{11}{3}}}\frac{d^{2}\bar{\Phi}(\xi^{\ast})}{d\xi^{2}}\,.
\end{equation}
Then, (\ref{Q1E3}) implies that
\begin{equation} \label{H6}
\bar{\Phi}(\xi) =| \xi^{\ast}|^{\frac{17}{3}} (Z)^{\frac{2}{3}} P(Z;M,\beta) \,.
\end{equation}
Changing to these variables in (\ref{int:H:bar}) we obtain  
\begin{equation}\label{S4E4}
P(Z;M,\beta) =\frac{(Z-1)^{3}}{60}(Z^{2}+3Z+6) + \frac{M}{9}(5 Z^{2}-16Z+20) +\frac{\beta}{2}(Z-1)^{2}\,.
\end{equation}

We point out that $Z$ has the opposite sign as $\xi$ and thus we shall be
interested in the region $Z\geq1$, or that one ahead of the local maximum of
$(Z)^{2/3}P$. We next show the existence, for each $M>0$, of a unique $\beta_{\ast}(M)$ 
such that $P(Z;M,\beta)$ has a unique and double root in $\{Z > 1\}$.

\begin{lemma}\label{Pol2} 
Let $P(Z;M,\beta)$ be as in (\ref{S4E4}) with $M>0$ as in (\ref{S4E3}). 
Then, there exists a $\beta_{\ast}(M)<0$ such that
\begin{enumerate}
\item If $\beta>\beta_{\ast}(M)$, $P(Z;M,\beta)$ is strictly positive in
$Z\geq 1$.

\item If $\beta<\beta_{\ast}(M)$ there exists a unique $Z_{r}=Z_{r}(M,\beta)>1$ 
such that $P(Z;M,\beta)>0$ in $1\leq Z<Z_{r}$, $P(Z_{r};M,\beta)=0$ and 
$dP(Z_{r};M,\beta)/dZ<0$.

\item If $\beta=\beta_{\ast}(M)$ there exists $Z_{\ast}(M)>1$ that is a double
zero of $P(Z;M,\beta)$, and $P(Z;M,\beta)>0$ for $1\leq Z\neq Z_{\ast}(M)$.
\end{enumerate}
\end{lemma}

\begin{proof}[Proof]
The polynomial $P(Z;M,\beta)$ is monotonically increasing in $\beta$. Since
$(\frac{5}{9}Z^{2}-\frac{16}{9}Z+\frac{20}{9})>0$, it follows that
$P(Z;M,\beta)>0$ if $Z\geq1$ and $\beta\geq0$. On the other hand given $M>0$
and $Z>1$, $P(Z;M,\beta)<0$ for negative values of $\beta$ if $|\beta|$ is
large enough.

For each $M$ we define
\begin{equation}\label{S4E9}
\beta_{\ast}(M) =\sup\left\{  \beta<0:\ P(Z;M,\beta)<0
\quad\mbox{for some}\quad Z\geq1\right\}  \,.
\end{equation}
Then, by continuity in $\beta$ there exists a $Z_{\ast}(M)$ such that
$P(Z_{\ast}(M) ;M,\beta_{\ast}(M)) =0$. Moreover, 
$dP(Z_{\ast}(M) ;M,\beta_{\ast}(M))/dZ=0$, 
since otherwise there would exists
$\beta>\beta_{\ast}(M)$ and $Z\geq1$ such that $P(Z;M,\beta) <0$.

We now claim that $d^{2}P(Z_{\ast}(M);M,\beta_{\ast}(M))/dZ^{2} \neq 0$. 
Indeed, otherwise the definition of $\beta_{\ast}(M)$ would 
imply that $Z=Z_{\ast}(M)$ is a zero of $P(Z;M,\beta_{\ast}(M))$ of order
four and, in particular, that 
$d^{3}P(Z_{\ast} (M);M,\beta_{\ast}(M))/dZ^{3}=0$. 
But $P$ solves (\ref{P:Z:eq}), thus
\begin{equation}\label{S5E1}%
\frac{d^{3}P(Z;M,\beta_{\ast}(M))}{dZ^{3}}>0\quad
\mbox{for any}\quad Z\neq0\,.
\end{equation}

Suppose that $\beta<\beta_{\ast}(M)$. The monotonicity of $P(Z;M,\beta)$ in
$\beta$, combined with the fact that $P(1;M,\beta) =M>0$ and 
$\lim_{Z\to\infty}P(Z;M,\beta)=\infty$ imply that there exists at least two zeros
of $P(Z;M,\beta)$ in $\{Z\geq1\}$. Actually, there are exactly two zeros of
$P(Z;M,\beta)$ in $\{Z\geq1\}$. For otherwise, there would be four zeros in
$\{Z\geq1\}$ counting multiplicities, and this would imply, by Rolle's
Theorem, the existence of three zeros of $dP(Z;M,\beta)/dZ$ 
in $\{Z\geq1\}$, and iterating the argument, also the existence of two zeros
of $d^{2}P(Z;M,\beta)/dZ^{2}$ in $\{Z\geq1\}$ and at least
one zero of $d^{3}P(Z;M,\beta)/dZ^{3}$ in $\{Z\geq1\}$. But
this contradicts (\ref{S5E1}).

Therefore, $P(Z;M,\beta)$ has exactly two zeros in $\{Z\geq1\}$. The smallest of
which, $Z_{r}$, satisfies $dP(Z_{r};M,\beta)/dZ<0$
(by continuity since $M>0$).
\end{proof}

We now compute the asymptotic behaviour of $\beta_{\ast}(M)$ in the limits
$M\to 0$ and $M\to\infty$.

\begin{lemma}\label{Pol4} For every $M>0$ let $\beta_{\ast}(M)$ and $Z_{\ast}(M)$ 
be as in Lemma~\ref{Pol2}. Then, they satisfy that $1<Z_{\ast}(M)<4$ and 
$\beta_{\ast}(M) <0$, and have the following asymptotic behaviour:
\begin{equation}\label{S5E4}%
\begin{split}
Z_{\ast} (M) \sim1 + (12 M)^{\frac{1}{3}}\,, \quad\beta_{\ast}(M) \sim-
\left(\frac{3M}{2}\right)^{\frac{1}{3}} \,,\\
\quad\frac{d^{2}P(Z_{\ast};M,\beta_{\ast}(M))}{dZ^{2}}
\sim\left(\frac{3 M}{2}\right)^{\frac{1}{3}} \quad \mbox{as}\quad M \to 0 
\end{split}
\end{equation}
and
\begin{equation}\label{S5E5}%
\begin{split}
Z_{\ast} (M) \to4 \,,\quad\beta_{\ast}(M) \sim-\frac{8M}{9} \,, \\ 
\quad\frac{d^{2}P(Z_{\ast};M,\beta_{\ast}(M))}{dZ^{2}}\sim\frac{2M}{9} 
\quad \mbox{as}\quad M\to\infty\,. 
\end{split}
\end{equation}
\end{lemma}

\begin{proof}[Proof]
In order to determine $Z_{\ast}(M)$ and $\beta_{\ast}(M)$ we have to solve the
system that results from imposing that $P(Z;M,\beta)$ has a double zero:%
\begin{align}
\frac{(Z-1)^{3}}{60} (Z^{2}+3Z+6) + \frac{M}{9} \left(  5Z^{2} - 16Z + 20
\right)  +\frac{\beta}{2}(Z-1)^{2}  &  =0\,,\label{double:zero1}\\
\frac{(Z-1)^{2}}{12} (Z^{2} + 2Z + 3) +\frac{2M}{9}(5Z-8) + \beta(Z-1)  &=0\,. 
\label{double:zero2}%
\end{align}
Subtracting the second equation of (\ref{double:zero2}) multiplied by
$(Z-1)/2$ to (\ref{double:zero1}) we obtain%
\begin{equation}\label{double:zero3}
(Z-1)^{3}(3Z^{2}+4Z+3) =40(4-Z)M\,,
\end{equation}
the solution of which gives $Z_{\ast}(M)$, the position of the double zero.
Now $\beta_{\ast}(M)$ can be computed from either (\ref{double:zero1}) or
(\ref{double:zero2}) by substituting $Z=Z_{\ast}(M)$. It is clear that
$Z_{\ast}(M) \in(1,4)$ for any $M>0$, since the left-hand side of
(\ref{double:zero3}) is positive for $Z>1$.

It then follows from (\ref{double:zero3}) that, if $M\to0$, $(Z_{\ast}(M)-1)$
is of order $M^{\frac{1}{3}}$ and that, using (\ref{double:zero1}),
$\beta_{\ast} (M)$ behaves like $M^{\frac{1}{3}}$. On the other hand, if
$M\to\infty$, we obtain that $Z_{\ast}(M)\to4^{-}$ and that $\beta_{\ast}(M)$
is of order $M$. The precise asymptotic behaviours stated in (\ref{S5E4}) and
in (\ref{S5E5}) follow easily from (\ref{double:zero3}) and
(\ref{double:zero1}) by using the leading order behaviour of $Z_{\ast}(M)$ in
both limits $M\to 0$ and $M\to\infty$.
\end{proof}

Finally, we identify the largest root of $P$ in the region $Z<1$:
\begin{lemma}\label{Pol5} 
For all $M>0$ and $\beta\in\mathbb{R}$ the following value is
well defined
\begin{equation}\label{max:Zroot}
Z_{0}(M,\beta):=\max\{Z<1:\ P(Z;M,\beta)=0\}\,. 
\end{equation}
If $\beta=\beta_{\ast}(M)$ then, setting $Z_{0}(M):=Z_{0}(M,\beta_{\ast}(M))$
\begin{equation} \label{Z0:Mzero}%
Z_{0}(M)\sim1-\frac{(12M)^{\frac{1}{3}}}{2}\,,
\quad\frac{dP(Z_{0};M)}{dZ}\sim\left(  \frac{3}{2}\right)^{\frac{5}{3}}
M^{\frac{2}{3}}\quad\mbox{as}\quad M\to 0\,, 
\end{equation}
and
\begin{equation}\label{Z0:Minfty}
Z_{0}(M)\sim-\left(  \frac{20M}{3}\right)^{\frac{1}{3}}\,, 
\quad\frac{dP(Z_{0};M)}{dZ}\sim
\quad\frac{1}{3}\left(\frac{20}{3}\right)^{\frac{1}{3}}
M^{\frac{4}{3}}\quad\mbox{as}\quad M\to \infty\,.
\end{equation}
Moreover, in this case, $Z_{0}(M)$ is the only real root in $Z<1$.
\end{lemma}

The proof follows by continuity, Lemma~\ref{Pol2} and the Implicit Function Theorem.

We need to derive some information concerning the derivative of the polynomial
$P(Z;M,\beta)  $ at $Z=Z_{0}(M,\beta)$. 
\begin{lemma}\label{new:pol}
The value $Z_{0}(M,\beta)$ is the unique root of $P(Z;M,\beta)$ in $\{ Z\leq 1\}$. 
Moreover, if $Z_{0}(M,\beta)\geq - 1$, there exists a positive $c_{0}$
independent of $M$ and $\beta$ such that
\begin{equation}\label{J1E6}
\frac{dP(Z_{0}(M,\beta);M,\beta)  }{dZ}\geq c_{0}\max\{M^{\frac{2}{3}}, M \}\,.
\end{equation}
\end{lemma}

\begin{proof}[Proof]
Using (\ref{P:Z:eq}), and differentiating $P(Z;M,\beta)$ then
\[
\frac{d^{2}P(Z;M,\beta)}{dZ^{2}}=\frac{1}{3}(Z^{3}-1) 
+\left(\frac{10M}{9}+\beta\right)
\]
and we obtain that $P(Z;M,\beta)$ is strictly concave
for $Z<1$ if $10M/9+\beta\leq 0$. On the other hand, 
if $10M/9+\beta>0$ we obtain that $P(Z;M,\beta)$ is 
strictly concave for $Z<\tilde{Z}:=-(1-3(10M/9+\beta))^{\frac{1}{3}}$
and strictly convex if $Z>\tilde{Z}$. 
Then, since $dP(1;M,\beta)/dZ=-2M/3$ it follows that there 
exists a $\hat{Z}(M,\beta)<1$ such that $dP(Z;M,\beta)/dZ>0$ if 
$Z<\hat{Z}(M,\beta)$ and $dP(Z;M,\beta)/dZ<0$ if $\hat{Z}(M,\beta)<Z\leq 1$. 
Using now that $P(1;M,\beta)=M>0$ it then follows that $P(Z;M,\beta)$ 
attains a positive maximum in the interval $Z\in(-\infty,1)$ at 
$Z=\hat{Z}(M,\beta)$. Therefore, since $\lim_{Z\to -\infty}P(Z;M,\beta)=-\infty$ 
and the concavity, the function $P(\cdot;M,\beta)$ has a unique zero 
$Z_{0}(M,\beta)\in(-\infty,1)$.

It is easy to prove that for $M$ bounded $dP(Z;M,\beta)/dZ\geq K_{0}>0$ 
uniformly for $\beta\leq 0$, uniformly on $Z\in(-1,1)$ (differentiating (\ref{S4E4})). 
In fact, one can actually show that for $M$ small enough there exists a positive 
constant $c_{0}$ such that $dP(Z;M,\beta)/dZ > c_{0}M^{\frac{2}{3}}$ if $Z\in(-1,1)$. 

In order to obtain (\ref{J1E6}) for large $M$ we first consider the case 
$|\beta| \leq\varepsilon_{0}M$ with $\varepsilon_{0}$ sufficiently small, 
and then (\ref{J1E6}) follows easily $Z\in(-1,1)$ 
Suppose then that $\beta\leq-\varepsilon_{0}M$. Then, if $M$ is large we can approximate 
$P(Z;M,\beta)$ and its derivatives in the interval $Z\in(-1,1)$ by 
\[
\frac{M}{9}(5Z^{2}-16Z+20)  +\frac{\beta}{2}(Z-1)^{2}=M W_{M}(Z)
\]
and its derivatives respectively. 
Observe that then $W_{M}(Z)$ is a quadratic polynomial with bounded
coefficients satisfying $W_{M}(1)=1$. Suppose that $Z_{0}=Z_{0}(M,\beta)\in(-1,1)$.
 Notice that we cannot have any other zero of $W_{M}(Z)$ in the region $Z<1$ 
because then, by continuity, $P(Z;M,\beta)$ would have more than 
one zero in the domain $\{Z<1\}$ and this would contradict the 
statement above. Then, using also that $W_{M}(1)=1$, we can write
\[
W_{M}(Z)  =\frac{(Z - Z_{0})(Z-Z_{1})}{(1-Z_{0})(1-Z_{1})}
\]
with $\frac{1}{(1-Z_{0})(1-Z_{1})}$ bounded and $Z_{1}>1$. 
Therefore $\min\{(1-Z_{0}),(Z_{1}-1)\}  \geq K_{1}>0$, 
uniformly for large $M$. It then follows that $(Z_{1}-Z_{0})\geq 2K_{1}$. 
Then, since $Z_{0}\geq-1$,
\[
\frac{dW_{M}(Z_{0})}{dZ}=\frac{(Z_{0}-Z_{1})}{(1-Z_{0})(1-Z_{1})}
=\frac{(Z_{1}-Z_{0})}{(1-Z_{0})(Z_{1}-1)}
\geq\frac{1}{2}\frac{(Z_{1}-Z_{0})}{(Z_{1}-1)}\geq c_{1}>0\,.
\]
And the result follows.
\end{proof}

\begin{figure}[th]
\begin{center}
\input{polsP.pstex_t}
\end{center}
\caption{Schematic depiction of the polynomials $P(Z;M,\beta)$ for fixed $M$ and different values of $\beta$. The solid line represents a polynomial with $\beta=\beta_{\ast}$, 
the dashed line one with $\beta >\beta_{\ast}$ and the dashed-dotted line one with $\beta< \beta_{\ast}$. The figure also reflects the fact that $dP(1;M,\beta)/dZ<0$ 
for any $\beta<0$ and $M>0$.}
\label{polsP}%
\end{figure}
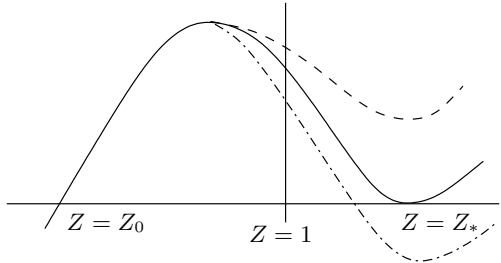

\section{Analysis of the bouncing region}\label{sec:control}
We seek to reformulate the results concerning the region of very small $H$ 
(or $\Phi$) that can be approximated by (\ref{ecuacionminima2}), that were
obtained in \cite{CV}. The result here is more general, namely the class 
of equations under consideration is
\begin{equation}\label{order1:eq}
\frac{d^{3}\mathcal{H}}{d\zeta^{3}}+R(\zeta)=\frac{\delta}{\mathcal{H}^{3}} \,,\quad \delta>0
\end{equation}
where, if $\zeta$ belongs to some given bounded interval, then
\begin{equation}
| R(\zeta)| +\left|\frac{dR}{d\zeta}(\zeta)\right|\leq C_{0}\,.\label{J3E1}%
\end{equation}

\begin{lemma}\label{MatchingLemma}
Suppose that the equation (\ref{order1:eq}) is satisfied in an interval 
$\zeta\in(\zeta_{0}-\varepsilon_{1},\zeta_{0}+\varepsilon_{1})$ for some 
$\varepsilon_{1}>0$ independent of $\delta$ and where $R$ satisfies (\ref{J3E1}). 
Suppose that:
\begin{equation}\label{J3E2}
\left| \mathcal{H}(\zeta)+K(  \zeta-\zeta_{0})  \right|  
\leq\varepsilon_{2}K\left| \zeta-\zeta_{0}\right| \,, \quad
\left| \frac{d\mathcal{H}}{d\zeta}(\zeta) + K\right|   
\leq\varepsilon_{2} K \,,\quad
\left| \frac{d^{2}\mathcal{H}}{d\zeta^{2}}(\zeta)\right| 
\leq\varepsilon_{2}
\end{equation}
for $\zeta\in(\zeta_{0}-\varepsilon_{1},\zeta_{0}-\varepsilon_{1}/2)$, $K\geq c_{0}>0$ 
and some $\varepsilon_{2}>0$. For any $\varepsilon_{3}>0$, there exists $\varepsilon_{0}>0$ 
independent of $\delta$ and $K$, but in general depending on $c_{0}$, and a real constant $A>0$
such that if $\varepsilon_{1}+\varepsilon_{2}\leq\varepsilon_{0}$ there
exists a $\delta_{0}=\delta_{0}(\varepsilon_{1},\varepsilon_{3},c_{0})>0$ such that for 
$\delta\leq\delta_{0}$ we have:%
\begin{align}
\left| \mathcal{H}(\zeta)-\frac{A K^{5}}{\delta}(\zeta-\zeta_{0})^{2}\right|  & 
\leq\frac{\varepsilon_{3}K^{5}}{\delta}(\zeta-\zeta_{0})^{2}\,,
\nonumber\\
\left| \frac{d\mathcal{H}}{d\zeta}(\zeta)  
-\frac{2AK^{5}}{\delta}(\zeta-\zeta_{0})  \right|  & 
\leq\frac{\varepsilon_{3}K^{5}}{\delta}(\zeta-\zeta_{0})  \,,
\label{J3E3}\\
\left| \frac{d^{2}\mathcal{H}}{d\zeta^{2}}(\zeta)-\frac{2AK^{5}}{\delta}\right|  & 
\leq\frac{\varepsilon_{3}K^{5}}{\delta}\,, \nonumber
\end{align}
for $\zeta\in(\zeta_{0}+\varepsilon_{1}/2,\zeta_{0}+\varepsilon_{1})$.
\end{lemma}

\begin{proof}[Proof]
This result can be adapted from the results proved in \cite{CV}. More precisely, 
the result follows from arguing as in the proof of the lemmas 4.4, 4.5, 4.6 and 4.7 
of this article. We recall the main ideas of the argument here. 

We introduce the variables $h(s)=K^{3}\mathcal{H}(\zeta)/\delta$ and
$s=K^{4}(\zeta-\zeta_{0})/\delta$ and obtain that $h$ satisfies the equation
\begin{equation}\label{h:eq}
\frac{d^{3}h}{ds^{3}} +O\left(  \frac{\delta^{2}}{K^{9}}\right)  =\frac{1}{h^{3}}
\end{equation}
with the matching condition
\[
h(s)=-s  \quad \mbox{as}\quad s \to-\infty\,.
\]
We can then reformulate (\ref{h:eq}) by the transformation (\ref{SEphi0-trans}) 
of Appendix~\ref{sec:summary:II}, thus, to leading order of approximation 
and due to the assumption (\ref{J3E2}), the solution of the resulting perturbation of (\ref{phi0system}) 
follows the separatrix $\bar{v}(u)$ (see Lemma~\ref{separa}). A key point in the argument is
that the system of ODEs (\ref{phi0system}) is integrated forward, the direction for which the separatrix
is stable. This allows to prove that $h(s)\sim A s^{2}$ as $s\to\infty$ for some $A>0$ (see Theorem~\ref{matchsolution}), 
thus the asymptotics of the resulting solution for $\zeta>\zeta_{0}$ can be described by means of (\ref{J3E2}).
\end{proof}

\def\cprime{$'$}

\end{document}

%% file: phase-plane-small.pstex_t
\begin{picture}(0,0)%
\includegraphics{phase-plane-small.pstex}%
\end{picture}%
\setlength{\unitlength}{3947sp}%
\begingroup\makeatletter\ifx\SetFigFont\undefined%
\gdef\SetFigFont#1#2#3#4#5{%
  \reset@font\fontsize{#1}{#2pt}%
  \fontfamily{#3}\fontseries{#4}\fontshape{#5}%
  \selectfont}%
\fi\endgroup%
\begin{picture}(3489,2844)(152,-2047)
\put(510,-1934){\makebox(0,0)[b]{\smash{{\SetFigFont{10}{12.0}{\familydefault}{\mddefault}{\updefault}{\color[rgb]{0,0,0}$\Gamma_1$}%
}}}}
\put(2855,-1690){\makebox(0,0)[b]{\smash{{\SetFigFont{10}{12.0}{\familydefault}{\mddefault}{\updefault}{\color[rgb]{0,0,0}$R_2$}%
}}}}
\put(3401,-908){\makebox(0,0)[b]{\smash{{\SetFigFont{10}{12.0}{\familydefault}{\mddefault}{\updefault}{\color[rgb]{0,0,0}$R_1$}%
}}}}
\put(1793,-937){\makebox(0,0)[b]{\smash{{\SetFigFont{10}{12.0}{\familydefault}{\mddefault}{\updefault}{\color[rgb]{0,0,0}$R_3$}%
}}}}
\put(991,338){\makebox(0,0)[b]{\smash{{\SetFigFont{10}{12.0}{\familydefault}{\mddefault}{\updefault}{\color[rgb]{0,0,0}$R_5$}%
}}}}
\put(2205,-1919){\makebox(0,0)[b]{\smash{{\SetFigFont{10}{12.0}{\familydefault}{\mddefault}{\updefault}{\color[rgb]{0,0,0}$\Gamma_2$}%
}}}}
\put(1753,677){\makebox(0,0)[b]{\smash{{\SetFigFont{10}{12.0}{\familydefault}{\mddefault}{\updefault}{\color[rgb]{0,0,0}$\Gamma_2$}%
}}}}
\put(2610,247){\makebox(0,0)[b]{\smash{{\SetFigFont{10}{12.0}{\familydefault}{\mddefault}{\updefault}{\color[rgb]{0,0,0}$R_4$}%
}}}}
\put(360,-961){\makebox(0,0)[b]{\smash{{\SetFigFont{10}{12.0}{\familydefault}{\mddefault}{\updefault}{\color[rgb]{0,0,0}$R_4$}%
}}}}
\end{picture}%

%% file: polsP.pstex_t
\begin{picture}(0,0)%
\includegraphics{polsP.pstex}%
\end{picture}%
\setlength{\unitlength}{3947sp}%
\begingroup\makeatletter\ifx\SetFigFont\undefined%
\gdef\SetFigFont#1#2#3#4#5{%
  \reset@font\fontsize{#1}{#2pt}%
  \fontfamily{#3}\fontseries{#4}\fontshape{#5}%
  \selectfont}%
\fi\endgroup%
\begin{picture}(3117,1646)(1368,-1695)
\put(3298,-1563){\makebox(0,0)[rb]{\smash{{\SetFigFont{9}{10.8}{\familydefault}{\mddefault}{\updefault}{\color[rgb]{0,0,0}$Z=1$}%
}}}}
\put(2251,-1486){\makebox(0,0)[rb]{\smash{{\SetFigFont{9}{10.8}{\familydefault}{\mddefault}{\updefault}{\color[rgb]{0,0,0}$Z=Z_{0}$}%
}}}}
\put(4351,-1486){\makebox(0,0)[rb]{\smash{{\SetFigFont{9}{10.8}{\familydefault}{\mddefault}{\updefault}{\color[rgb]{0,0,0}$Z=Z_{\ast}$}%
}}}}
\end{picture}%